\documentclass{amsart}
\usepackage{amsmath}
\usepackage{amssymb} 
\usepackage{amscd} 
\usepackage{graphicx}   
\usepackage{xcolor}
\usepackage{marvosym}
\usepackage{float}

\usepackage{hyperref}
\hypersetup{
	colorlinks=true,
	linktoc=all,
    	linkcolor={red!50!black},
   	citecolor={blue!50!black},
  	urlcolor={blue!80!black}
	} 

\allowdisplaybreaks 

\usepackage{rotating}

\usepackage{booktabs} 

\usepackage{array}
\newcolumntype{M}[1]{>{\centering\arraybackslash}m{#1}} 

\newcommand{\N}{\mathbb N}
\newcommand{\Q}{\mathbb Q}

\newcommand{\calL}{\mathcal L}

\newcommand{\nbhd}{\mathcal{N}}
\newcommand{\bdry}{\partial}

\newcommand{\nos}{\normalsize}
\newcommand{\bpc}[1]{\begin{picture} #1 \end{picture}}

\newcommand{\pcr}[2]{\mbox{$\begin{array}{c}
			\includegraphics[scale=#2]{#1.eps}
		\end{array}$}}

\newcommand{\dis}{\displaystyle}

\newtheorem{theorem}{Theorem}[section]
\newtheorem{lemma}[theorem]{Lemma}

\newtheorem{proposition}[theorem]{Proposition}
\newtheorem{corollary}[theorem]{Corollary}
\newtheorem{claim}[theorem]{Claim}
\newtheorem{addendum}[theorem]{Addendum}

\newtheorem{conjecture}[theorem]{Conjecture}
\newtheorem{remark}[theorem]{Remark}
\newtheorem{question}[theorem]{Question}

\newtheorem*{thm_slope_conjecture_M}{Theorem~\ref{thm:SC_and_SSC_M}(1)}

\newtheorem*{thm_mazur_crossing}{Theorem~\ref{crossing_Mazur}}

\theoremstyle{definition}
\newtheorem{definition}[theorem]{Definition}
\newtheorem{example}[theorem]{Example}

\numberwithin{equation}{section}
\numberwithin{figure}{section}
\numberwithin{table}{section}

\DeclareMathOperator{\wind}{wind}
\DeclareMathOperator{\wrap}{wrap}

\definecolor{dartmouthgreen}{rgb}{0.05, 0.5, 0.06}

\renewcommand{\(}{\textup{(}}
\renewcommand{\)}{\textup{)}}

\usepackage{fullpage}

\begin{document}
\baselineskip 14pt

\title{The Strong Slope Conjecture and crossing numbers for Mazur doubles of knots}

\author[K.L. Baker]{Kenneth L. Baker}
\address{Department of Mathematics, University of Miami, 
Coral Gables, FL 33146, USA}
\email{k.baker@math.miami.edu}

\author[K. Motegi]{Kimihiko Motegi}
\address{Department of Mathematics, Nihon University, 
3-25-40 Sakurajosui, Setagaya-ku, 
Tokyo 156--8550, Japan}
\email{motegi.kimihiko@nihon-u.ac.jp}

\author[T. Takata]{Toshie Takata}
\address{Graduate School of Mathematics, Kyushu University, 
744 Motooka, Nishi-ku, Fukuoka 819--0395, Japan}

\dedicatory{}

\begin{abstract}
The Slope Conjecture proposed by Garoufalidis asserts that the degree of the colored Jones polynomial determines 
a boundary slope, and its refinement, the Strong Slope Conjecture proposed by Kalfagianni and Tran asserts that the linear term in the degree determines the topology of an essential surface that satisfies the Slope Conjecture.  
Under certain hypotheses,
we show that Mazur doubles of knots satisfy the Strong Slope Conjecture if the original knot does.  
Consequently, any knot obtained by a finite sequence of cabling, untwisted $\omega$--generalized Whitehead doublings 
with $\omega > 0$, 
connected sums and Mazur doublings of $B$--adequate knots or torus knots satisfies the Strong Slope Conjecture.
On the other hand, it may be worth mentioning that under these hypotheses,
if there exists a knot with a Jones slope less than $-\frac{1}{4}$, then its Mazur double would either provide a counterexample to the Strong Slope Conjecture or have a Jones surface that is unrelated to any Jones surface of the knot.

Following work of Kalfagianni and Lee, 
we also use our results to show that the Mazur double of an adequate knot $K$ with trivial writhe has crossing number either $9c(K)+2$ or $9c(K)+3$. 
\end{abstract}

\maketitle

\renewcommand{\thefootnote}{}
\footnotetext{2020 \textit{Mathematics Subject Classification.}
Primary 57K10, 57K14, 57K30, Secondary 57K16
\footnotetext{ \textit{Key words and phrases.}
colored Jones polynomial, Jones slope, boundary slope, Mazur double, Slope Conjecture, Strong Slope Conjecture}
}

\tableofcontents

\section{Introduction} 
\label{intro}

Let $K$ be a knot in the $3$--sphere $S^3$.   
The Slope Conjecture of Garoufalidis \cite{Garoufalidis} and the Strong Slope Conjecture of Kalfagianni and Tran \cite{KT} propose relationships between a quantum knot invariant, 
the degrees of the colored Jones function of $K$, 
and a classical invariant, the boundary slope and the topology of essential surfaces in the exterior of $K$.

The \textit{colored Jones function} of $K$ is a sequence of Laurent polynomials $J_{K, n}(q) \in \mathbb{Z}[q^{\pm\frac{1}{2}}]$ for $n \in \mathbb{N}$, 
where $J_{\bigcirc, n}(q)=\frac {q^{n/2}-q^{-n/2}}{q^{1/2}-q^{-1/2}}$ for the unknot $\bigcirc$ and 
$\frac {J_{K, 2}(q)}{J_{\bigcirc, 2}(q)}$ is the ordinary Jones polynomial of $K$. 
Since the colored Jones function is $q$--holonomic \cite[Theorem~1]{GL}, 
the degrees of its terms are given by \textit{quadratic quasi-polynomials} for suitably large $n$ 
\cite[Theorem 1.1 \& Remark 1.1]{Gqqp}.   
For the maximum degree $d_+[J_{K,n}(q)]$, 
we set the quadratic quasi-polynomials to be 
\[ \delta_K(n) = a(n) n^2 + b(n) n+ c(n) \]
for rational valued periodic functions $a(n), b(n), c(n)$ with integral period.
Now define the sets of {\em Jones slopes} of $K$:
\[  js(K) = \{ 4a(n) \ |\ n \in \mathbb{N} \}.\]

Allowing surfaces to be disconnected, we say a properly embedded surface in a $3$--manifold is {\em essential} if each component is orientable, incompressible, and boundary-incompressible. A number $p/q \in \Q \cup \{\infty\}$ is a {\em boundary slope} of a knot $K$ if there exists an essential surface in the knot exterior $E(K)=S^3-\mathrm{int}N(K)$ with a boundary component representing $p[\mu]+q[\lambda] \in H_1(\bdry E(K))$ with respect to the standard meridian-longitude pair $(\mu, \lambda)$.  
Now define the set of  boundary slopes of $K$:
\[ bs(K) = \{ r \in \mathbb{Q}\cup \{\infty\}\ |\ r\ \mbox{ is a boundary slope of }\ K\}. \]
Since a Seifert surface of minimal genus is an essential surface, 
$0 \in bs(K)$ for any knot. 
Let us also remark that $bs(K)$ is always a finite set \cite[Corollary]{Hat1}. 

Garoufalidis conjectures that Jones slopes are boundary slopes.

\begin{conjecture}[\textbf{Slope Conjecture} \cite{Garoufalidis}]
\label{slope conjecture}
For any knot $K$ in $S^3$,  
every Jones slope is a boundary slope.  
That is $js(K) \subset bs(K)$. 
\end{conjecture}

Garoufalidis' Slope Conjecture concerns only the quadratic terms of $\delta_K(n)$.
Later Kalfagianni and Tran proposed the Strong Slope Conjecture which subsumes the Slope Conjecture and asserts that the topology of the surfaces whose boundary slopes are Jones slopes may be predicted by the linear terms of $\delta_K(n)$.   
Define
\[  jx(K) =  \{ 2b(n) \ |\ n \in \mathbb{N} \}. \]

Let $K$ be a knot in $S^3$ with $\delta_K(n) = a(n) n^2 + b(n) n+ c(n)$. 
For a given Jones slope $p/q \in js(K)$ ($q > 0$), 
we say that $p/q$ satisfies $SS(n)$ ($n \in \N$)
if there is an essential surface $F_n$ in the exterior of $K$ 
such that 
\begin{itemize}
\item $F_n$ has the boundary slope $4a(n) = p/q$, and
\item $\displaystyle 2b(n) = \frac{\chi(F_n)}{|\bdry F_n| q}$.
\end{itemize}	

We call such an essential surface $F_n$ a \textit{Jones surface} of $K$. 

\begin{conjecture}[\textbf{The Strong Slope Conjecture} \cite{KT,K-AMSHartford}]
\label{YSSC2}
For any knot in $S^3$, 
every Jones slope satisfies $SS(n)$ for some $n \in \mathbb{N}$. 
\end{conjecture}

It is convenient to say that 
$K$ satisfies the 
\textit{Strong Slope Conjecture with $SS(n_0)$}, 
if a Jones slope $p/q = 4a(n_0)$ satisfies $SS(n_0)$. 
(This is weaker than the Strong Slope Conjecture in the sense that 
we do not consider if another Jones slope $p'/q'$ other than $p/q$ satisfies $SS(n)$ for some other $n$.)

\begin{remark}
\label{constant}
Let $K$ be a knot with constant $a(n)$ \(i.e. the period of $a(n)$ is $1$\).  
Then it has a single Jones slope $p/q$ and 
if it satisfies $SS(n_0)$ for some $n_0$, then $K$ satisfies the Strong Slope Conjecture. 
If $b(n)$ is also constant, then we may take $n_0 = 1$ and the Strong Slope Conjecture 
implies the Strong Slope Conjecture with $SS(1)$. 
Presently, no knots are known for which the period of $a(n)$ is not $1$. 
\end{remark}

\medskip

\begin{example}
\label{earlyexamples}
Let $K$ be a knot which appears in the following list.
\begin{enumerate}
\item Torus knots \cite{Garoufalidis}, \cite[Theorem 3.9]{KT}. 
\item Adequate knots \cite{FKP}, \cite[Lemma 3.6, 3.8]{KT}, and hence alternating knots. 
\item Non-alternating knots with up to $9$ crossings except for  $8_{20}$, $9_{43}$, $9_{44}$ \cite{Garoufalidis}, \cite{KT,Ho}.   
($8_{20}$, $9_{43}$, $9_{44}$ satisfy the Strong Slope Conjecture, 
but for these knots the coefficient $b(n)$ has period $3$.)
\item Graph knots \cite{MT,BMT_graph}. 
\end{enumerate}
Then, writing 
$\delta_K(n) = a(n) n^2 + b(n) n + c(n)$,
we have that  
$a(n), b(n)$ are constant, 
and $c(n)$ has period at most two. 
Moreover, $K$ satisfies the Slope Conjecture, the unique Jones slope satisfies $SS(1)$, 
and hence $K$ satisfies the Strong Slope Conjecture \(Remark~\ref{constant}\).  
Note also that if $K$ is nontrivial, then $b(n) = b \le 0$. 
\end{example}

It should be mentioned that certain families of $3$-tangle pretzel knots 
and certain families of Montesinos knots are also known to satisfy the Slope Conjecture and the Strong Slope Conjecture; see \cite{LV} and \cite{GLV,LYL}.

\subsection{Main Results}
In this article we give further supporting evidence for the Slope Conjecture and the Strong Slope Conjecture by examining them for satellite knots with the Mazur pattern. 

Let $V$ be a standardly embedded solid torus in $S^3$ which contains a knot $k$ in its interior 
so that $k$ is not a core of $V$ and there is no $3$--ball in $V$ containing $k$. 
We call $k$ a \textit{pattern knot} and the pair $(V, k)$ a \textit{pattern}. 
Take an orientation preserving embedding $f \colon V \to S^3$ which sends 
a core of $V$ to a nontrivial knot $K \subset S^3$ and 
sends a longitude of $V$ to that of $K$.  
Such an embedding is often called a \textit{faithful embedding}. 
The image $f(k)$ is a \textit{satellite knot} with a companion knot $K$ and pattern $(V, k)$. 

The (unoriented) \textit{winding number} of $k$ in $V$ is the absolute value of the algebraic intersection number between $k$ and a meridian disk of $V$, 
which is denoted by $\wind_V(k) \ge 0$. 
The \textit{wrapping number} of $k$ in $V$ is the minimal geometric intersection number between $k$ and a meridian disk of $V$, which is denoted by $\wrap_V(k) \ge 1$. 
Note that $\wind_V(k)\leq \wrap_V(k)$ and they are equal $\pmod{2}$.
If $\wrap_V(k) = 1$, 
then the satellite operation is nothing but a connected sum with $k$. 
The Slope Conjecture and the Strong Slope Conjecture have been studied for connected sums 
in \cite{BMT_graph}. 
So we may assume $\wrap_V(k) \ge 2$.

Among satellite knots with $\wrap_V(k) \ge 2$, the Slope Conjecture and the Strong Slope Conjecture have been studied for
\begin{itemize}
	\item ``cables'' in which $0 < \wind_V(k) = \wrap_V(k)$ \cite{KT,BMT_graph}, and
	\item ``twisted generalized Whitehead doubles'' in which $0 = \wind_V(k) < \wrap_V(k)$ \cite{BMT_tgW}. 
\end{itemize}
We are interested in the situation between these two extremes where $0 < \wind_V(k) < \wrap_V(k)$. 

Perhaps the simplest and most well-known  pattern with $0<\wind_V(k) < \wrap_V(k)$ is the {\em Mazur pattern} $(V,k)$ depicted in Figure~\ref{fig:Mazur_double} (Left) which has $0 < 1 = \mathrm{wind}_V(k) < 3= \mathrm{wrap}_V(k)$. 
For a faithful embedding $f \colon V \to S^3$, 
the knot $f(k)$ is a {\em Mazur double} of $K$. 
For short we will denote a Mazur double of $K$ by $M(K)$.   
For some notable literature involving Mazur doubles, see \cite{Levine,FPR,PetWon}. 

\begin{figure}[!ht]
\includegraphics[width=0.6\linewidth]{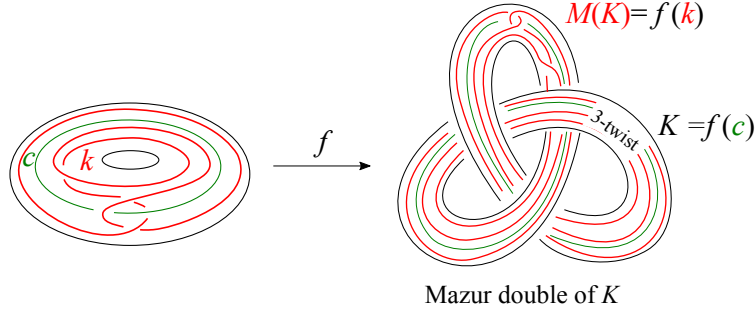}
\caption{A Mazur double of a knot $K$}
\label{fig:Mazur_double}
\end{figure}

In what follows we use the following abbreviation. 
For a quadratic quasi-polynomial $\delta_K(n) = a(n)n^2 + b(n) n + c(n)$ with period $\le 2$, 
we put $a_i = a(i),\ b_i = b(i)$ and $c_i= c(i)$, where $i = 1, 2$.

\begin{theorem}[The Slope Conjecture and Strong Slope Conjecture for Mazur Doubles]
\label{thm:SC_and_SSC_M}
Let $K$ be a knot. 
We assume that the period of $\delta_K(n)$ is less than or equal to $2$. 
Assume that $b_i \le 0$, and that $b_i = 0$ implies $a_i \ne 0$.
\begin{enumerate}
\item If $K$ satisfies the Slope Conjecture, 
then its Mazur double $M(K)$ also satisfies the Slope Conjecture. 
\item Further assume that $-\frac{1}{4} \le a_i$.  
If $K$ satisfies the Strong Slope Conjecture,  
then its Mazur double also satisfies the Strong Slope Conjecture. 
\end{enumerate}
\end{theorem}

Note that any torus knot and any $B$--adequate knot satisfies $a_i \ge 0$ \cite[Lemma 6]{FKP}. 
Furthermore, we see that the operation of cabling, untwisted $\omega$--generalized Whitehead doubling
with $\omega > 0$, connected sum,  and Mazur doubling 
preserve this property. 
So far we have no example with $a_i < 0$. 
It may be reasonable to ask: 

\begin{question}
\label{a_i<0}
Does there exist a knot $K$ with negative Jones slope?  That is, does there exist a knot for which $a_i < 0$? 
\end{question}

\begin{remark}
In the proof of 
Theorem~\ref{thm:SC_and_SSC_M}(2), 
 the Strong Slope Conjecture is affirmed for the Mazur double $M(K)$ through using a Jones surface for $K$ to construct a Jones surface for $M(K)$.
However, if there were a knot $K$ as in the hypotheses of Theorem~\ref{thm:SC_and_SSC_M}(2) but with $a_i <-\frac14$, then either the Mazur double $M(K)$ would provide a counterexample to the Strong Slope Conjecture, or it would have a Jones surface that is unrelated to any Jones surface of $K$.  See Addendum~\ref{add:mazurSSCfailure}.
\end{remark}
\medskip

\begin{remark}
\label{iteration}
Suppose $K$ satisfies the following conditions: 
\begin{enumerate}
\item
$\delta_K(n)$ has period less than or equal to $2$ and has $b_i \le 0$. 
\item
If $a_i = 0$, then $b_i \ne 0$ \(equivalently if $b_i =0$, then $a_i \ne 0$\). 
\item
$a_i \ge -\frac{1}{4}$.
\item 
$K$ satisfies the Strong Slope Conjecture.
\end{enumerate}

Then Proposition~\ref{maxdeg_Mazur} shows that 
$M(K)$ also satisfies conditions (1), (2), and (3), 
and Theorem~\ref{thm:SC_and_SSC_M}(2) shows that $M(K)$ also satisfies condition (4). 
Hence, we may conclude that arbitrarily many iterations of Mazur doubling preserves the property of satisfying the Strong Slope Conjecture.
\end{remark}

\medskip

We observe in Proposition~\ref{Normalized Sign Condition} that 
the Sign Condition introduced in \cite{BMT_tgW}
is preserved under taking Mazur doubles. 
See Subsection~\ref{computation_degrees}, 
for the definition of the Sign Condition. 
Given that 
$B$--adequate knots and torus knots satisfy $a_i \ge 0$, 
and the operations of cabling, untwisted $\omega$--generalized Whitehead doubling
with $\omega > 0$, connected sum, and Mazur doubling 
preserve this property of having $a_i \ge 0$ (as mentioned just after Theorem~\ref{thm:SC_and_SSC_M}(2)), 
we may extend \cite[Corollary~1.9]{BMT_tgW} to include Mazur doubles. 

\begin{corollary}
\label{family_SCC}
Any knot obtained by a finite sequence of cabling, untwisted $\omega$--generalized Whitehead doublings with $\omega > 0$, connected sums, and Mazur doubling
of $B$--adequate knots or torus knots satisfies the Slope Conjecture and the Strong Slope Conjecture.  \qed
\end{corollary}

\bigskip

\subsection{Minimum degree and mirror images}

The minimum degree $d_-[J_{K,n}(q)]$ is also known to be a quadratic quasi-polynomial for sufficiently large $n$ as well as the maximum degree 
$d_-[J_{K,n}(q)]$. 
We set the quadratic quasi-polynomial to be $\delta^*_K(n) = a^*(n) n^2 + b^*(n) n+ c^*(n)$
for rational valued periodic functions $a^*(n), b^*(n), c^*(n)$ with integral period.
We may define the sets of {\em Jones slopes} of $K$ for minimum degree as 
$ js^*(K) = \{ 4a^*(n) \ |\ n \in \mathbb{N} \}$. 
Then the Slope Conjecture with the minimum degree asserts that $ js^*(K) \subset bs(K)$ for any knot $K \subset S^3$. 
Similarly define $ jx^*_K =  \{ 2b^*(n) \ |\ n \in \mathbb{N} \}$ and 
for a given $p^*/q^* \in js^*(K)$ $(q^* > 0)$, 
we say that $p^*/q^*$ satisfies $SS^*(n)$ $(n \in \mathbb{N})$ if there is an essential surface $F^*_n$ in the exterior of $K$ such that 
\begin{itemize}
\item
$F^*_n$ has the boundary slope $4a^*(n) = p^*/q^*$, and 
\item
$2b^*(n) = - \dfrac{\chi(F^*_n)}{| \partial F^*_n|q^*}$. 
\end{itemize}
The essential surface $F^*_n$ is also called a \textit{Jones surface} of $K$.  
Then the Strong Slope Conjecture with the minimum degree asserts that 
every $p^*/q^* \in js^*(K)$ satisfies $SS^*(n)$ for some $n \in \mathbb{N}$. 

For a quadratic quasi-polynomial 
$\delta^*_K(n) = a^*(n)n^2 + b^*(n) n + c^*(n)$ with period $\le 2$, 
We put $a^*_i = a^*(i),\ b^*_i = b^*(i)$ and $c^*_i= c^*(i)$, where $i = 1, 2$. 

\begin{theorem}[The Slope  Conjecture and Strong Slope Conjecture for the minimum degree of Mazur doubles]
\label{strong_slope_conjecture_min_deg}
Let $K$ be a knot. 
We assume that the period of $\delta^*_K(n)$ is less than or equal to  $2$.  Assume that $b^*_i \ge 0$, 
that $b^*_i=0$ implies $a^*_i \neq 0$, 
and that $a^*_1 + b^*_1 + c^*_1 \ge 0$. 
Furthermore, assume that $d_{-}[J_{K, n}(q)] = \delta^*_K(n)$ for all $n \ge 2$. 
\begin{enumerate}
\item
If $K$ satisfies the Slope Conjecture  for the minimum degree, 
then its Mazur double also satisfies the Slope Conjecture for the minimum degree. 
\item
If $K$ satisfies the Strong Slope Conjecture  for the minimum degree, 
then its Mazur double also satisfies the Strong Slope Conjecture for the minimum degree. 
\end{enumerate}
\end{theorem}

\begin{remark}
\label{min_condition}
\begin{enumerate}
\item
Every torus knot satisfies the condition of Theorem~\ref{strong_slope_conjecture_min_deg}. 
See \cite[4.8]{Garoufalidis} in which Garoufalidis uses $J_{K, n}(q)$ to mean the normalized colored Jones polynomial $J'_{K, n}(q)$. 

\item
$A$-adequate knots and $B$--adequate knots satisfies the condition of Theorem~\ref{strong_slope_conjecture_min_deg} \(\cite[Lemma~3.6]{KT}\). 
In particular, adequate knots and hence alternating knots satisfy the condition of Theorem~\ref{strong_slope_conjecture_min_deg}. 
\end{enumerate}
\end{remark}

Denote by $\overline{K}$ the mirror image of $K$. 
Note that the companion knot of $\overline{M(K)}$ is $\overline{K}$. 
The quadratic quasi-polynomial associated with $d_+[J_{\overline{K}, n}(q)]$ is denoted by 
$\delta_{\overline{K}}(n) = \overline{a}(n)n^2 + \overline{b}(n)n + \overline{c}(n)$.  
Recall that 
$\delta_{\overline{K}}(n) = -\delta^*_K(n)$, 
and hence $js(\overline{K}) = - js^*(K)$. 

\begin{lemma}
\label{mirror}
$K$ satisfies the Slope Conjecture and the Strong Slope Conjecture for the minimum degree if and only if 
$\overline{K}$ satisfies the Slope Conjecture and the Strong Slope Conjecture \(for the maximum degree\). 
\end{lemma}

\begin{proof}
Assume that $K$ satisfies the Slope Conjecture and the Strong Slope Conjecture for minimum degree. 
Pick a Jones slope $\overline{p}/\overline{q} \in js(\overline{K})= - js^*(K)$. 
Then the Jones slope (for minimum degree) $(-\overline{p})/\overline{q} \in js^*(K)$ satisfies $SS^*(n)$ for some $n$, 
i.e. there exists an essential surface $F^*_n \subset E(K)$ such that 
the boundary slope of $F^*_n$ is $4a^*(n) = (-\overline{p})/\overline{q}$ and $2b^*(n) = -\frac{\chi(F^*_n)}{|\partial F^*_n| \overline{q}}$. 
Let us take the mirror images of $K$ and $F^*_n \subset E(K)$ to obtain $\overline{K}$ and $\overline{F^*_n} \subset E(\overline{K})$. 
Then the boundary slope of $\overline{F^*_n}$ is $-4a^*(n) = 4\overline{a}(n) = \overline{p}/\overline{q}$ and 
$2\overline{b}(n) = 2(-b^*(n)) = \frac{\chi(F^*_n}{| \partial F^*_n| \overline{q}} 
=  \frac{\chi(\overline{F^*_n})}{| \partial \overline{F^*_n}| \overline{q}}$. 
Thus the Jones slope $\overline{p}/\overline{q} \in js(\overline{K})$ satisfies $SS(n)$ and, 
by definition 
$\overline{K}$ satisfies the Slope Conjecture and the Strong Slope Conjecture. 

Conversely, suppose that $\overline{K}$ satisfies the Slope Conjecture and the Strong Slope Conjecture. 
Pick a Jones slope $p^*/q^* \in js^*(K)$ for minimum degree. 
Then the Jones slope $(-p^*)/q^* \in -js^*(K) = js(\overline{K})$ satisfies $SS(n)$ for some $n$, 
i.e. we have an essential surface $\overline{F_n} \subset E(\overline{K})$ such that 
the boundary slope of $\overline{F_n}$ is $4\overline{a}(n) = (-p^*)/q^*$ and 
$2\overline{b}(n) = \frac{\chi(\overline{F_n})}{|\partial \overline{F_n}| \overline{q}}$. 
Take the mirror images of $\overline{K}$ and $\overline{F_n} \subset E(\overline{K})$ 
to obtain $K$ and $F_n \subset E(K)$. 
Then the boundary slope of $F_n$ is $-4\overline{a}(n) = 4a^*(n) = p^*/q^*$ and 
$-2b^*(n) = 2\overline{b}(n) =  \frac{\chi(\overline{F_n})}{|\partial \overline{F_n}| q^*} 
= \frac{\chi(F_n)}{|\partial F_n| q^*}$. 
This means that the Jones slope $p^*/q^*$ satisfies $SS(n)$,  
and hence $K$ satisfies the Slope Conjecture and the Strong Slope Conjecture for minimum degree. 
\end{proof}

Now Theorem~\ref{strong_slope_conjecture_min_deg}
can be rephrased as 

\begin{corollary}
\label{strong_slopeconjectures_mirror}
Let $K$ be a knot in $S^3$ which satisfies the condition in Theorem~\ref{strong_slope_conjecture_min_deg}. 
Then $\overline{M(K)}$ also satisfies the Slope Conjecture and the Strong Slope Conjecture. 
\end{corollary}

\begin{proof}
Theorem~\ref{strong_slope_conjecture_min_deg} asserts that 
$M(K)$ satisfies the Slope Conjecture and the Strong Slope Conjecture for the minimum degree. 
By Lemma~\ref{mirror} $\overline{M(K)}$ satisfies the Slope Conjecture and the Strong Slope Conjecture. 
\end{proof}

\begin{conjecture}
For any knot in $S^3$, 
$a(n) \ge 0$ and $a^*(n) \le 0$.
\end{conjecture}

Actually, if $a(n) < 0$, and $a^*(n) > 0$, 
then for sufficiently large $n$, 
$\delta_K(n) = a(n)n^2 + b(n) n + c(n)$ is strictly smaller than 
$\delta^*_K(n) = a^*(n)n^2 + b^*(n) n + c^*(n)$. 
This contradicts $\delta_K(n)$ is the maximum degree.
Hence such a knot should have $a(n),\  a^*(n) > 0$ or $a(n),\  a^*(n) < 0$.

\subsection{Crossing numbers of Mazur doubles}
Kalfagianni and Ruey Shan Lee have explicitly determined
the crossing number of untwisted Whitehead doubles of a nontrivial adequate knots with trivial writhe \cite{K-RSL}. 
Precisely, they show that $c(W_{\pm 1}^0(K)) = 4c(K) + 1$ for any nontrivial adequate knot with $wr(K) = 0$.  
This is the first instance of results that determine the crossing numbers for broad families of prime satellite knots.

Following their work, we use our calculations of the colored Jones polynomials of Mazur doubles to give strong constraints on their crossing numbers.  
\begin{theorem}
\label{crossing_Mazur}
Let $K$ be an adequate knot with crossing number $c(K)$ and the writhe $wr(K)$. 
Then we have the following. 
\begin{enumerate}
\item 
The Mazur double of $K$ is non-adequate. 
\item 
Assume that $c_+(K),\ c_-(K) \ne 0$. 
Then the crossing number of Mazur double of $K$ satisfies:
\[
9c(K) + 2 \le  c(M(K)) \le 9c(K)+3 + 6|wr(K)|.
\]
\item
If $wr(K) = 0$, 
then $c(M(K))$ is either $9c(K) + 2$ or $9c(K) + 3$. 
\end{enumerate}
\end{theorem}
This is proven in Section~\ref{sec:crossing}.

\begin{example}
\label{fig-eight}
Let $K$ be the figure-eight knot, whose minimal diagram with $c(K) = 4$ is alternating (hence adequate) and $wr(K) = 0$. 
So Theorem~\ref{crossing_Mazur} shows that the crossing number of its Mazur double is either $38$ or $39$. 
\end{example}

\bigskip

\subsection{Acknowledgments}
KLB was partially supported by a grant from the Simons Foundation (grant \#523883  to Kenneth L.\ Baker).
KM was partially supported by JSPS KAKENHI Grant Number JP19K03502, 21H04428, 
and Joint Research Grant of Institute of Natural Sciences at Nihon University for 2021.

\bigskip
\section{Colored Jones polynomials of Mazur doubles and preliminary degree computations}  
\label{Jones_generalized_M}

The goal of this section is to establish the common foundations for the proofs of Propositions~\ref{maxdeg_Mazur} and \ref{mindeg_Mazur} which give the maximum and minimum degrees of the colored Jones function of a Mazur double of a knot $K$ under the hypotheses of 
Theorems~\ref{thm:SC_and_SSC_M} and \ref{strong_slope_conjecture_min_deg}.

\medskip

\subsection{Normalizations and their transition.}
\label{transition}
For knot $K$ and a nonnegative integer $n$, set
\[J'_{K, n}(q):=\frac{J_{K, n+1}(q)}{J_{\bigcirc, n+1}(q)}\] 
so that $J'_{\bigcirc, n}(q)=1$ for the unknot $\bigcirc$ and 
$J'_{K, 1}(q)$ is the ordinary Jones polynomial of a knot $K$.

To derive Proposition~\ref{maxdeg_Mazur} 
from Proposition~\ref{maxdeg_Mazur_normalized} 
we apply the transformation with respect to normalization. 
Recall from the Introduction that $J_{\bigcirc, n+1}(q) = [n+1] = (-1)^{n}\langle n \rangle$ (where $[n+1]$ and $\langle n \rangle$ are defined in Equation~(\ref{def:<n>})). 
Then, in general, we have 
\begin{equation}
\label{KT-G-relation.formula}
\langle n \rangle J'_{K, n}(q) 
= \langle n \rangle \frac{J_{K, n+1}(q)}{J_{\bigcirc, n+1}(q)}
= \langle n \rangle  \frac{J_{K, n+1}(q)}{(-1)^n \langle n \rangle}
= (-1)^n J_{K, n+1}(q),\ J'_{K, 0}(q) =1,\ J_{K, 1}(q) =1. 
\end{equation}

This implies: 
\[
d_{\pm}\left[ \frac {q^{(n+1)/2}-q^{-(n+1)/2}}{q^{1/2}-q^{-1/2}}\right]+ d_{\pm} [J'_{K, n}(q)] = 
d_{\pm} [ J_{K, n+1}(q)]. 
\]
Since $\displaystyle d_{\pm}\left[ \frac {q^{(n+1)/2}-q^{-(n+1)/2}}{q^{1/2}-q^{-1/2}}\right] 
= \pm \frac{n}{2}$, 
we have
\[
\delta'_K(n) = \delta_K(n+1) - \frac{1}{2}n, \quad 
\delta'^*_K(n) = \delta^*_K(n+1) + \frac{1}{2}n
\]
and thus  
\[
\delta_{K}(n)=\delta'_{K}(n-1)+\frac{1}{2}n-\frac{1}{2}, 
\quad
\delta^*_{K}(n)=\delta'^*_{K}(n-1)-\frac{1}{2}n+\frac{1}{2}.
\]

Note also that 
comparing the coefficients of the quadratic quasipolynomial $\delta_K$ and its normalization $\delta'_K$ for which
\begin{align*}
\alpha (n) n^2 + \beta (n) n + \gamma(n) 
&= \delta'_K(n)\\
& = \delta_K(n+1)-\frac{1}{2}n \\
&= a(n+1) n^2 + b(n+1) n + c(n+1)  -\frac{1}{2}n,
\end{align*}
since we are assuming their periods are at most $2$, we have
\begin{center}
\begin{tabular}{lll}
 $\alpha_0=a_1$, & $\beta_0=2a_1+b_1-\frac{1}{2}$, & $\gamma_0=a_1+b_1+c_1$,\\
 $\alpha^*_0=a^*_1$, & $\beta^*_0=2a^*_1+b^*_1+\frac{1}{2}$, &$\gamma^*_0=a^*_1+b^*_1+c^*_1$\\
 \\

$\alpha_1=a_0$, & $\beta_1=2a_0+b_0-\frac{1}{2}$, &$\gamma_1=a_0+b_0+c_0$,\\
$\alpha^*_1=a^*_0$, &$ \beta^*_1=2a^*_0+b^*_0+\frac{1}{2}$, &$\gamma^*_1=a^*_0+b^*_0+c^*_0$, 
\end{tabular}
\end{center}
where 
\begin{center}
\begin{tabular}{llllll}
 $a_0=a(2m)$, & $b_0=b(2m)$, & $c_0=c(2m)$, &   
 $a_1=a(2m+1)$, & $b_1=b(2m+1)$, & $c_1=c(2m+1)$\\ 
 $\alpha_0=\alpha(2m)$, & $\beta_0=\beta(2m)$, & $\gamma_0=\gamma(2m)$, &  
 $\alpha_1=\alpha(2m+1)$, & $\beta_1=\beta(2m+1)$, & $\gamma_1=\gamma(2m+1)$. 
\end{tabular}
\end{center}

\bigskip

\subsection{Computations of colored Jones polynomials of $M(K)$}

Recall that $M(K)$ is the Mazur double of a knot $K$, 
which has the following diagram depicted in Figure~\ref{Mazur_double_diagram}. 
We use $D_3(K)$ to denote three parallels of a diagram of $K$ whose blackboard framing is $0$. 

\begin{figure}[!ht]
\begin{center}
\includegraphics[width=0.33\linewidth]{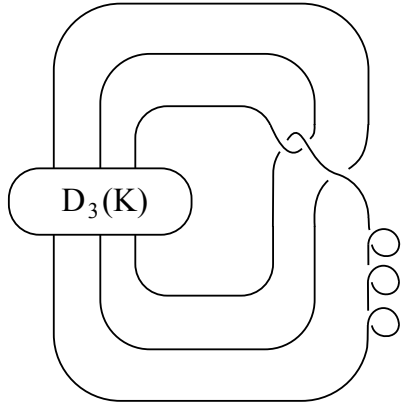}
\caption{A diagram of the Mazur double $M(K)$ of $K$ with trivial writhe}
\label{Mazur_double_diagram}
\end{center}
\end{figure}

We begin by recalling the following functions with respect to $q$ for  non-negative integers $s,t,u$. 
See \cite{MV}.

\begin{eqnarray}
\label{def:<n>}
&& \langle s \rangle  := (-1)^s [s+1], \quad [s]= \frac{q^{s/2} - q^{-s/2}}{q^{1/2} - q^{-1/2}},
\quad [s]! = \prod_{t=1}^s[t] 
\end{eqnarray}

\begin{eqnarray}
&& \langle s,t,u \rangle := (-1)^{i+j+k} 
\frac {[i+j+k+1]![i]![j]![k]!}{[s]![t]![u]!}, 
\end{eqnarray}
where $i=\frac {t+u-s}{2}$, $j=\frac {u+s-t}{2}$, and $k=\frac {s+t-u}{2}$,
\begin{eqnarray}\label{delta}
&& \delta(u;s,t):=(-1)^{\frac {s+t+u} 2} q^{-\frac{1}{8}(u^2-s^2-t^2+2u-2s-2t)},
\end{eqnarray}
and

\begin{equation}\label{tetsymb}
\left\langle \begin{array}{ccc} A & B & E \\ D & C & F \end{array}\right\rangle
=\frac {\prod_{i=1}^3 \prod_{j=1}^4 [b_i-a_j]!}
  {[A]![B]![C]![D]![E]![F]!}
  \sum_{\max\{a_j\} \le s \le \min \{b_i \}}
  \frac {(-1)^s [s+1]!} {\prod_{i=1}^3 [b_i-s]!\prod_{j=1}^4 [s-a_j]!}, 
\end{equation}
where  
$a_1=\frac {A+B+E} 2$, $a_2=\frac {B+D+F} 2$, $a_3=\frac {C+D+E} 2$, 
$a_4=\frac {A+C+F} 2$, $\Sigma=A+B+C+D+E+F$, 
$b_1=\frac {\Sigma-A-D} 2$, $b_2=\frac {\Sigma-E-F} 2$, and  
$b_3=\frac {\Sigma-B-C} 2$. 

\medskip

By definition, we have the following. 

\begin{remark}
\label{6j}
$\langle s,t,u \rangle = \langle \sigma(s), \sigma(t), \sigma(u) \rangle$ for any permutation $\sigma$, 
and  
\[\left\langle \begin{array}{ccc} A & B & E \\ D & C & F \end{array}\right\rangle
= 
\left\langle \begin{array}{ccc} E & C & D \\ F & B & A \end{array}\right\rangle\]
\end{remark}

\medskip

We will also use the following equalities introduced by Masbaum and Vogel \cite{MV}.  equalities  (refer \cite{MV}).

\begin{equation}
\label{pall}
 \bpc{(82,25)
  \put(10,-1.5){\pcr{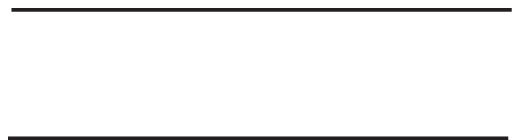}{0.4}}
  \put(42,16){\nos $s$}
  \put(42,-17){\nos $t$} } 
  =\sum_u \frac {< u >}{<s,t,u>} 
  \bpc{(82,25)
  \put(0,-1){\pcr{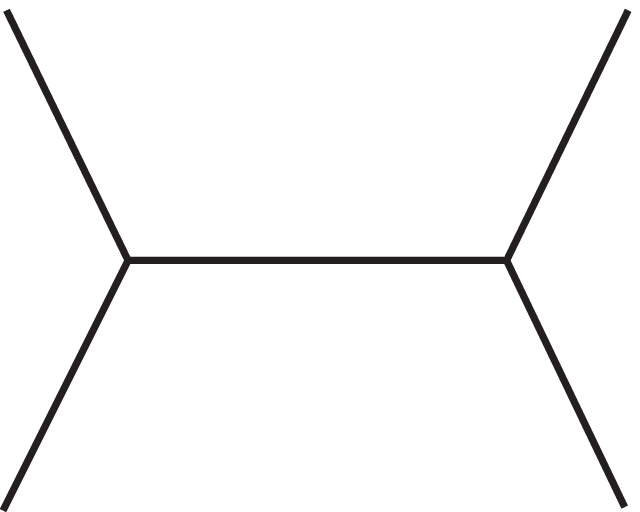}{0.4}}
  \put(11,29){\nos $s$}
  \put(11,-26){\nos $t$} 
   \put(38,10){\nos $u$}
   \put(65,28){\nos $s$}
  \put(64,-26){\nos $t$}} .
\end{equation}

\vskip 5truemm

Here the sum is over those colors $u$ such that the triple $(s,t,u)$ satisfies 
$s+t+u \equiv 0 \pmod 2$ and $|s-t| \le u \le s+t$. 
\begin{equation}
\label{twist-tri}
\hskip 5truemm 
 \bpc{(62,25)
  \put(3,-2){\pcr{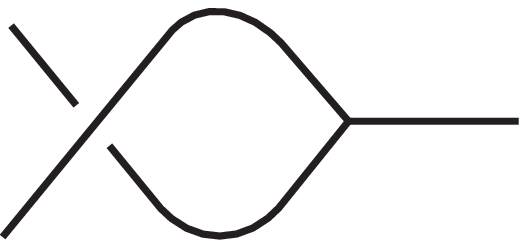}{0.4}}
  \put(4,20){\nos $s$}  
  \put(4,-22){\nos $t$}
   \put(55,10){\nos $u$} } \quad 
  =\delta (u; s, t)^{-1} 
  \bpc{(62,25)
  \put(0,-1){\pcr{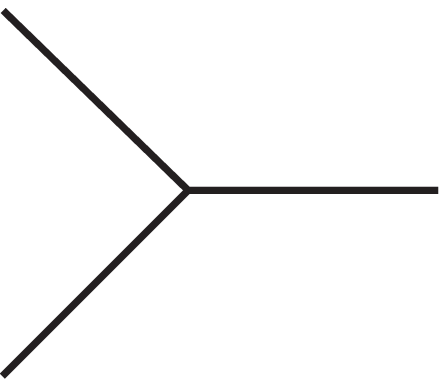}{0.4}}
  \put(6,28){\nos $s$}
  \put(6,-28){\nos $t$} 
   \put(45,10){\nos $u$}}, \quad
    \bpc{(52,25)
     \put(3,-2){\pcr{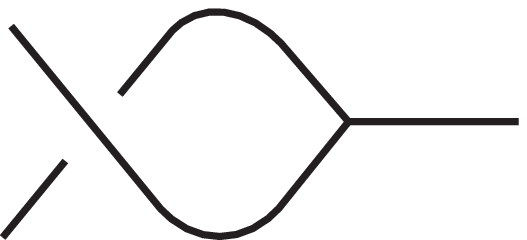}{0.4}}
  \put(5,20){\nos $s$}  
  \put(5,-22){\nos $t$}
   \put(55,10){\nos $u$} } \quad \quad
  =\delta (u; s, t) 
  \bpc{(62,25)
  \put(0,-1){\pcr{tri}{0.4}}
  \put(6,28){\nos $s$}
  \put(6,-28){\nos $t$} 
   \put(45,10){\nos $u$}}
  \end{equation}  

\vskip 5truemm

\begin{equation}
\label{curl}
\hskip 5truemm 
 \bpc{(45,25)
  \put(3,-2){\pcr{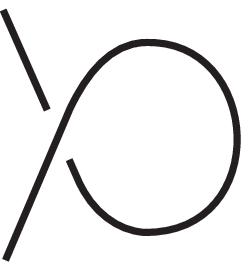}{0.5}}
  \put(30,20){\nos $s$}} 
  =\delta (0; s, s)^{-1} 
  \bpc{(22,25)
  \put(0,-1){\pcr{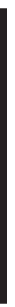}{0.5}}
  \put(10,3){\nos $s$}}, \quad \quad 
  \bpc{(45,25)
  \put(3,-2){\pcr{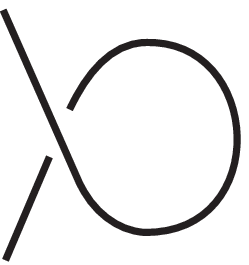}{0.5}}
  \put(30,20){\nos $s$}} 
  =\delta (0; s, s) 
  \bpc{(82,25)
  \put(0,-1){\pcr{1-string}{0.5}}
  \put(10,3){\nos $s$}} 
\end{equation}

\vskip 0.5truecm

\begin{equation}\label{tri-open}
  \bpc{(52,25)
  \put(5,-2){\pcr{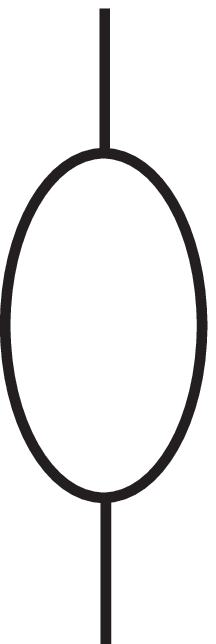}{0.3}}
  \put(22,22){\nos $u$}
  \put(3,0){\nos $s$}
  \put(32,0){\nos $t$}
  \put(22,-22){\nos $u$}} 
  =\frac {<s,t,u>} {< u >} 
  \bpc{(82,25)
  \put(0,-1){\pcr{1-string}{0.5}}
  \put(10,3){\nos $u$}} 
\end{equation}

\vskip 1truecm

\begin{equation}\label{tetra-s}
\bpc{(42,25)
  \put(-20,15){\pcr{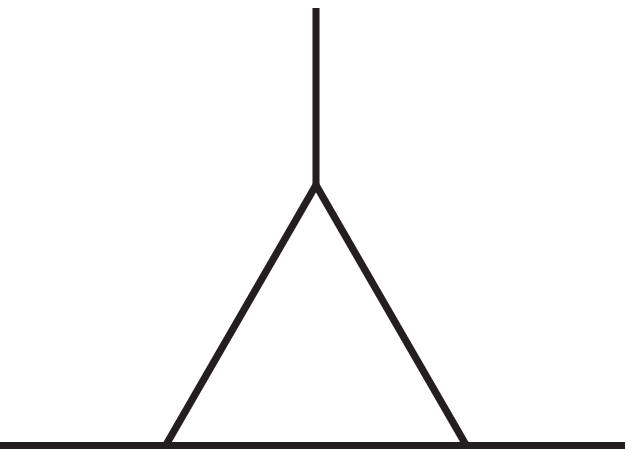}{0.3}}
  \put(-5,10){\nos $B$}
  \put(-10,-9){\nos $F$} 
   \put(15,26){\nos $A$}
   \put(23,10){\nos $E$}
  \put(10,-9){\nos $D$} 
  \put(30,-9){\nos $C$}} 
   =\frac {\left\langle \begin{array}{ccc} A & B & E \\ D & C & F \end{array}\right\rangle}
    {<A,F,C>} \quad
\bpc{(42,25)
  \put(-10,12){\pcr{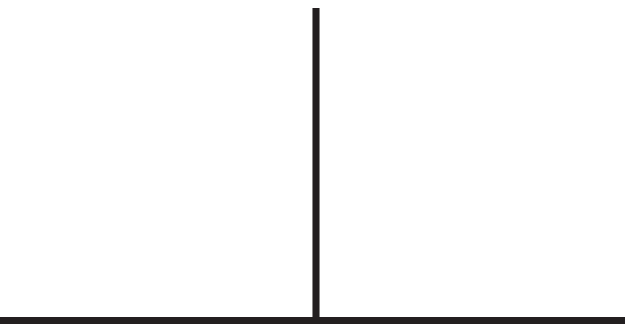}{0.3}}
 \put(5,-7){\nos $F$} 
   \put(25,13){\nos $A$}
   \put(35,-7){\nos $C$} } 
\end{equation}

\medskip

\begin{lemma}
\label{lem:collapse}

\quad 

\[
		\hskip 1.75truecm 
		\bpc{(30,30)
			\put(0,0){\pcr{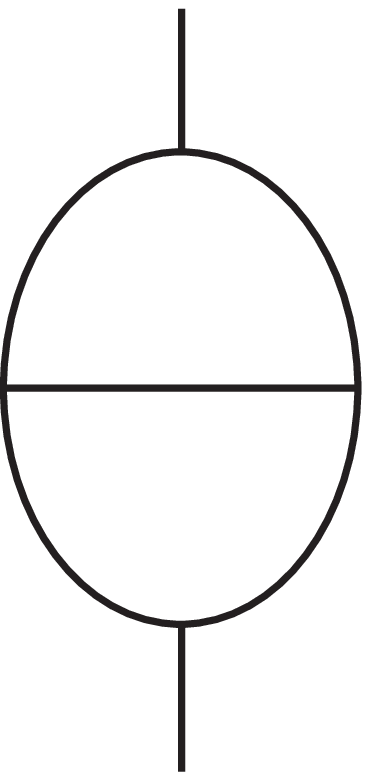}{0.4}}
			\put(12,37){\nos $2k$}
			\put(11,15){\nos $n$}
			\put(11,-12){\nos $n$}
			\put(18,9){ $2j$}
			\put(35,15){\nos $n$}
			\put(35,-12){\nos $n$}
			\put(12,-34){\nos $2k$}
		}   
	\quad \quad = \quad 
		\frac {\left\langle \begin{array}{ccc} n & n & 2j \\ n & n & 2k \end{array}\right\rangle}
		{\langle n,2k,n \rangle} \quad 
		\bpc{(52,25)
			\put(5,-2){\pcr{tri-open}{0.3}}
			\put(22,22){\nos $2k$}
			\put(2,0){\nos $n$}
			\put(30,0){\nos $n$}
			\put(22,-22){\nos $2k$}}
	= \quad 
		\frac {\left\langle \begin{array}{ccc} n & n & 2j \\ n & n & 2k \end{array}\right\rangle}
		{\langle 2k \rangle} 
		\bpc{(82,25)
			\put(0,-1){\pcr{1-string}{0.5}}
			\put(10,3){\nos $2k$}} 
\]
\end{lemma}
\vskip 1truecm
\begin{proof}
	Apply formula (\ref{tetra-s}) and then (\ref{tri-open}).
\end{proof}

\begin{lemma}
\label{graphicalcoloredjones}
For a $0$ framed diagram of any knot $K$:

\quad

 \[ 
 \langle n \rangle J'_{K,n}(q) = 
 \bpc{(30,30)
  \put(0,0){\pcr{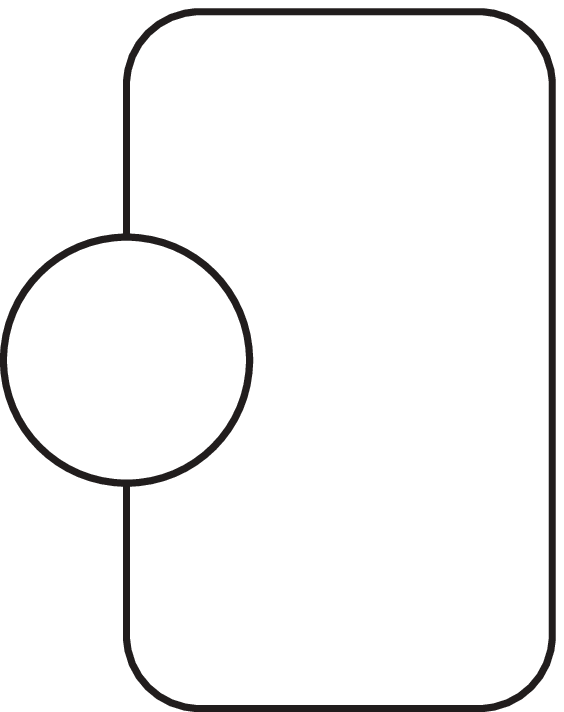}{0.4}}
  \put(12,0){ $K$}
   \put(9,37){\nos $n$}
     }
     \]
\end{lemma}
\vskip 1truecm
\begin{proof}
It follows from \cite[Section~5]{Mas} that the right hand side describes 
$(-1)^{n}J_{K, n+1}(q)$.  
On the other hand, 
(\ref{KT-G-relation.formula}) shows that 
\[
\langle n \rangle  J'_{K, n}(q) 
= (-1)^n J_{K, n+1}(q).
\]
Thus we obtain the desired equality.
\end{proof}

\vskip 0.5truecm

Now we are ready to compute the normalized colored Jones polynomial $J_{M(K), n}'(q)$. 

\begin{proposition}[\textbf{Normalized colored Jones polynomial of Mazur doubles}]  
\label{CJP_M(K)}

\begin{equation}\label{doublesum}
 J'_{M(K), n}(q) = \frac 1 {\langle n \rangle} \sum_{\substack{j,k=0 \\ j\le 2k}}^n \left( \sum_{\substack{|2k-n| \le l \le 2k+n \\ n+l:even}}
g(j,k,l;q) \right)
\end{equation}
where we define
\begin{multline}
g(j,k,l;q) := \\
\frac {\langle 2k \rangle}{\langle n,n,2k\rangle} \frac {\langle 2j\rangle}{\langle n,n,2\rangle} 
\frac { \left\langle 
\begin{array}{ccc} 2k & 2k & 2j \\ n & n & n \end{array}\right\rangle}
{\langle 2k,2k,2j \rangle}
\frac{ \delta(2j;n,n)^2 }{\delta(2k;n,n)} q^{-\frac{3}{4} n(n+2)} 
 \frac {\langle l \rangle}{\langle2k,n,l\rangle}
\left\langle 
\begin{array}{ccc} 2k & 2k & 2j \\ n & n & l \end{array}\right\rangle
J'_{K,l}(q).
\end{multline} 
\end{proposition}

\begin{proof}

Using the above formulas, we can compute

\begin{eqnarray*}
 &&\langle n \rangle J_{M(K), n}'(q)\\
 &=&  \sum_{k=0}^n \left(
\frac {\langle 2k \rangle}{\langle n,n,2k\rangle} q^{-\frac 3 4 n(n+2)}
\pcr{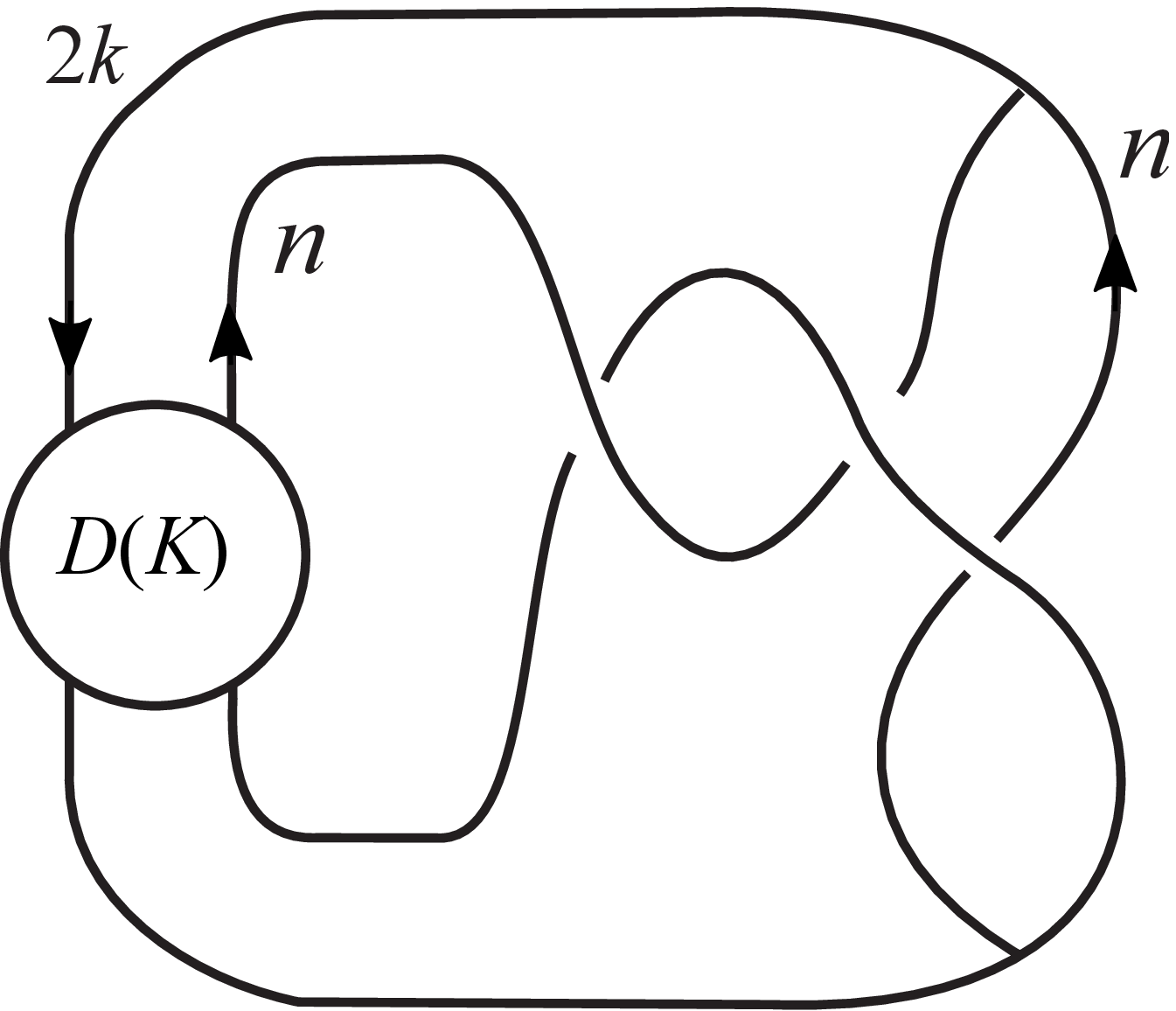}{0.25} \right)\\
 && \\
 &&\\
 &&\\
 &&\\
&=&  \sum_{j,k=0}^n \left(
\frac {\langle  2k \rangle}{\langle n,n,2k\rangle}\frac {\langle 2j \rangle}{\langle n,n,2j\rangle} q^{-\frac 3 4 n(n+2)}
\pcr{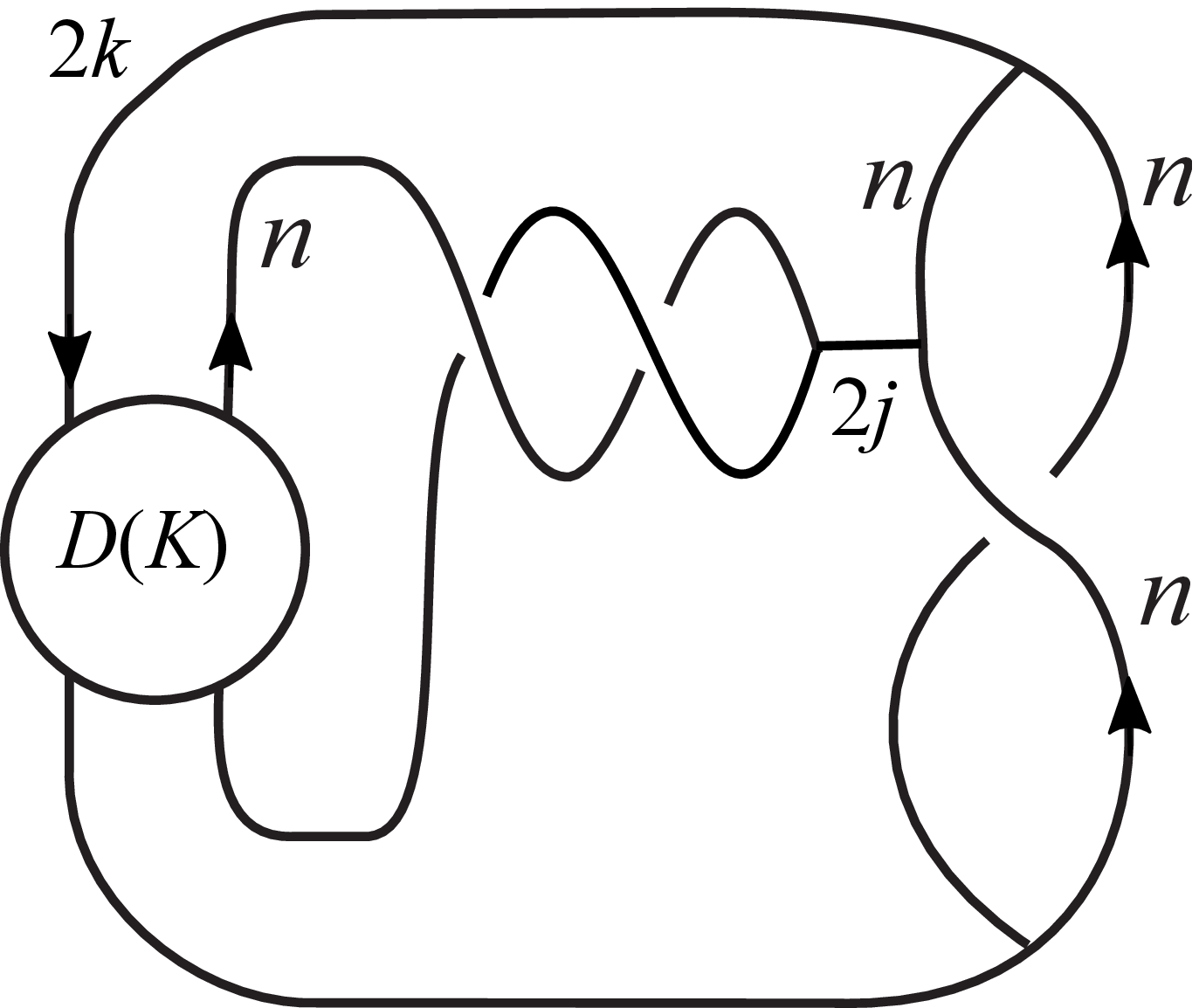}{0.25} \right) \\
 &&\\
 &=&  \sum_{j,k=0}^n \left(
\frac {\langle 2k \rangle}{\langle n,n,2k\rangle}\frac {\langle 2j \rangle}{\langle n,n,2j\rangle} q^{-\frac 3 4 n(n+2)}
\delta(2j;n,n)^{2}\delta(2k;n,n)^{-1}
\pcr{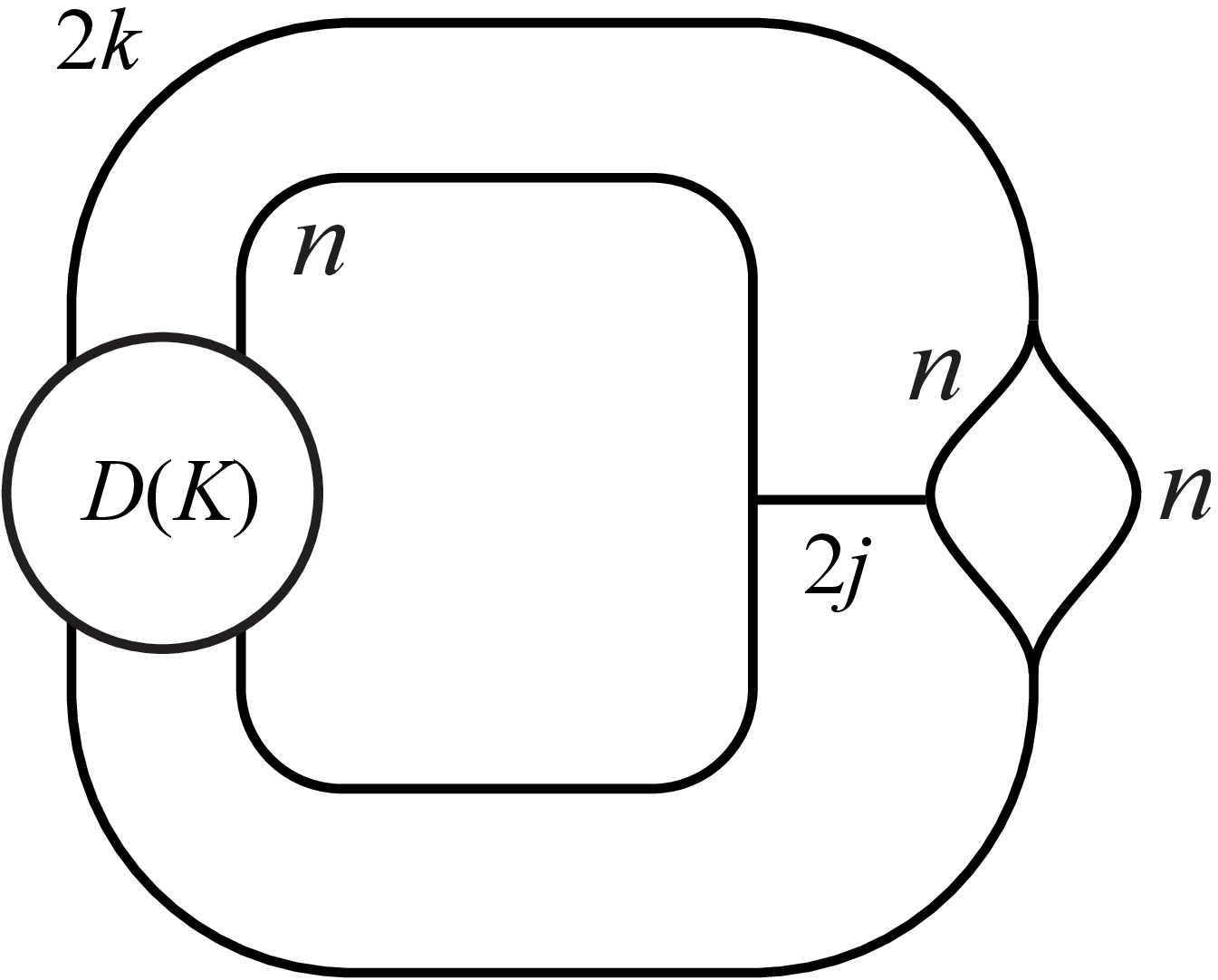}{0.25} \right)\\
 &&\\
 &=&  \sum_{\substack{j,k=0 \\ j\le 2k}}^n \left(
\frac {\langle 2k\rangle}{\langle n,n,2k\rangle}\frac {\langle  2j \rangle}{\langle n,n,2j\rangle} q^{-\frac 3 4 n(n+2)}
\delta(2j;n,n)^{2}\delta(2k;n,n)^{-1} \right. \\
& & \left. \hskip 6.5truecm \times 
\frac {\left\langle 
       \begin{array}{ccc} 2j & n & n \\ n & 2k & 2k \end{array}\right\rangle}
      {\langle 2j,2k,2k\rangle}
      \pcr{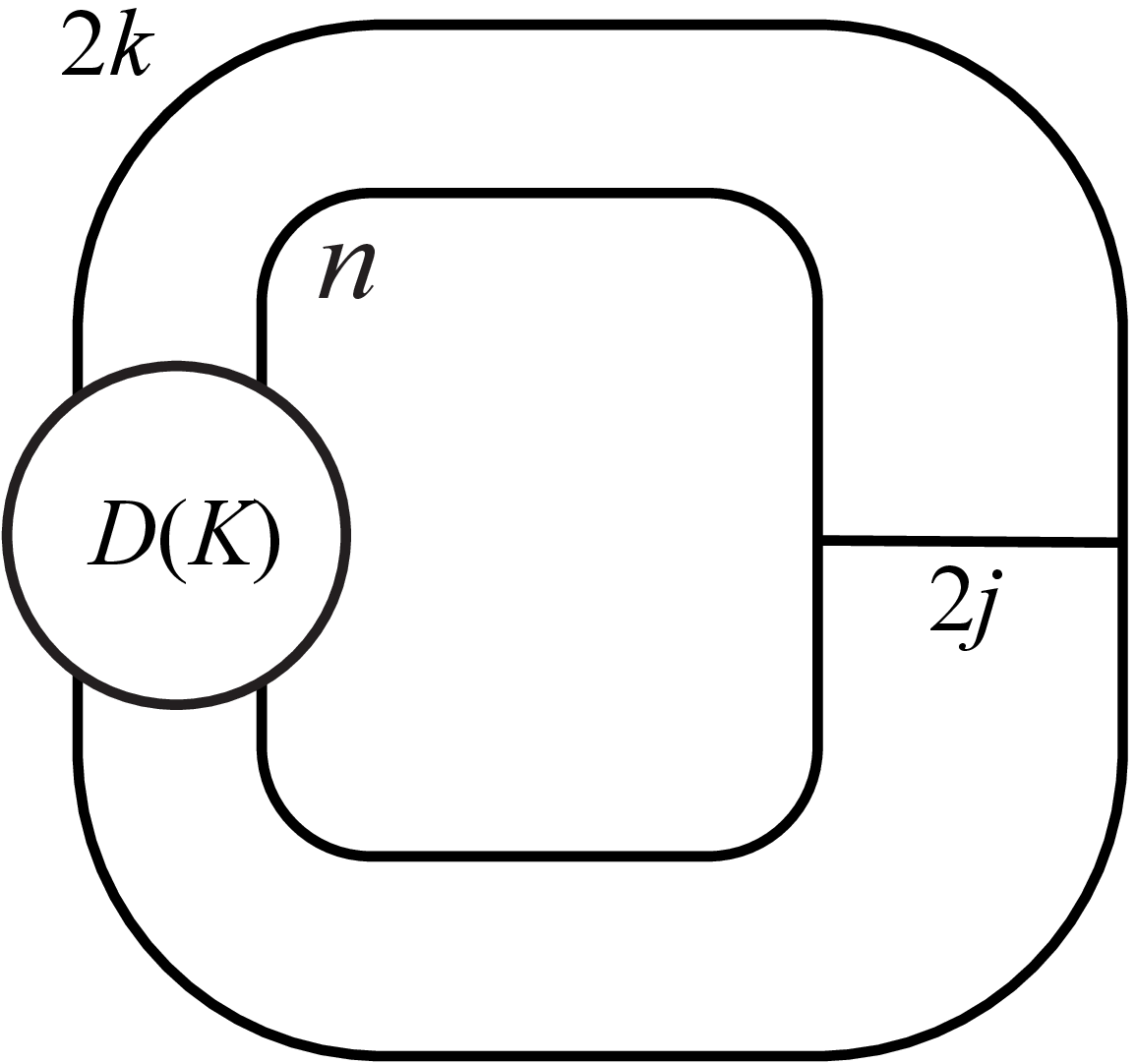}{0.25} \right)\\
 &&\\
  &=&  \sum_{\substack{j,k=0 \\ j\le 2k}}^n \left(
\frac {\langle  2k \rangle}{\langle n,n,2k\rangle}\frac {\langle  2j \rangle}{\langle n,n,2j\rangle} q^{-\frac 3 4 n(n+2)}
\delta(2j;n,n)^{2}\delta(2k;n,n)^{-1}
\frac {\left\langle 
       \begin{array}{ccc} 2k & 2k & 2j \\ n & n & n \end{array}\right\rangle}
      {\langle 2k,2k,2j\rangle} \right. \\
 & & \left. \hskip 5.5truecm \times 
\sum_{\substack{|2k-n|\le l \le 2k+n \\ n+l:even}} \frac {\langle l \rangle}{\langle 2k,n,l\rangle}
\pcr{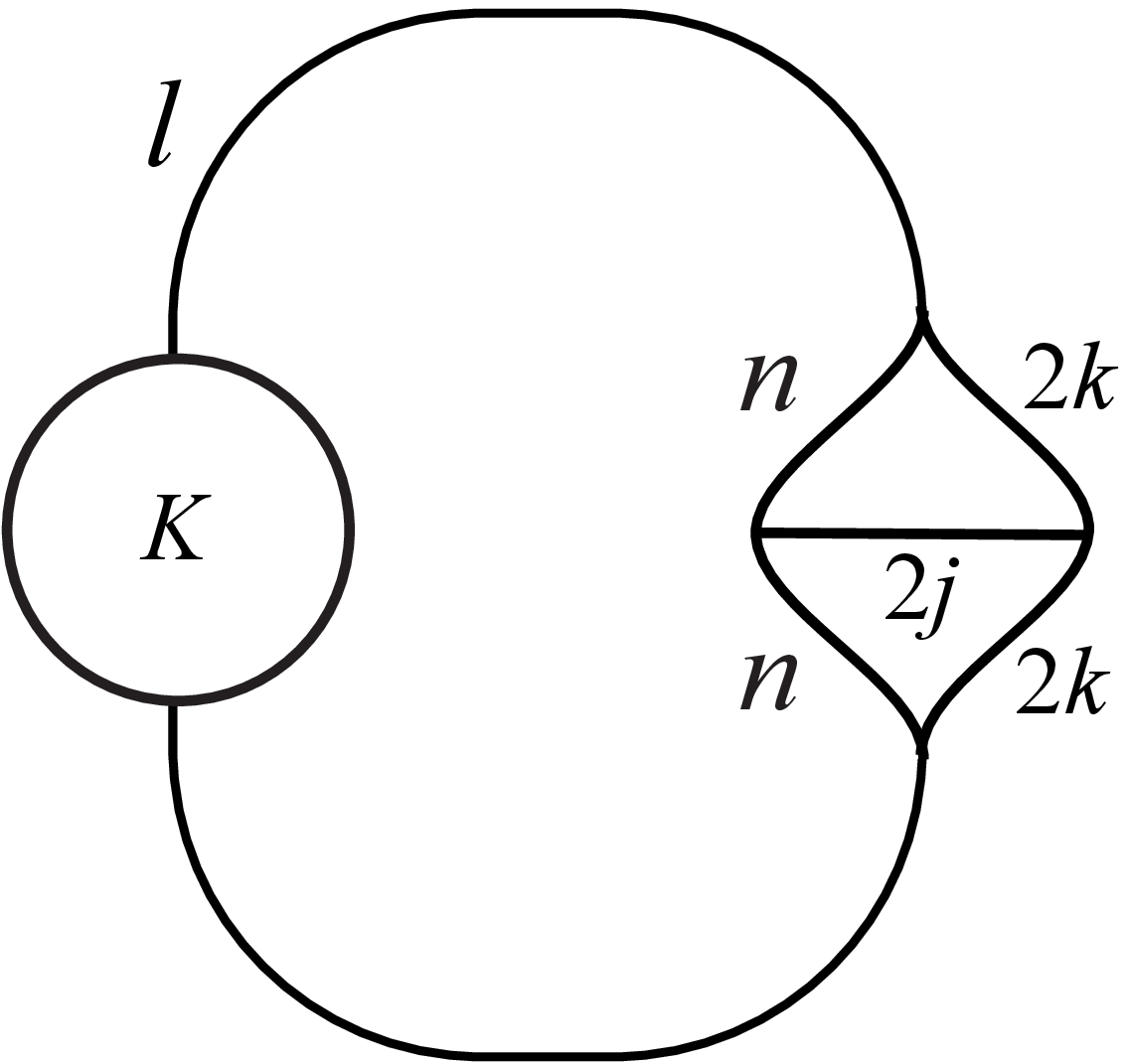}{0.25} \right)\\
 &&\\
   &=&  \sum_{\substack{j,k=0 \\ j\le 2k}}^n \left(
\frac {\langle  2k \rangle}{\langle n,n,2k\rangle}\frac {\langle  2j \rangle}{\langle n,n,2j\rangle} q^{-\frac 3 4 n(n+2)}
\delta(2j;n,n)^{2}\delta(2k;n,n)^{-1}
\frac {\left\langle 
       \begin{array}{ccc} 2k & 2k & 2j \\ n & n & n \end{array}\right\rangle}
      {\langle 2k,2k,2j\rangle}\right.\\
 & &  \left. \hskip 5truecm \times 
\sum_{\substack{|2k-n|\le l \le 2k+n \\ n+l:even}} \frac {\langle  l \rangle}{\langle 2k,n,l\rangle}
\frac {\left\langle 
       \begin{array}{ccc} l & n & 2k \\ 2j & 2k & n \end{array}\right\rangle}
      {\langle  l \rangle}
      \pcr{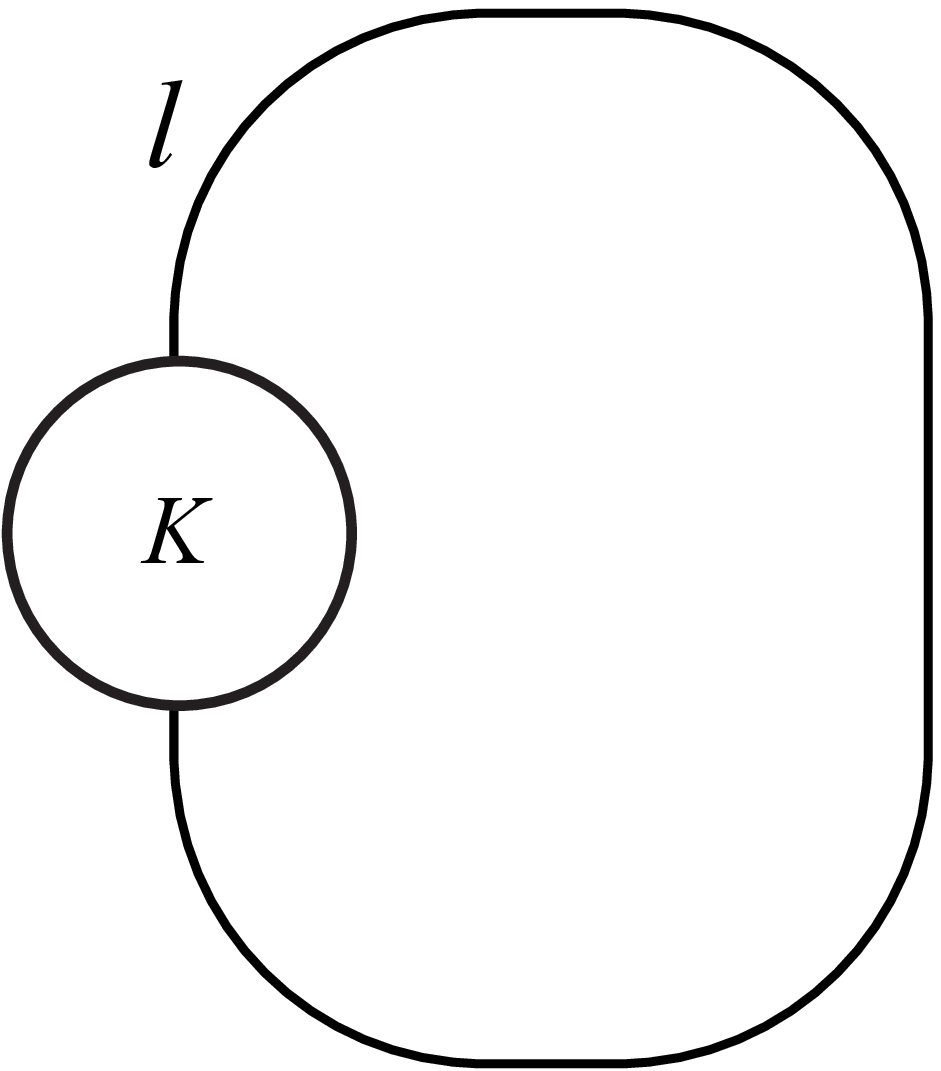}{0.25}\right)\\
 &&\\
   &=&  \sum_{\substack{j,k=0 \\ j\le 2k}}^n \left(
\frac {\langle  2k \rangle}{\langle n,n,2k\rangle}\frac {\langle  2j \rangle}{\langle n,n,2j\rangle} q^{-\frac 3 4 n(n+2)}
\delta(2j;n,n)^{2}\delta(2k;n,n)^{-1}
\frac {\left\langle 
       \begin{array}{ccc} 2k & 2k & 2j \\ n & n & n \end{array}\right\rangle}
      {\langle 2k,2k,2j\rangle}\right.\\
 & & \left.\hskip 6truecm \times 
\sum_{\substack{|2k-n|\le l \le 2k+n \\ n+l:even}} \frac {\langle  l \rangle}{\langle 2k,n,l\rangle}
 \left\langle 
       \begin{array}{ccc} 2k & 2k & 2j \\ n & n & l \end{array}\right\rangle
  J_{K,l}'(q)\right)\\
   &=&  \sum_{\substack{j,k=0 \\ j\le 2k}}^n \left( \sum_{\substack{|2k-n|\le l \le 2k+n \\ n+l:even}}
\frac {\langle  2k \rangle}{\langle n,n,2k\rangle}\frac {\langle  2j \rangle}{\langle n,n,2j\rangle} 
\frac {\left\langle 
       \begin{array}{ccc} 2k & 2k & 2j \\ n & n & n \end{array}\right\rangle}
      {\langle 2k,2k,2j\rangle}
\frac{\delta(2j;n,n)^{2}}{\delta(2k;n,n)} q^{-\frac 3 4 n(n+2)}
      \right.\\
 & & \left.\hskip 8truecm \times 
 \frac {\langle  l \rangle}{\langle 2k,n,l\rangle}
 \left\langle 
       \begin{array}{ccc} 2k & 2k & 2j \\ n & n & l \end{array}\right\rangle
  J_{K,l}'(q)\right)\\
  &=& \sum_{\substack{j,k=0 \\ j\le 2k}}^n \left( \sum_{\substack{|2k-n| \le l \le 2k+n \\ n+l:even}}
g(j,k,l;q) \right)
\end{eqnarray*}
\end{proof}

\subsection{Preliminary computations of the maximum and minimum degree}
\label{preliminary_degree_computation}

For a polynomial  $f(q)\in \Q [q^{\pm \frac{1}{4}}]$, its maximum degree is denoted by $d_+[f(q)]$ and its minimum degree is denoted by $d_-[f(q)]$.
We extend the maximum and minimum degrees to a rational function $f(q) = \frac{f_1(q)}{f_2(q)}$ with $f_1(q), f_2(q)\in \Q [q^{\pm \frac{1}{4}}]$ and $f_2(q)\ne 0$, by setting $d_+[f(q)] =d_+[f_1(q)]-d_+[f_2(q)]$ and $d_-[f(q)] = =d_-[f_1(q)]-d_-[f_2(q)]$.

Propositions~\ref{maxdeg_Mazur_normalized} and \ref{mindeg_Mazur_normalized} 
give the maximum and minimum degree of 
$J'_{M(K), n}(q) \in \mathbb{Z}[q^{\pm 1}]$ under various hypotheses.   
For convenience we recall the maximum and minimum degrees of the functions which appear in the expression 
of $J'_{M(K), n}(q)$ given in Proposition~\ref{CJP_M(K)}.

By Definition \ref{def:<n>} we have: 

\begin{lemma}
\label{<n>}
$d_{\pm}[\langle n \rangle] = \pm\dfrac{1}{2}n$.
\end{lemma}

\begin{lemma}[\cite{GV}] 
\label{deglem1}
The maximum and minimum degrees of $\langle s,t,u \rangle$ are  
 given by 
\[d_{\pm}[\langle s,t,u \rangle]= \pm\frac {s+t+u} 4.\]
\end{lemma}

\begin{lemma}[\cite{GV}] 
\label{deglem2} 
The maximum and minimum degrees of 
$\dis{
\left\langle \begin{array}{ccc} A & B & E \\ D & C & F \end{array}\right\rangle
}$ are given by 
\begin{eqnarray*}
&& d_{\pm} \left[\left\langle \begin{array}{ccc} A & B & E \\ D & C & F \end{array}\right\rangle \right]\\
&&= \pm\frac 1 2[-\Sigma^2-\frac 1 2 (A^2+B^2+C^2+D^2+E^2+F^2-\Sigma )
+\frac 3 2 \sum_{i=1}^3 b_i(b_i-1) +\sum_{j=1}^4 a_j(a_j+1)\\
&& \quad -3M^2+M(1+2\Sigma)], 
\end{eqnarray*}
where $\Sigma, a_j, b_i$ are as in (\ref{tetsymb}) and  $M=\min b_i$. 
\end{lemma}

Let us unravel the terms of $d_\pm[g(j,k,l;q)]$.
From Lemma \ref{deglem1}, 
we have that 
\begin{equation}
\label{pdeg3j}
 d_{\pm}\left[ \frac {\langle 2k \rangle}{\langle n,n,2k\rangle} \right]
 = \pm \left( -\frac n 2+\frac k 2 \right) 
 \qquad \mbox{and} \qquad
  d_{\pm}\left[ \frac {\langle 2j \rangle}{\langle n,n,2j\rangle} \right]
 = \pm \left(-\frac n 2+\frac j 2 \right).
\end{equation}
Equation (\ref{delta}), the definition of $\delta$, implies:
\begin{equation}
\label{pdegdelta}
d_{\pm}[\delta (2j; n,n)]
 =-\frac {j(j+1)} 2+\frac 1 4 n(n+2) 
\end{equation}
where $d_+ =d_-$ since $\delta(u;s,t)$ is a monomial.
Following Lemma \ref{deglem2}, we have
\begin{eqnarray}
\label{pdeg6ja}
 & & d_{\pm}\left[\left\langle \begin{array}{ccc} 2k & 2k & 2j \\ n & n & l \end{array}\right\rangle \right]\\
& & = \left\{\begin{array}{ll}
           \pm \frac 1 2 (j+k+\frac l 2+\frac n2) & (j\le \frac {n-l} 2+k), \\
           \pm \frac 1 2 (-j^2+2jk-k^2+2k-jl+kl+jn-kn-\frac {l^2} 4+ \frac n 2 l-\frac{n^2} 4+n) 
              & (j\ge \frac {n-l} 2+k),
            \end{array}
      \right. \nonumber
\end{eqnarray}
and in particular (when $l=n$), 
\begin{equation}
\label{pdeg6jb}
 d_{\pm}\left[ \left\langle \begin{array}{ccc} 2k & 2k & 2j \\ n & n & n \end{array}\right\rangle\right]
 = \left\{\begin{array}{ll}
           \pm \frac 1 2 (j+k+n) & (j\le k), \\
            \pm \frac 1 2 (-j^2+2jk-k^2+2k+n) 
              & (j\ge k).
            \end{array}
      \right. 
\end{equation}

From the equalities (\ref{pdeg3j}), (\ref{pdegdelta}), (\ref{pdeg6ja}),  
and (\ref{pdeg6jb}),  and the definition of $g(j,k,l;q)$ of Proposition~\ref{CJP_M(K)}, one then computes that 
\begin{lemma}
\label{lem:partialdegreecomputation}
\begin{multline*}
d_\pm[g(j,k,l;q)]  = \frac{1}{2} \left( -2j - 2j^2 + k + k^2 - 2n - n^2 \pm \left( -2k - \frac{5n}{2}+ \frac{l}{2}\right) \right) + d_\pm[J'_{K,l}(q)] \\
+  d_{\pm}\left[\left\langle \begin{array}{ccc} 2k & 2k & 2j \\ n & n & l \end{array}\right\rangle \right]
+ d_{\pm}\left[ \left\langle \begin{array}{ccc} 2k & 2k & 2j \\ n & n & n \end{array}\right\rangle\right]. \qed
\end{multline*} 
\end{lemma}

\section{The maximum degree of a Mazur double}

\subsection{Computations of the maximum degree}
\label{computation_degrees}

\begin{proposition}[\textbf{max-degree of $J'_{M(K), n}(q)$}]
\label{maxdeg_Mazur_normalized}
Let $K$ be a knot in $S^3$ and 
$N'_K$ the smallest nonnegative integer such that $d_+[J'_{K,n}(q)]$ is a quadratic quasi-polynomial 
$\delta'_K(n)=\alpha(n) n^2 +\beta(n) n+\gamma(n)$ for $n \ge 2N'_K$. 
We define $i \in \{0,1 \}$ by $i \equiv n \pmod 2$ and 
 put $\alpha_i:=\alpha (2N'_K+i)$, $\beta_i:=\beta (2N'_K+i)$, and $\gamma_i:=\gamma (2N'_K+i)$. 
We assume that the period of $\delta'_K(n)$ is less than or equal to $2$ and that 
$-2\alpha_i+\beta_i+\frac 1 2 \le 0$. 
Assume further that if $\alpha_i = 0$, then $\beta_i \ne -\frac 1 2$. 
Then, for suitably large $n$,  the maximum degree of the colored Jones polynomial 
of its Mazur double 
is given by 
\begin{equation}
  \delta'_{M(K)}(n)=
   \left\{ \begin{array}{ll} 
             9\alpha_i n^2 +3\beta_i n+\gamma_i & (\alpha_i> 0) \\
             \alpha_i n^2 +(\beta_i-1) n+\gamma_i  &  (\alpha_i \le 0).
       \end{array} \right. 
\end{equation}
\end{proposition}

\begin{proof}
From Proposition~\ref{CJP_M(K)} we have that 
\begin{equation*}
\langle n \rangle J'_{M(K), n}(q) =\sum_{\substack{j,k=0 \\ j\le 2k}}^n \left( \sum_{\substack{|2k-n| \le l \le 2k+n \\ n+l:even}}
g(j,k,l;q) \right).
\end{equation*}
Thus our goal now is to prove that, 
for each fixed integer $n$ above some bound of at least $2N'_K$,
in each case of $\alpha>0$ and $\alpha \le 0$ there exists a unique triple  $(j_0,k_0,l_0)$ such that 
\[
  \max_{\substack{0\le j,k \le n \\ |2k-n|\le l \le 2k+n \\ n+l:even}} d_+[g(j,k,l;q)]
    =d_+[g(j_0,k_0,l_0;q)].
\]
Then, due to the uniqueness, there can be no cancellation among the highest degree terms of the double sum Equation~(\ref{doublesum}) shown above so that $\delta'_{M(K)}(n)$ may be identified.

\medskip

From Lemma~\ref{lem:partialdegreecomputation}, 
we have that 
\begin{multline*}
d_+[g(j,k,l;q)]  = \frac{1}{2}\left(-2j-2j^2-k + k^2 + \frac{l}{2} - \frac{9n}{2} - n^2\right) + d_+[J'_{K,l}(q)] \\
+  \left\{\begin{array}{ll}
            \frac 1 2 (j+k+\frac l 2+\frac n2) & (j\le \frac {n-l} 2+k), \\
            \frac 1 2 (-j^2+2jk-k^2+2k-jl+kl+jn-kn-\frac {l^2} 4+ \frac n 2 l-\frac{n^2} 4+n) 
              & (j\ge \frac {n-l} 2+k),
            \end{array}
      \right. \\
+ \left\{\begin{array}{ll}
            \frac 1 2 (j+k+n) & (j\le k), \\
            \frac 1 2 (-j^2+2jk-k^2+2k+n) 
              & (j\ge k).
            \end{array}
      \right. 
\end{multline*}

To handle these last two summands, we will accordingly partition the $(j,k,l)$ domain and consider the following four cases: 

\medskip

\noindent
\textbf{Case (1).}\quad 
$j\le k$ and  $j\le \frac {n-l} 2+k$, 

\medskip

\noindent
\textbf{Case (2).}\quad  
$j\le k$ and $j\ge \frac {n-l} 2+k$, 

\medskip

\noindent
\textbf{Case (3).}\quad 
$j\ge k$ and  $j\le \frac {n-l} 2+k$, 

\medskip

\noindent
\textbf{Case (4).}\quad  
$j\ge k$ and $j\ge \frac {n-l} 2+k$.

\medskip
\noindent
We will see that the overall maximum of $d_+[g(j,k,l;q)]$ occurs in the domain of Case (1) and occurs uniquely.
The maxima on the domains of Cases (2) and (3) occur at their boundaries with the domain of Case (1), and the maxima on the domain of Case (4) occurs at its boundary with the domains of Cases (2) and (3).

\bigskip

\noindent 
\textbf{Case (1).} \quad  
$j\le k$ and  $j\le \frac {n-l} 2+k$.

From the equalities (\ref{pdeg3j}), (\ref{pdegdelta}), (\ref{pdeg6ja}),  
and (\ref{pdeg6jb}), 
one can  see that 
\[
d_+[g(j,k,l;q)]=\frac 1 2(-2j^2+k+k^2+l-n^2-3n)+d_+[J_{K,l}'(q)]. 
\]
This is maximized at $j=0$ for fixed $k$ and $l$ ($n$ is already fixed). 
Then we have 
\[
  \max_{0\le j \le \min \{k,\frac {n-l} 2+k \}}d_+[g(j,k,l;q)] 
    =d_+[g(0,k,l;q)] =\frac 1 2(k+k^2+l-n^2-3n)+d_+[J_{K,l}'(q)]. 
    \]
Since $k\le n$, 
this is maximized at $k = n$ for a fixed $l$. 
Then, as $2k-n \le l \le 2k+n$, 
putting $k = n$ gives $n \le l \le 3n$.  Hence $l \geq n \geq 2N'_K$ so that $d_+[J'_{K,l}(q)] = \delta'_K(l)$.  Thus
\[
  \max_{\max \{0,\frac {l-n} 2 \}\le k \le n}d_+[g(0,k,l;q)]
        =d_+[g(0,n,l;q)] =\frac 1 2(l-2n)+\delta_K'(l), 
\]
where the condition that both $0 \le k$ and $0 \le \frac{n-l}{2} + k$ is equivalent to $\mathrm{max}\{ 0, \frac{l -n}{2} \} \le k$.

Since $\delta_K'(l) = \alpha_i l^2 + \beta_i l + \gamma_i$, 
\[
d_+[g(0,n,l;q)] =
\frac 1 2(l-2n)+\delta_K'(l)
                            =\alpha_i l^2+ (\beta_i+\frac 1 2) l +\gamma_i-n.
\]
For a nonzero $\alpha_i$, this may be rewritten as
\[  
d_+[g(0,n,l;q)] =\alpha_i \left(l+\frac {2\beta_i +1}{4\alpha_i} \right)^2
                  - \frac {(2\beta_i +1)^2}{16\alpha_i}+\gamma_i-n. 
\]

If $\alpha_i >0$, since $-\frac {2\beta_i +1}{4\alpha_i} < n $ for a suitably large $n$, 
 then this is maximized uniquely at $l=3n$, and we have that 
\[
  \max_{n\le l\le 3n} d_+[g(0,n,l;q)]=d_+[g(0,n,3n;q)]=9\alpha_i n^2+(3\beta_i+\frac 1 2)n+\gamma_i. 
\]

\medskip

Now assume $\alpha_i \le 0$. 
Then from the assumption of Proposition~\ref{maxdeg_Mazur_normalized} that 
$-2\alpha_i+\beta_i+\frac 1 2 \le 0$, 
we have $\beta_i +\frac 1 2 \le 0$. 

\medskip

If $\alpha_i=0$, from the assumption that $\beta_i\ne -\frac 1 2$, 
then $\beta_i +\frac 1 2 < 0$. 
So, we get 
\[
  \max_{n\le l\le 3n} d_+[g(0,n,l;q)]=d_+[g(0,n,n;q)]=(\beta_i-\frac 1 2)n+\gamma_i. 
\]
If $\alpha_i<0$, since $\beta_i +\frac 1 2 \le 0$, then $-\frac {2\beta_i +1}{4\alpha_i} \le 0 $. 
So, we get 
\[
  \max_{n\le l\le 3n} d_+[g(0,n,l;q)]=d_+[g(0,n,n;q)]=\alpha_i n^2+(\beta_i-\frac 1 2)n+\gamma_i. 
\]

\smallskip 

\noindent 
\textbf{Case (2).}\quad 
$j\le k$ and  $j\ge \frac {n-l} 2+k$.

 We note that $n-l \le 0$ in this case, so $l \ge n \geq 2N'_K$.  (Also $2k+n \geq l \ge n$.)
 From the equalities (\ref{pdeg3j}), (\ref{pdegdelta}), (\ref{pdeg6ja}),  and (\ref{pdeg6jb}), 
one can  see that 

\begin{align*}
d_+[g(j,k,l;q)]
&=\frac 1 2\left(-j -3j^2 + 2k + 2jk - jl + kl + jn - kn - \frac {l^2}{4} + \frac {l(n+1)}{2}
      - \frac 5 4n^2 - \frac 5 2n\right)+\delta_K'(l)\\
   &=-\frac 3 2 \left(j-\frac 1 6(-1+2k-l+n)\right)^2+g_1(k,l,n),
\end{align*} 
where $g_1(k,l,n)$ is a function of $k$, $l$, and $n$.

Since
\[ j \ge \frac {n-l} 2+k \ge \frac 1 2 (2k-l+n) \ge \frac 1 6(2k-l+n) > \frac 16(-1+2k-l+n)\]
 we have that,  with these constraints,  $d_+[g(j,k,l;q)]$ is maximized at $j=\frac {n-l} 2+k \geq 0$. 
Therefor, this case is contained in the case (1).

\smallskip 

\noindent 
\textbf{Case (3).}\quad 
$j\ge k$ and  $j\le \frac {n-l} 2+k$.

 We note that $n-l \ge 0$ in this case. 
 From the equalities (\ref{pdeg3j}), (\ref{pdegdelta}), (\ref{pdeg6ja}),  and (\ref{pdeg6jb}), 
one can  see that 
\begin{eqnarray*}
d_+[g(j,k,l;q)]&=&\frac 1 2(-j-3j^2+2k+2jk+l-n^2-3n)+d_+[J_{K,l}'(q)],\\
      &=&-\frac 3 2(j-\frac 1 6(-1+2k))^2+g_2(k,l,n),
\end{eqnarray*} 
where $g_2(k,l,n)$ is a function of $k$, $l$, and $n$. 
Since $\frac  1 6(-1+2k) < k$ and $j\ge k$, 
 with these constraints, $d_+[g(j,k,l;q)]$ is maximized at $j=k$.
Therefore this case is contained in the case (1).

\smallskip 

\noindent 
\noindent 
\textbf{Case (4).}\quad 
$j\ge k$ and  $j\ge \frac {n-l} 2+k$.

 From the equalities (\ref{pdeg3j}), (\ref{pdegdelta}), (\ref{pdeg6ja}),  and (\ref{pdeg6jb}), 
one can  see that 
\begin{eqnarray*}
 & & d_+[g(j,k,l;q)]\\
&=&\frac 1 2(-2j-4j^2+3k+4jk-k^2-jl+kl+jn-kn-\frac {l^2} 4+\frac {l(n+1)} 2
      -\frac 5 4n^2-\frac 5 2n)+d_+[J_{K,l}'(q)],\\
      &=&-2(j-\frac 1 8(2+4k-l+n) )^2+g_3(k,l,n),
\end{eqnarray*} 
where $g_3(k,l,n)$ is a function of $k$, $l$, and $n$. 

When $k=0$, the condition $j\le 2k$ implies that $j=0$.  Then since $j=k$,  
this is contained in the case (2). 

So we assume that $k\ge 1$. 

\medskip

\noindent
(i) If $n\ge l$, then $j\ge \frac {n-l} 2+k\ge k$. Moreover, we have that
\[
(\frac {n-l} 2+k)-\frac 1 8(2+4k-l+n)=-\frac 1 4+\frac k 2+\frac 3 8(n-l)>0,
\]
since $k\ge 1$ and $n\ge l$. 
Thus, with these constraints, $d_+[g(j,k,l;q)]$ is maximized at $j= \frac {n-l} 2+k$.  Therefore this case is contained in 
the case (3). 

\medskip

\noindent
(ii) If $n\le l$, then $j\ge k \ge  \frac {n-l} 2+k$. Moreover, we have that
\[
k-\frac 1 8(2+4k-l+n)=-\frac 1 4+\frac k 2+\frac 1 8(l-n)>0,
\]
since $k\ge 1$ and $n\le l$. 
Thus, with these constraints, $d_+[g(j,k,l;q)]$ is maximized at $j= k$.  Therefore this case is contained in 
the case (2). 

\bigskip

Finally, considering the contribution to the maximum degree of $\frac 1 {\langle n \rangle}$, the maximization of $d_+[(j,k,l;q)]$ determined in Case (1) above gives the following.
\[ 
\delta'_{M(K)}(n)=\begin{cases} 
d_+[g(0,n,3n;q)]-\frac 1 2 n=9\alpha_i+3\beta_i n+\gamma_i & \mbox{ if } \alpha_i >0,\\
d_+[g(0,n,n;q)]-\frac 1 2 n=\alpha_i+(\beta_i-1) n+\gamma_i & \mbox{ if } \alpha_i \leq 0.
\end{cases}
\]
\end{proof}

\begin{proposition}[\textbf{max-degree of $J_{M(K), n}(q)$}]
\label{maxdeg_Mazur} 
Let $K$ be a knot in $S^3$, and let 
$N_K$ be the smallest nonnegative integer such that 
$d_+[J_{K,n}(q)]$ is a quadratic quasi-polynomial $\delta_K(n)=a(n) n^2 +b(n) n +c(n)$ 
for $n \ge 2N_K+1$. 
We assume that the period of $\delta_K(n)$ is at most $2$.
Then define $i \in \{0,1 \}$ by $i \equiv n \pmod 2$ and put $a_i:=a (2N_K + i)$, $b_i:=b (2N_K + i)$, and $c_i:=c (2N_K + i)$. 
Assume that $b_i \le 0$, and if $a_i = 0$, 
then $b_i \ne 0$. 
Then, for suitably large $n$, the maximum degree of the colored Jones polynomial 
of its Mazur  double 
is given by 
\begin{equation}
  \delta_{M(K)}(n)=
   \left\{ \begin{array}{ll} 
             9a_i n^2 +(- 12a_i+3b_i -1) n+4a_i-2b_i+c_i+1 & (a_i> 0), \\
             a_i n^2+(b_i-1) n+c_i+1  &  (a_i \le 0).
       \end{array} \right. 
\end{equation}
\end{proposition}

\begin{proof}
This 
is obtained from Proposition~\ref{maxdeg_Mazur_normalized} by using the formula \[
\delta_{M(K)}(n)={\delta'}_{M(K)}(n-1)+\frac{1}{2} n-\frac{1}{2}. 
\]
given in Subsection~\ref{transition}.
\end{proof}

\bigskip

\subsection{Mazur doubles and the Sign Condition}

Let us recall the Sign Condition introduced in \cite{BMT_tgW}, 
which will be used to prove Corollary~\ref{family_SCC}.   

\begin{definition}[\textbf{Sign Condition}]
\label{sign}
Let $\varepsilon_n(K)$  be the sign of the coefficient of the term of the maximum degree  of $J_{K, n}(q)$. 
A knot $K$ satisfies the \textit{Sign Condition} if 
$\varepsilon_m(K) = \varepsilon_n(K)$ for $m \equiv n\ \mathrm{mod}\ 2$. 
\end{definition}

Torus knots and $B$--adequate knots are known to satisfy the Sign Condition \cite{BMT_tgW}. 
Furthermore, this condition is preserved under taking connected sum, cables, untwisting $\omega$-generalized Whitehead doubles; see \cite[Proposition~6.6]{BMT_tgW}.  

We close this section by observing that the Sign Condition is preserved under taking Mazur doubles as well. 

\begin{proposition}
\label{Sign Condition}
If $K$ satisfies the Sign Condition, 
then $M(K)$ also satisfies the Sign Condition.
\end{proposition}

To prove this we show that the ``Normalized'' Sign Condition, analogous to the Sign Condition, 
is preserved under taking Mazur doubles.

\begin{definition}[\textbf{Normalized Sign Condition}]
\label{sign'}
Let $\varepsilon'_n(K)$  be the sign of the coefficient of the term of the maximum degree  of $J'_{K, n}(q)$. 
A knot $K$ satisfies the \textit{Normalized Sign Condition} if 
$\varepsilon'_m(K) = \varepsilon'_n(K)$ for $m \equiv n\ \mathrm{mod}\ 2$. 
\end{definition}

Since 
\[J'_{K, n}(q):=\frac{J_{K, n+1}(q)}{J_{\bigcirc, n+1}(q)}\] 
Proposition~\ref{Normalized Sign Condition} below immediately implies the desired result. 

\begin{proposition}
\label{Normalized Sign Condition}
If $K$ satisfies the Normalized Sign Condition, 
then $M(K)$ also satisfies the Normalized Sign Condition.
\end{proposition}

\begin{proof}
Recall that 

\begin{equation*}
\langle n\rangle J'_{M(K), n}(q) =\sum_{\substack{j,k=0 \\ j\le 2k}}^n \sum_{\substack{|2k-n| \le l \le 2k+n \\ n+l:even}}
g(j,k,l;q), 
\end{equation*}
where
\begin{eqnarray*}
g(j,k,l;q)&:=&\left(
\frac {\langle  2k \rangle }{\langle n,n,2k\rangle } \frac {\langle  2j \rangle }{\langle n,n,2j\rangle } 
\frac { \left\langle 
\begin{array}{ccc} 2k & 2k & 2j \\ n & n & n \end{array}\right\rangle}
{\langle 2k,2k,2j\rangle }
\delta(2j;n,n)^2 \delta(2k;n,n)^{-1}q^{-\frac{3}{4} n(n+2)} \right.\\
 & & \phantom{justasteptotheright} \left.
 \frac {\langle  l \rangle }{\langle 2k,n,l\rangle }
\left\langle 
\begin{array}{ccc} 2k & 2k & 2j \\ n & n & l \end{array}\right\rangle
\right) J'_{K,l}(q) .
\end{eqnarray*} 

Given a polynomial $f \in \Q[q^{\pm 1/4}]$, 
let $\ell_+[f]$ be the coefficient of the term of highest degree.
Since $\ell_+[\langle n \rangle^2] = 1$ and 
\begin{equation*}
\langle n\rangle ^2J'_{M(K), n}(q) 
= \langle n\rangle  \sum_{\substack{j,k=0 \\ j\le 2k}}^n \sum_{\substack{|2k-n| \le l \le 2k+n \\ n+l:even}}
g(j,k,l;q)
= \sum_{\substack{j,k=0 \\ j\le 2k}}^n \sum_{\substack{|2k-n| \le l \le 2k+n \\ n+l:even}}
\langle n\rangle  g(j,k,l;q), 
\end{equation*}
we have that 
\[\ell_+[  J'_{M(K), n}(q) ] 
= \ell_+\left[ \langle n \rangle^2 J'_{M(K), n}(q) \right] 
= \ell_+\left[ \sum_{\substack{j,k=0 \\ j\le 2k}}^n \sum_{\substack{|2k-n| \le l \le 2k+n \\ n+l:even}}
\langle n\rangle  g(j,k,l;q)\right].\]
It follows from the proof of Proposition~\ref{maxdeg_Mazur_normalized} that we have:
\begin{eqnarray*}
d_+[\langle n \rangle J'_{M(K), n}(q)] 
&=& d_+\left[ \sum_{\substack{j,k=0 \\ j\le 2k}}^n \sum_{\substack{|2k-n| \le l \le 2k+n \\ n+l:even}}
\langle n\rangle  g(j,k,l;q)\right] \\
&=& \left\{\begin{array}{ll}
d_+[ \langle  n \rangle  g(0, n, 3;q)] & \textrm{if}\ \alpha_1 >  0\\
d_+[ \langle  n \rangle  g(0, n, n;q)] & \textrm{if}\ \alpha_1\le 0
 \end{array}
      \right. \nonumber
\end{eqnarray*} 
Therefore, in the case of $\alpha_1 >  0$, 
we need to compute $\ell_+[g(0, n, 3n;q)]$, 
and in the case of $\alpha_1 \le 0$, 
we need to compute $\ell_+[g(0, n, n;q)]$. 
To do so, we first observe:
\begin{eqnarray*}
\langle  n \rangle  g(0, n, 3n;q)&:=&\langle  n \rangle  \left(
\frac {\langle  2n \rangle }{\langle n,n,2n\rangle } \frac {\langle  0 \rangle }{\langle n, n, 0\rangle } 
\frac { \left\langle 
\begin{array}{ccc} 2n & 2n & 0 \\ n & n & n \end{array}\right\rangle}
{\langle 2n, 2n, 0\rangle }
\delta(0; n, n)^2 \delta(2n; n, n)^{-1}q^{-\frac{3}{4} n(n+2)} \right.\\
 & &\phantom{justasteptotheright} \left.
 \frac {\langle  3n \rangle }{\langle 2n, n, 3n\rangle }
\left\langle 
\begin{array}{ccc} 2n & 2n & 0 \\ n & n & 3n \end{array}\right\rangle
\right) J'_{K, 3n}(q) 
\end{eqnarray*} 
and
\begin{eqnarray*}
\langle  n \rangle  g(0, n, n;q)&:=&\langle  n \rangle  \left(
\frac {\langle  2n \rangle }{\langle n,n,2n\rangle } \frac {\langle  0 \rangle }{\langle n, n, 0\rangle } 
\frac { \left\langle 
\begin{array}{ccc} 2n & 2n & 0 \\ n & n & n \end{array}\right\rangle}
{\langle 2n, 2n, 0\rangle }
\delta(0; n, n)^2 \delta(2n; n, n)^{-1}q^{-\frac{3}{4} n(n+2)} \right.\\
 & & \phantom{justasteptotheright}\left.
 \frac {\langle  n \rangle }{\langle 2n, n, n\rangle }
\left\langle 
\begin{array}{ccc} 2n & 2n & 0 \\ n & n & n \end{array}\right\rangle
\right) J'_{K, n}(q) .
\end{eqnarray*} 
Simple computation gives us the following: 
\[
\delta(0; n, n) = (-1)^nq^{\frac{1}{4}n^2 + 2n},\ 
\delta(2n; n, n) = q^{-\frac{1}{4}n^2}, 
\]
\[
\langle n\rangle  = (-1)^n[n+1],\ \langle 2n\rangle  = \langle n, n, 2n\rangle  = \langle 2n, n, n\rangle  = \langle 2n, 2n, 0\rangle  = [2n+1], 
\]
\[
 \langle 3n \rangle  = \langle 2n, n, 3n\rangle  = (-1)^{3n}[3n+1],\ \langle n, n, 0\rangle  = (-1)^n[n+1], 
\]
\[
\left\langle \begin{array}{ccc} 2n & 2n & 0 \\ n & n & n \end{array}\right\rangle 
= [2n+1],\quad 
\left\langle \begin{array}{ccc} 2n & 2n & 0 \\ n & n & 3n \end{array}\right\rangle
= (-1)^{3n}[3n+1].
\]
Hence we have: 
\begin{align*}
 \ell_+[ J'_{M(K), n}(q) ]  
 &= \ell_+\left[\langle  n \rangle  g(0, n, 3n;q) \right] \\
 &= (-1)^{3n}\ell_+\left[J'_{K, 3n}(q)\right] \quad \mathrm{for}\quad  \alpha_i >  0, 
  \end{align*} 
and 
\begin{align*}
 \ell_+[ J'_{M(K), n}(q) ]  
 &= \ell_+\left[\langle  n \rangle  g(0, n, n;q) \right] \\
 &= (-1)^n \ell_+\left[J'_{K, n}(q)\right]\quad \mathrm{for}\quad  \alpha_i \le 0. 
  \end{align*}
Since $K$ satisfies the Sign Condition and the parity of $3n$ and $n$ are the same, 
$M(K)$ also satisfies the Sign Condition. 
\end{proof}

\section{Computations of slopes and Euler characteristics for Mazur doubles}
\label{essential_surfaces}

\subsection{Exteriors of Mazur doubles knots and those of two-bridge links}
\label{exterior}

We start with a $2$--bridge link $k_1 \cup k_2$, which is expressed as 
$[2, 1, 4]$ depicted in Figure~\ref{2_1_4} below. 
Then $k_2$ lies in an unknotted solid torus $V = S^3 - \mathrm{int}N(k_1)$. 

\begin{figure}[!ht]
\includegraphics[width=0.6\linewidth]{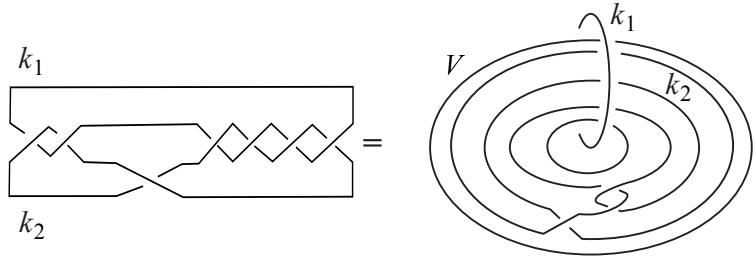}
\caption{$k_1 \cup k_2$ is a two bridge link $[2, 1, 4] $.}
\label{2_1_4}
\end{figure}

Let us take  preferred meridian-longitude pairs 
$(\mu_1, \lambda_1)$, $(\mu_2, \lambda_2)$ of $k_1$, $k_2$, respectively. 
Then take a faithful embedding 
$f \colon V \to S^3$ which sends the core of $V$ to a knot $K$ and 
$f(\mu_1) = \lambda_K$ and $f(\lambda_1) = \mu_K$, 
where $(\mu_K, \lambda_K)$ is a preferred meridian-longitude pair of $K$. 
The image $f(k_2)$ is $M(K)$, 
the Mazur double of $K$. 

Hence,  
the exterior of $M(K)$ is the union of 
the exterior $E(K)$ and $V - \mathrm{int}N(k_2)$; 
the latter is the exterior of the two-bridge link $k_1 \cup k_2$, 
which is expressed as $[2, 1, 4]$. 

In the following subsections~\ref{sec:two-bridge_link}--\ref{sec:Mazurgorithm}, 
we will investigate essential surfaces in the exterior $E(k_1 \cup k_2)$ of a two-bridge link 
$[2, 1, 4]$ in details. 

\subsection{Essential surfaces in two-bridge link exteriors}
\label{sec:two-bridge_link}
Here we catalogue all the properly embedded essential surfaces in the exterior of the two-bridge link $[2, 1, 4]$ for integers.

Hatcher-Thurston show how a certain collection of ``minimal edge paths'' in the Farey diagram from $1/0$ to $p/q$ are in correspondence with the properly embedded incompressible and $\bdry$--incompressible surfaces with boundary in the exterior of the two bridge knot $\calL_{p/q}$ \cite{HT}.
Floyd-Hatcher extend this to two-bridge links of two components \cite{FH} from which Hoste-Shanahan discern the boundary slopes of such surfaces \cite{HS}, building upon work of Lash \cite{Lash}.  

Here, for use with satellite constructions, 
we use the works of Floyd-Hatcher \cite{FH} and Hoste-Shanahan\cite{HS} to catalog all the properly embedded essential surfaces in the exterior of the Mazur double $M(K)$, 
their Euler characteristics, their boundary slopes, 
and number of boundary components.

\begin{remark}
While \cite{FH} uses the continued fraction convention $[x_1, x_2, \dots, x_n] = 1/(x_1 + 1/(x_2 + \dots 1/x_n))$, \cite{HS} appears to use the convention $[x_0, x_1, x_2, \dots, x_n] = x_0 + 1/(x_1 + 1/(x_2 + \dots 1/x_n))$.  
To remain consistent with this notation and the depiction of $\calL_{5/14}$ in \cite[Figure 1]{HS}. 
\end{remark}

We refer the reader to both the original paper \cite{FH} and Hoste-Shanahan's recounting of it \cite[Section 2]{HS} for details on the Floyd-Hatcher algorithm.  
Here we briefly recall the algorithm and quickly work through the application of it for 
the Mazur link $\calL_{5/14}$.

\subsection{The Algorithm}
\label{algorithm}
Figure~\ref{fig:FareyDiagrams} shows three diagrams. 
The diagram $D_1$ is the common Farey diagram.  
Pair adjacent triangles into quadrilaterals containing a diagonal so that a vertex is an endpoint of either all or none of the diagonals of the incident quadrilaterals.  
The diagram $D_0$ is obtained by switching the diagonal in each of the quadrilaterals.   
The diagram $D_t$ is obtained by replacing these diagonals with inscribed quadrilaterals.  
Actually, $D_t$ represents a parameterized family of diagrams for $t \in [0,\infty]$: 
with appropriate parameterizations of the edges of the quadrilaterals by $[0,1]$ the vertices of the inscribed quadrilaterals in $D_t$ are located at either $t$ or $1/t$.  
The diagrams $D_0=D_\infty$ and $D_1$ arise as limits where the inscribed quadrilaterals degenerate to diagonals.  
The edges of $D_1$ are labeled $A$ and $C$, 
the edges of $D_0=D_\infty$ are labeled $B$ and $D$, 
and these induce labels on $D_t$.   
Orientations are chosen on a basic set of edges in $D_t$ and passed to the rest of the edges of $D_t$ by the action of the M\"obius transformations in which the ideal triangle with vertices $\{1/0,0/1,1/1\}$ is a fundamental domain.  
We omit the orientations in Figure~\ref{fig:FareyDiagrams}; see \cite{HS} for details.
\begin{figure}[H] 
	\centering
	\includegraphics[width=.9\textwidth]{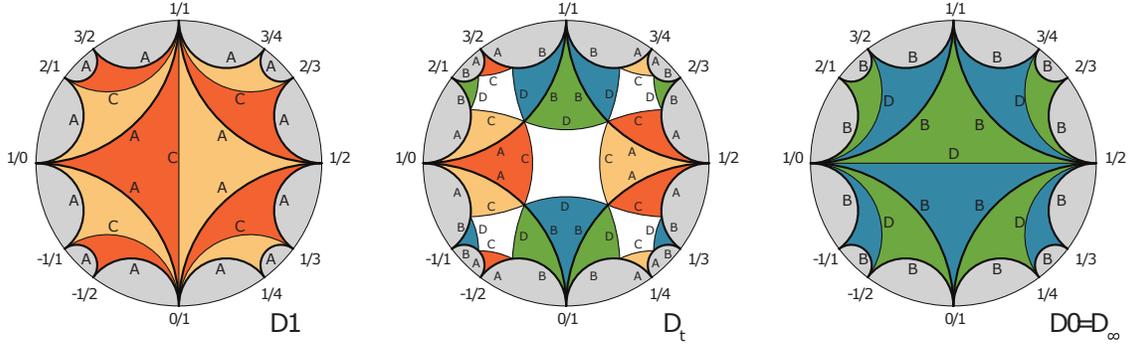}
	\caption{The diagrams $D_1$, $D_t$, and $D_0=D_\infty$, cf. \cite[Figures 3 and 4]{HS}. }
	\label{fig:FareyDiagrams}
\end{figure}
For a two bridge link $\calL_{p/q}$ (where $q$ is even), 
Floyd-Hatcher show that a properly embedded essential surface in the exterior of the link is carried by one of finitely many branched surfaces associated to ``minimal edge paths'' in $D_t$ from $1/0$ to $p/q$.  
A {\em minimal edge path} in $D_t$ is a consecutive sequence of edges of $D_t$ 
(ignoring their orientations) such that the boundary of any face of $D_t$ contains at most one edge of the path. 
Then for each minimal edge path, a branched surface is assembled from the sequence of edges by stacking four blocks of basic branched surface $\Sigma_A, \Sigma_B, \Sigma_C, \Sigma_D$ corresponding to the labels $A,B,C,D$ that are positioned according to the endpoints and orientation of its edge and whether $t< 1$ or 
$t > 1$.  
These blocks of basic branched surfaces are illustrated in Figure~\ref{fig:branchtypes} for 
$t > 1$
(cf.\ \cite[Figure 2]{HS} and \cite[Figure 3.1]{FH}) and are weighted in terms of the parameters 
$\alpha > \beta>0$ where  $t = \alpha/\beta$
and the extra integral parameter $n$ between $0$ and $\beta$ for $\Sigma_A$ or between $0$ and $\alpha-\beta$ for $\Sigma_D$. 
(This extra parameter $n$ allows for the construction of homeomorphic but non-isotopic surfaces with the same boundary slopes, see \cite{HS,FH}.) 
For $t < 1$, the blocks are rotated $180^\circ$ corresponding to an exchange of the components of $\calL_{p/q}$ and the parameters $\alpha$ and $\beta$ are swapped in the figure.  

\begin{figure}[H] 
\centering
\includegraphics[width=4.75in]{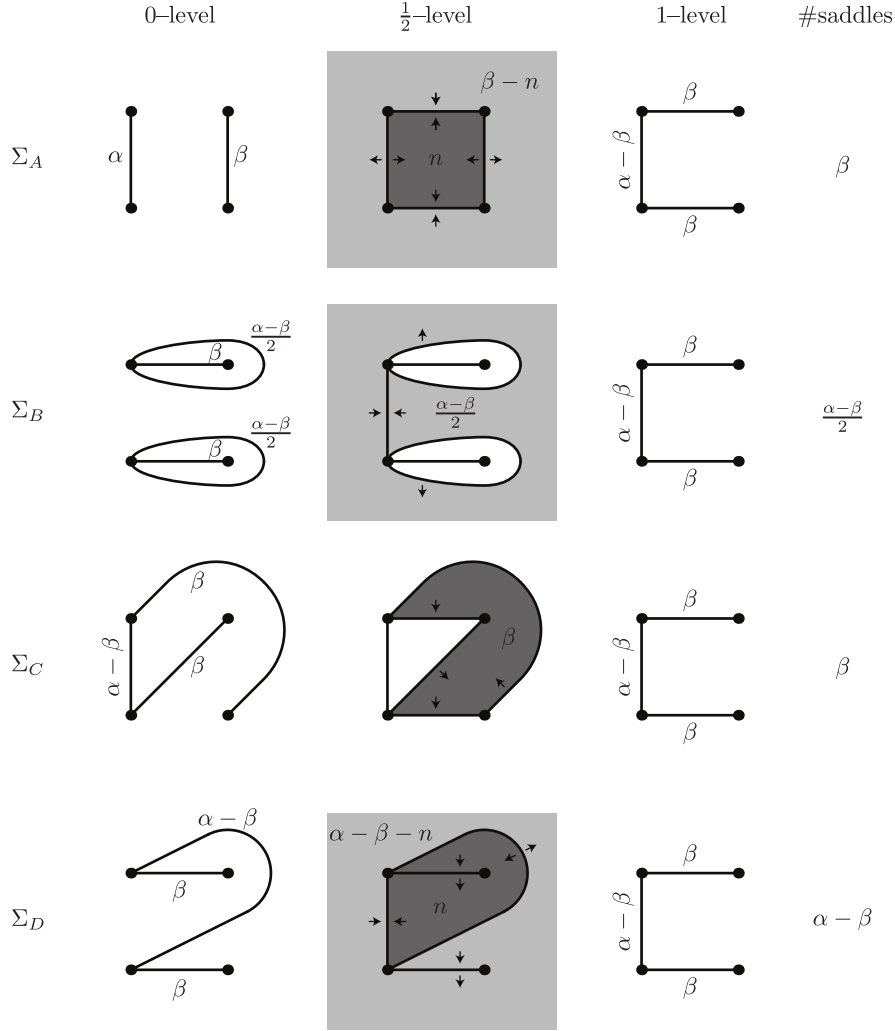}
\caption{The four basic weighted branched surfaces (reproduced from \cite[Figure 2]{HS}, see also \cite[Figure 3.1]{FH}) along with the corresponding number of saddles for the carried surface, when $\alpha\geq \beta$.  
When $\alpha < \beta$, rotate the images $180^\circ$ and swap $\alpha$ and $\beta$.}
\label{fig:branchtypes}
\end{figure}

 In this manner, every minimal edge path in $D_t$ for $t \in (0,1) \cup (1, \infty)$ produces a weighted branched surface, 
 with weights in terms of the parameters $\alpha$ and $\beta$ such that $t=\alpha/\beta$ 
 (along with auxiliary parameters for instances of the blocks $\Sigma_A$ and $\Sigma_D$). 
These minimal edge paths $\gamma$ in $D_t$ with their parameters $\alpha, \beta$ describe specific surfaces $F_{\gamma, \alpha, \beta}$ 
which may have multiple components and may be non-orientable. 
If it is non-orientable, 
then we may replace $F_{\gamma, \alpha, \beta}$ by the boundary of a tubular neighborhood (a twisted $I$--bundle over $F_{\gamma, \alpha, \beta}$), 
which is orientable and associated with parameters $2\alpha, 2\beta$; 
so the resulting orientable essential surface is associated with $F_{\gamma, 2\alpha, 2\beta}$. 
In the following we omit parameters $\alpha, \beta$ and assume that $F_{\gamma}$ is orientable, 
but it may have multiple components.

Taking the limits $t \to 0$ or $t \to \infty$ so that $\alpha=0$ or $\beta=0$ produces surfaces associated to minimal edges paths in $D_0 = D_\infty$.   
Taking the limits $t \to 1$ so that $\alpha = \beta$ also produces surfaces associated to minimal edge paths in $D_1$.  
However, since $\alpha-\beta=0$ in this case, 
the basic surface $\Sigma_A$ with its extra parameter $n$ may be used in place of $\Sigma_D$ to produce more surfaces.   

\medskip

Floyd and Hatcher \cite{FH} establish the following classification of essential surfaces in the exterior of two-bridge links. 

\begin{theorem}[\cite{FH}]
\label{FloydHatcher}
Let $\calL_{p/q}$ be a two-bridge link (with $q$ even). 
The orientable incompressible and meridionally incompressible surfaces in $S^3 - \nbhd(\calL_{p/q})$ without peripheral components are (up to isotopy) exactly the orientable surfaces carried by the collection of branched surfaces associated to minimal edge paths in $D_t$ from $1/0$ to $p/q$ for $t \in [0,\infty]$. 
\end{theorem}

\begin{remark}  
Let $S$ be a properly embedded surface in the exterior of a link $L$ in $S^3$.
Then $S$ is {\em meridionally incompressible}  if for any embedded disk $D$ in $S^3$ with $D \cap S = \bdry D$ such that $L$ intersects $D$ transversally in a single interior point, 
there is an annulus embedded in $S$ whose boundary is $\bdry D$ and a component of $\bdry S$ that is a meridian of $L$.
A component of $S$ is {\em peripheral} if it is isotopic through the exterior of $L$ into $\bdry \nbhd(L)$.
If $S$ has a $\bdry$--compressing disk, 
then either $L$ is a split link or the $\bdry$-compressible component of $S$ is either compressible or peripheral.  
Hence the surfaces in Theorem~\ref{FloydHatcher} are also $\bdry$--incompressible.
\end{remark}

\subsection{Euler characteristics of carried surfaces}

The Euler characteristic of a surface carried by one of these weighted branched surfaces associated to an edge path in $D_t$ may be calculated from the branch pattern associated to the edge path and the weights $\alpha$ and $\beta$. 

\begin{lemma}[\cite{BMT_tgW}]
Let $S$ be the surface carried by the weighted branched surface associated to an edge path $\gamma$ in $D_t$ where $t = \alpha/\beta$. 
If $\alpha \geq \beta$, then  \[\chi(S) = (\alpha+\beta) - \sum s_i(\alpha,\beta)\] where $s_i(\alpha,\beta)$ is the number of saddles of the surface carried by the basic branched surface associated to the label of the $i$th edge of $\gamma$ and weighted by $\alpha$ and $\beta$ as shown in Figure~\ref{fig:branchtypes}.  
If $\alpha<\beta$, exchange $\alpha$ and $\beta$.
\end{lemma}

\subsection{Boundary slopes and count of boundary components}
Note that surfaces carried by these branched surfaces are given by non-negative integral weights $\alpha$ and $\beta$ (with the auxiliary integral parameters $n$ as needed), 
and these weights indicate the algebraic (and geometric) intersection numbers of the surface with the meridians $\mu_1, \mu_2$ of the two components of $\calL_{p/q}$.

Hoste and Shanahan use a certain blackboard framing $\lambda_1, \lambda_2$ of the two components of $\calL_{p/q}$ to further keep track of how the branched surfaces associated to minimal edge paths in $D_t$ intersect this framing.  
They then determine how to correct this framing to the canonical framings $\lambda_1^0, \lambda_2^0$ of the individual unknot components of the two-bridge link. 
From this, one then obtains the boundary slopes of the carried surfaces in terms of the canonical framings of the components.

Furthermore, by a calculation in the homology of a torus, the greatest common divisor ($\gcd$) 
of the algebraic intersection numbers of the boundary of a surface with the meridian and longitudinal framing of a component of $\calL_{p/q}$ produces the number of boundary components of the surface meeting that component of $\calL_{p/q}$.

\subsection{Applying the Algorithm to the Mazur Link, the $2$--bridge link $ \lbrack  2, 1, 4 \rbrack  $}
\label{sec:Mazurgorithm}
Figure~\ref{fig:DtDiagram3-8} shows the portions of the diagrams $D_0=D_\infty$, $D_t$, 
and $D_1$ that carry the minimal edge paths from $1/0$ to $5/14$. 

\begin{figure}[H] 
	\centering
	\includegraphics[width=5in]{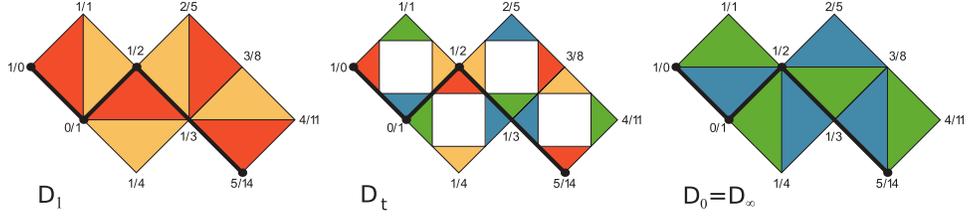}
	\caption{The diagrams $D_1$, $D_t$, and $D_0=D_\infty$ that carry the minimal edge paths from $\frac{1}{0}$ to $\frac{5}{14}$.}
	\label{fig:DtDiagram3-8}
\end{figure}

\begin{figure}[!ht]
	\includegraphics[width=.9\linewidth]{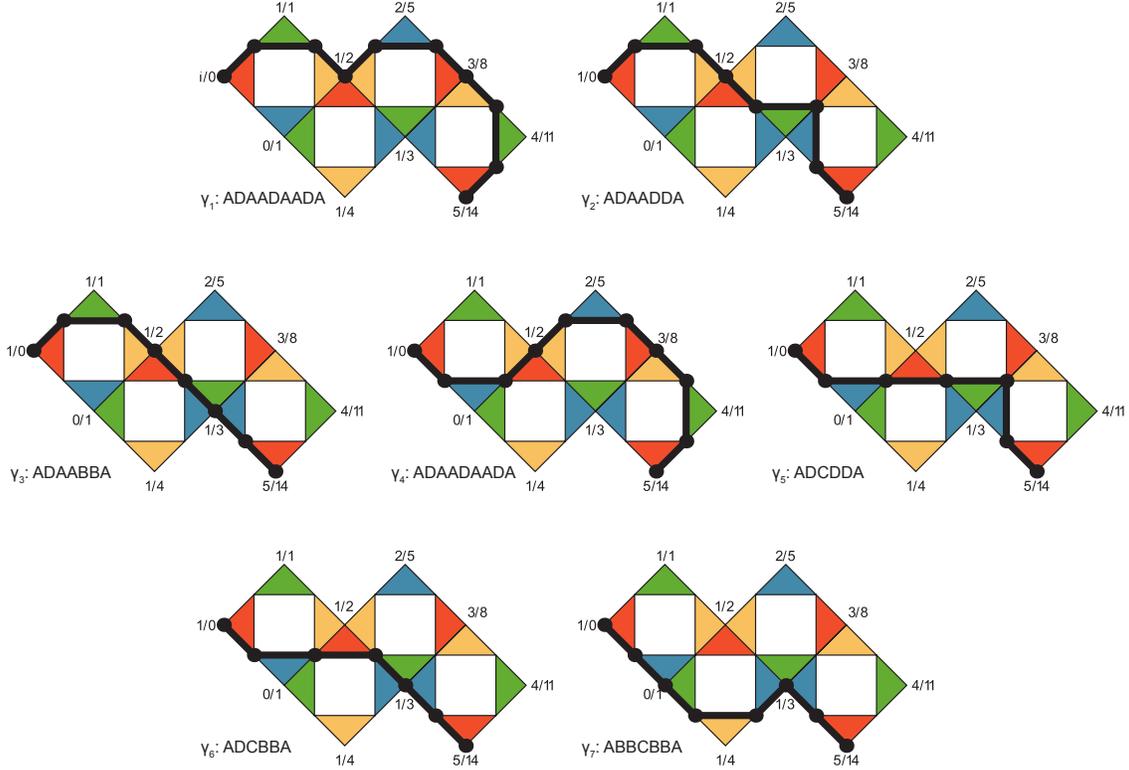}
	\caption{The minimal edge paths $\gamma_1, \dots, \gamma_7$ are shown  with their branch patterns. }
	\label{fig:edgepaths}
\end{figure}

Table~\ref{table:5/14paths} lists each minimal edge path as it appears in Figure~\ref{fig:edgepaths}, 
the branch pattern of the induced branched surface (i.e.\ the sequence of edge labels), 
and the Euler characteristic of the carried surface corresponding to weights $\alpha\geq \beta$.  
It also lists for each of these paths, when $\alpha \geq \beta \geq 0$, the boundary slopes of the carried surfaces relative to the canonical meridian-longitude framings of the two unknot components of the two-bridge links and the count of the number of  boundary components on each link component.  
These are also calculated from the given preliminary data of algebraic intersections of the boundary components with the meridians and blackboard framed longitudes and the boundary slopes in terms of the blackboard framing; 
refer to \cite{HS} for details. 
Note that when $\alpha > \beta=0$, the associated essential surface meets the second component in meridians except for path $\gamma_7$ where it is disjoint from the second link component. (This may be observed from the appearance of at least one $D$ block in the branch patterns for each of the paths except $\gamma_7$.) 
When $\alpha < \beta$ we may continue to use the table with $\alpha$ and $\beta$ swapped and with the two link components swapped.

\bigskip

\begin{table*}
	\centering
	\caption{\small First are 
		the minimal edge paths in $D_t$ from $1/0$ to $5/14$  as presented in Figure~\ref{fig:edgepaths} and
		the branch patterns of their associated branched surfaces.  
		Then, for $\alpha > \beta \geq 0$ so that $t = \alpha/\beta \in (1,\infty]$,  
		the Euler characteristics, boundary slopes, and number of boundary components for the carried surfaces are shown.
		The  boundary slopes and number of boundary components of the first and second link component are presented as pairs.  For $t \in [0,1)$, apply the homeomorphism of the two-bridge link that interchanges its two components, i.e.\ exchange coordinates and swap $\alpha$ and $\beta$; see Figure~\ref{2_1_4_t><1}.  For $\alpha = \beta$ so that $t=1$, the seven paths limit to five distinct paths in $D_1$ ($\gamma_1$, $\gamma_2=\gamma_3$, $\gamma_4$, $\gamma_5=\gamma_6$, and $\gamma_7$). The table continues to apply in this case.
	}
	\label{table:5/14paths} 
	\begin{sideways}
	\begin{tabular}{@{}M{20mm}M{30mm}M{25mm}M{20mm}M{30mm}M{35mm}@{}}
		\toprule
		edge path    & path picture & branch pattern  & $\chi$ & boundary slopes & number of boundary components \\
		\midrule
		$\gamma_1$  & \includegraphics[width=1in]{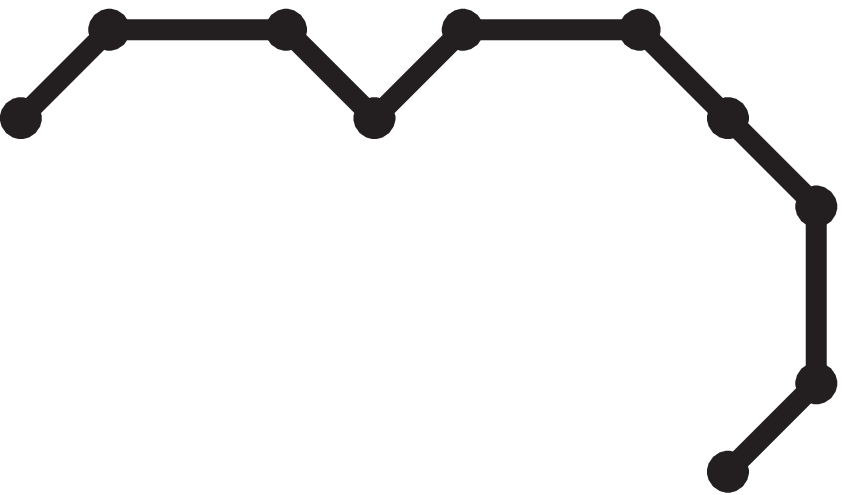} & $ADAADAADA$ & $-2\alpha - 2\beta$ 
			& $(3 \tfrac{\beta}{\alpha}, 3\tfrac{\alpha}{\beta})$ & $(\gcd(3\beta,\alpha), \gcd(3\alpha,\beta))$ \\ \addlinespace[0.5em]
		$\gamma_2$  & \includegraphics[width=1in]{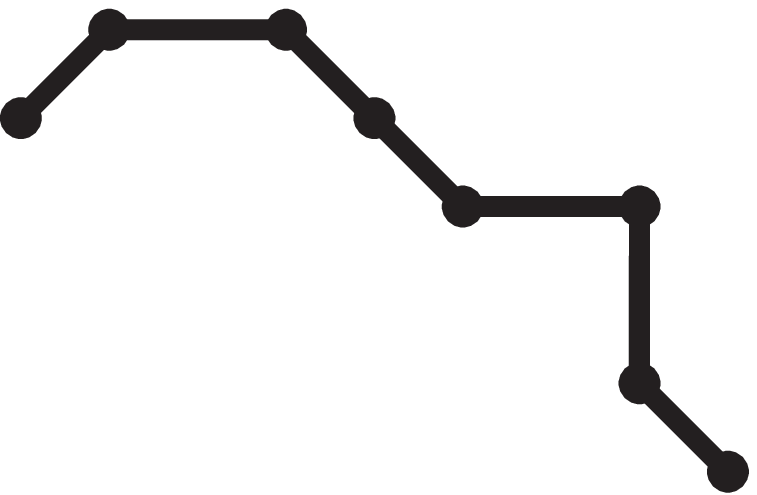} & $ADAADDA$ & $-2\alpha$ 
			& $(-\tfrac{\beta}{\alpha},-\tfrac{\alpha}{\beta})$ & $(\gcd(\beta,\alpha), \gcd(\alpha,\beta))$ \\ \addlinespace[0.5em]
		$\gamma_3$  & \includegraphics[width=1in]{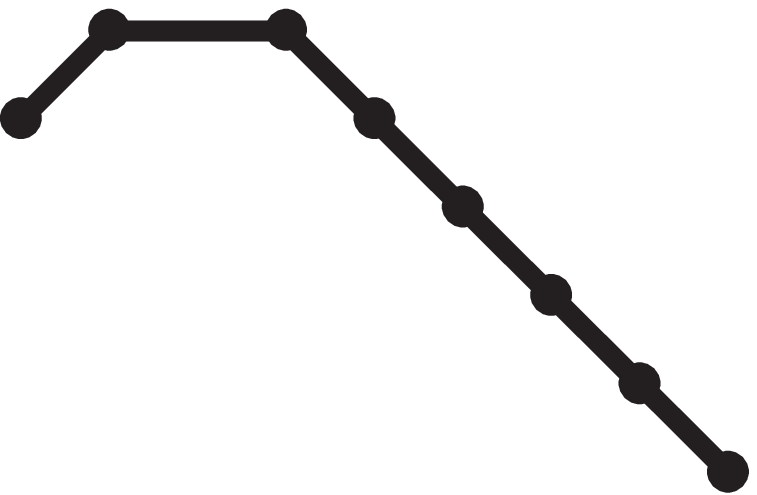} & $ADAABBA$ & $-\alpha - \beta$ 
			& $(\tfrac{\beta}{\alpha}-2,\tfrac{\alpha}{\beta}-2)$ & $(\gcd(\beta,\alpha), \gcd(\alpha,\beta))$\\ \addlinespace[0.5em]
		$\gamma_4$  & \includegraphics[width=1in]{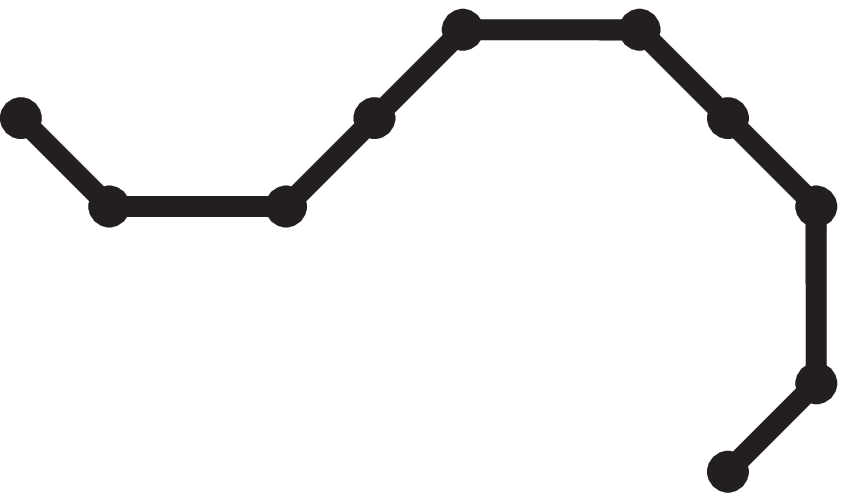} & $ADAADAADA$  & $-2\alpha -2\beta$ 
			& $(\tfrac{\beta}{\alpha}, \tfrac{\alpha}{\beta})$ & $(\gcd(\beta,\alpha), \gcd(\alpha,\beta))$\\ \addlinespace[0.5em]
		$\gamma_5$ &  \includegraphics[width=1in]{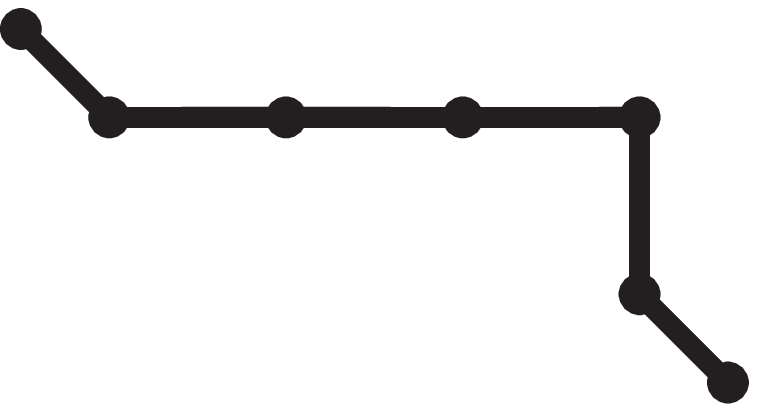} & $ADCDDA$ & $-2\alpha+\beta$ 
			& $(-3 \tfrac{\beta}{\alpha}, -3\tfrac{\alpha}{\beta}-2)$ & $(\gcd(3\beta,\alpha), \gcd(3\alpha,\beta))$ \\ \addlinespace[0.5em]
		$\gamma_6$ &  \includegraphics[width=1in]{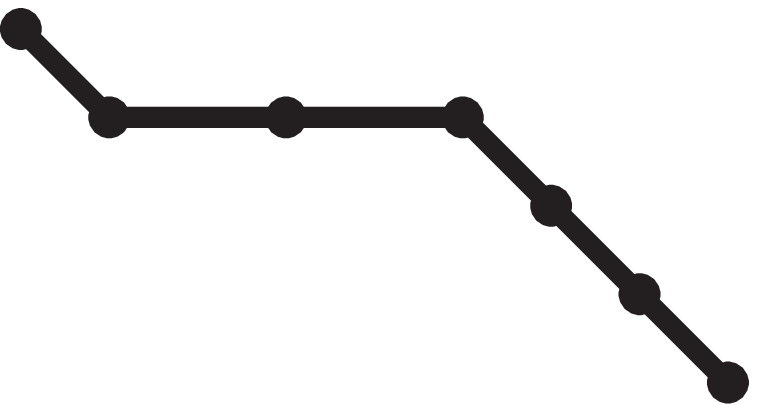} & $ADCBBA$ & $-\alpha$ 
			& $(-\tfrac{\beta}{\alpha}-2,-\tfrac{\alpha}{\beta}-4)$ & $(\gcd(\beta,\alpha), \gcd(\alpha,\beta))$\\ \addlinespace[0.5em]
		$\gamma_7$ &  \includegraphics[width=1in]{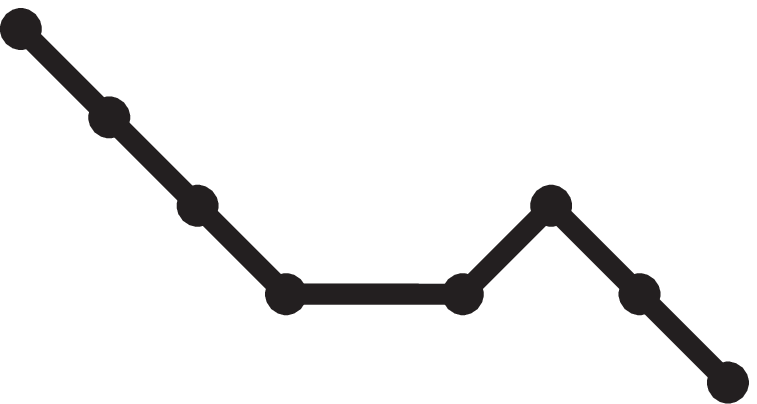} & $ABBCBBA$ & $-\alpha$ 
			& $(-5, -3)$  & $(\alpha, \beta)$\\ \addlinespace[0.5em]
		\bottomrule
	\end{tabular}
\end{sideways}
\end{table*}

\newpage 

\section{The Slope Conjecture for Mazur doubles of knots}
\label{sec:slope_conjecture_M}

In this section we prove the Slope Conjecture for Mazur doubles of knots, Theorem~\ref{thm:SC_and_SSC_M}(1). 

For a given knot $K$, 
recall that $N_K$ the smallest nonnegative integer such that 
$d_+[J_{K,n}(q)]$ is a quadratic quasi-polynomial 
$\delta_K(n)=a(n) n^2 +b(n) n +c(n)$ 
for $n \ge 2N_K+1$. 
If the period of $\delta_K(n)$ is less than or equal to $2$, 
then define $i \in \{1,2\}$ by $i\equiv n \pmod 2$ and 
put $a_i:=a (2N_K + i)$, $b_i:=b (2N_K + i)$.   

\begin{thm_slope_conjecture_M}
Let $K$ be a knot. 
We assume that the period of $\delta_K(n)$ is less than or equal to $2$. 
Assume that $b_i \le 0$, and that $b_i = 0$ implies $a_i \ne 0$.

 If $K$ satisfies the Slope Conjecture, 
then its Mazur double $M(K)$ also satisfies the Slope Conjecture. 
\end{thm_slope_conjecture_M}

\begin{proof}
Let $K$ be a knot such that $\delta_K(n)=a(n) n^2 +b(n) n+c(n)$is quadratic quasi-polynomial of period $\le 2$
with $b_i \le 0$.  
Then Proposition~\ref{maxdeg_Mazur} shows that the Jones slopes of $M(K)$ satisfy 
\[js_{M(K)} \subset \{ 9 js(K) , js(K)\}.\]
Let us find essential surfaces in $E(M(K))$ whose boundary slopes are these Jones slopes. 

Recall that $k_1 \cup k_2$ is a $2$--bridge link expressed as $[2, 1, 4]$ depicted in Figure~\ref{2_1_4_t><1} (Left). 
Let $(\mu_i, \lambda_i)$ be a preferred meridian-longitude pair of $k_i$. 
As in Figure~\ref{2_1_4_t><1} (Upper Right) we take a solid torus $V = S^3 - \mathrm{int}N(k_1)$ which contains $k_2$ in its interior; 
let $(\mu_V, \lambda_V)$ be the standard meridian-longitude pair of $V$, 
which satisfies $(\mu_V, \lambda_V) = (\lambda_1, \mu_1)$. 
Switching $k_1$ and $k_2$, 
we may also take a solid torus $V = S^3 - \mathrm{int}N(k_2)$ which contains $k_1$ in its interior (Lower Right) so that $(\mu_V, \lambda_V) = (\lambda_1, \mu_1)$. 
Since there is an isotopy of $S^3$ deforming $k_1$ to $k_2$, 
patterns $(V, k_2)$ is pairwise homeomorphic to $(V, k_1)$; 
see Figure~\ref{2_1_4_t><1} (Right). 

\begin{figure}[!ht]
\includegraphics[width=0.7\linewidth]{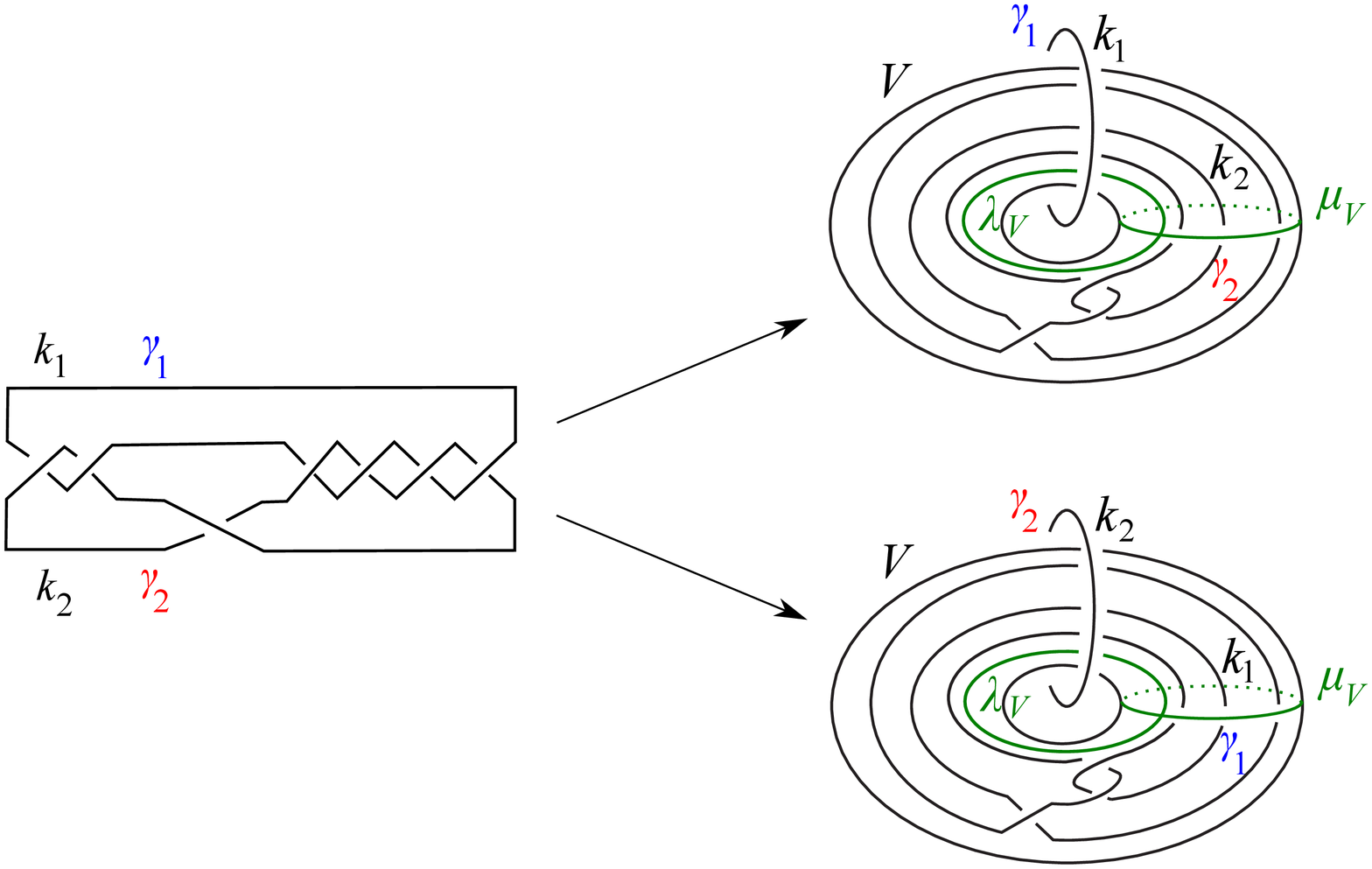}  
\caption{$k_1 \cup k_2$ is a two bridge link $[2, 1, 4]$; we can switch $k_1$ 
(with slope $\gamma_1$) and $k_2$ (with slope $\gamma_2$).}
\label{2_1_4_t><1}
\end{figure}

Let $f$ be a faithful embedding $f \colon V \to S^3$ which sends  the core of $V$ to $K$ and 
$f(\mu_V) = \mu_K$ and $f(\lambda_V) = \lambda_K$. 
Then $f(k_j) = M(K)$, 
the Mazur double of $K$.  
Thus the exterior 
$E(M(K))$ is the union of $E(K)$ and $X = f(V - \mathrm{int}N(k_j))$. 
$\partial X$ consists of two tori $T_M = \partial N(M(K))$ and $T_K = f(\partial V) = \partial E(K)$. 
Then $(f(\mu_j), f(\lambda_j))$ is a preferred meridian-longitude pair $(\mu_M, \lambda_M)$ of $M(K)$ ($j = 1, 2$). 

\medskip 

According to Proposition~\ref{maxdeg_Mazur} 
we divide into two cases depending upon $a_i > 0$ or $a_i \le 0$. 

\medskip

\noindent
\textbf{Case 1.\ $a_i > 0$.}\quad 
Since $K$ satisfies the Slope Conjecture, 
the Jones slope $4a_i$ is realized by a boundary slope of an essential surface $S_{K,i} \subset E(K)$.  

\begin{claim}
\label{F_i_case1}
There exists an essential surface $F_i$ in $V - \mathrm{int}N(k_2)$ 
\(resp.\ $V - \mathrm{int}N(k_1)$\) 
such that each component of $F_i \cap \partial V$ has slope $4a_i$ and 
each component of $F_i \cap \partial N(k_2)$ 
\(resp.\ $F_i \cap \partial N(k_2)$\) has $36a_i$ with $a_i \ge \frac{1}{12}$ \(resp.\ $0 < a_i \le \frac{1}{12}$\). 
\end{claim}

\begin{proof}
Let us take an essential surface $F_{\gamma_1}$ in 
$S^3 - \mathrm{int}N(k_1 \cup k_2)$ 
described in Section~\ref{essential_surfaces}. 
Then it has a pair of boundary slopes 
$(3\frac{\beta}{\alpha}, 3\frac{\alpha}{\beta})$ on $\partial N(k_1),\  \partial N(k_2)$, 
i.e. $F_{\gamma_1}$ has boundary slopes $\frac{3\beta}{\alpha}$ on $\partial N(k_1)$ and 
$\frac{3\alpha}{\beta}$ on $\partial N(k_2)$.  
 
First assume that $a_i \ge \frac{1}{12}$. 
Then regard $S^3 - \mathrm{int}N(k_1 \cup k_2)$ as $V - \mathrm{int}N(k_2)$. 
Using the preferred meridian-longitude $(\mu_V, \lambda_V)$ instead of $(\mu_1, \lambda_1)$, 
$F_{\gamma_1} \cap \partial V$ has slope $\frac{\alpha}{3\beta}$. 
Choose $\alpha, \beta$ ($\alpha \ge \beta$) so that 
$\frac{\alpha}{3\beta} = 4a_i$, 
i.e. $\frac{\alpha}{\beta} = 12a_i \ge 1$. 
Let us denote $F_{\gamma_1}$ associated with $a_i$ by $F_i$.  
Then each component of $F_i \cap \partial V$ has slope $4a_i$,  
and each component of $F_i \cap \partial N(k_2)$ has slope $36a_i$ as desired. 

Next assume that $0 < a_i \le \frac{1}{12}$. 
Regard $S^3 - \mathrm{int}N(k_1 \cup k_2)$ as $V - \mathrm{int}N(k_1)$ by exchanging $k_1$ and $k_2$. 
Using the preferred meridian-longitude $(\mu_V, \lambda_V)$ instead of $(\mu_2, \lambda_2)$, 
$F_{\gamma_1} \cap \partial V$ has slope $\frac{\beta}{3\alpha}$. 
Choose $\alpha, \beta$ so that 
$\frac{\beta}{3\alpha} = 4a_i$
i.e. $\frac{\beta}{\alpha} = 12a_i \le 1$. 
Take $F_i$ as above so that 
each component of $F_i \cap \partial V$ has slope $4a_i$,  
and each component of $F_i \cap \partial N(k_1)$ has slope $36a_i$ as desired. 
\end{proof}

\medskip

Let us take the image $f(F_i)$ in $X = f(V - \mathrm{int}N(k_2))$, and denote it by $S_i$.  
By construction $S_i$ is essential in $X$ 
and each component of $S_i \cap T_K$ has slope $4a_i$ and each component of 
$S_i \cap T_M$ has slope $36a_i$. \par

To build a required essential surface $\tilde{S}_i \subset E(M(K))$ we take $m$ parallel copies 
$m S_i$ of the essential surface $S_i$ and $n$ parallel copies $n S_{K,i}$ of the essential surface $S_{K,i}$, 
and then glue them along their boundaries to obtain a connected surface 
$\tilde{S}_i = m S_i \cup n S_{K,i}$ in $E(M(K))$.  
Even when both $S_i$ and $S_{K,i}$ are orientable, 
$\tilde{S}_i$ may not be orientable. 
If $\tilde{S}_i$ is non-orientable, 
then consider a regular neighborhood of $\tilde{S}_i$ in $E(M(K))$, 
which is a twisted $I$--bundle of $\tilde{S}_i$ whose $\partial I$--subbundle 
is an orientable double cover of $\tilde{S}_i$. 
We use the same symbol $\tilde{S}_i$ to denote this $\partial I$--subbundle. 
Note that $S_{K,i}$ and $S_i$ are orientable,  
so $\tilde{S}_i \cap E(K)$ consists of parallel copies of $S_{K,i}$ and similarly 
$\tilde{S}_i \cap X$ consists of parallel copies of $S_i$.  
Since $\partial E(K)$ is incompressible in $E(M(K))$ and 
$S_i$, $S_{K,i}$ are essential in $X$, $E(K)$, respectively, 
$\tilde{S}_i$ is incompressible in $E(M(K))$.  
If $\tilde{S}_i$ were boundary-compressible, 
then a component of $\tilde{S}_i$ would be a boundary-parallel annulus.
However, obviously each component of $\tilde{S}_i$ is not an annulus, and we have a contradiction. 
Hence $\tilde{S}_i$ is the desired essential surface.

\medskip

\noindent
\textbf{Case 2. $a_i \le 0$.}\quad 
We follow the same argument in Case 1. 
Take the surface $S_{K,i}$ as in Case 1. 

\begin{claim}
\label{F_i_case2}
There exists an essential orientable surface $F_i$ in $V - \mathrm{int}N(k_2)$
\(resp.\  $V - \mathrm{int}N(k_1)$\) 
such that each component of $F_i \cap \partial V$ has slope $4a_i$ and 
each component of $F_i \cap \partial N(k_2)$ 
\(resp.\ $F_i \cap \partial N(k_2)$\) has $4a_i$ with $a_i \le -\frac{1}{4}$ 
\(resp.\ $- \frac{1}{4} \le a_i \le 0$\). 
\end{claim}

\begin{proof}
Let us take an essential surface $F_{\gamma_2}$ in $S^3 - \mathrm{int}N(k_1 \cup k_2)$ 
described in Section~\ref{essential_surfaces}.  
Then it has a pair of boundary slopes 
$(-\frac{\beta}{\alpha}, -\frac{\alpha}{\beta})$ on $\partial N(k_1),\  \partial N(k_2)$, 
i.e. $F_{\gamma_2}$ has boundary slopes $-\frac{\beta}{\alpha}$ on $\partial N(k_1)$ and 
$-\frac{\alpha}{\beta}$ on $\partial N(k_2)$.  

Assume that $a_i \le -\frac{1}{4}$, i.e.\ $-4a_i \ge 1$. 
Then regard $S^3 - \mathrm{int}N(k_1 \cup k_2)$ as $V - \mathrm{int}N(k_2)$. 
Using the preferred meridian-longitude $(\mu_V, \lambda_V)$ instead of $(\mu_1, \lambda_1)$, 
$F_{\gamma_2} \cap \partial V$ has slope $-\frac{\alpha}{\beta}$. 
Choose $\alpha, \beta$ so that 
$-\frac{\alpha}{\beta} = 4a_i \le -1$, 
i.e. $\frac{\alpha}{\beta} = -4a_i \ge 1$. 
Let us denote $F_{\gamma_2}$ associated with $a_i$ by $F_i$.  
Then each component of $F_i \cap \partial V$ has slope $4a_i$,  
and each component of $F_i \cap \partial N(k_2)$ has slope $4a_i$ as desired. 

When $- \frac{1}{4} \le a_i \le 0$, i.e. $0 \le -4a_i \le 1$, 
regard $S^3 - \mathrm{int}N(k_1 \cup k_2)$ as $V - \mathrm{int}N(k_1)$ by exchanging $k_1$ and $k_2$. 
Using the preferred meridian-longitude $(\mu_V, \lambda_V)$ instead of $(\mu_2, \lambda_2)$, 
$F_{\gamma_2} \cap \partial V$ has slope $-\frac{\beta}{\alpha}$. 
Choose $\alpha, \beta$ so that 
$-\frac{\beta}{\alpha} = 4a_i$
i.e. $\frac{\beta}{\alpha} = -4a_i \le 1$. 
Take $F_i$ as above so that 
each component of $F_i \cap \partial V$ has slope $4a_i$,  
and each component of $F_i \cap \partial N(k_1)$ has slope $4a_i$ as desired.  
\end{proof}

Let us take the image $f(F_i)$ in $X = f(V - \mathrm{int}N(k_2))$ if $a_i \le -\frac{1}{4}$ 
(resp.\ $X = f(V - \mathrm{int}N(k_1))$ if $-\frac{1}{4} \le a_i \le 0$), and denote it by $S_i$.  
As in the proof of Case 1, 
we obtain the desired essential surface 
$\tilde{S}_i = m S_i \cup n S_{K,i}$ or the $\partial I$--subbundle of 
a twisted $I$--bundle of $\tilde{S}_i$ (when $\tilde{S}_i $ is non-orientable). 
\end{proof}

\medskip

\section{Strong Slope Conjecture for Mazur doubles of knots}
\label{strong Slope_M}

Theorem~\ref{thm:SC_and_SSC_M}(2) follows from Theorem~\ref{strong_slope_conjecture_SS(i)} below. 

Recall that for a quadratic quasi-polynomial $\delta_K(n) = a(n)n^2 + b(n) n + c(n)$ with period $\le 2$, 
we put $a_i = a(i),\ b_i = b(i)$ and $c_i= c(i)$, where $i = 1, 2$. 

\begin{theorem}
\label{strong_slope_conjecture_SS(i)}
Let $K$ be a knot. 
We assume that the period of $\delta_K(n)$ is less than or equal to $2$. 
Assume that $b_i \le 0$, and that $b_i = 0$ implies $a_i \ne 0$.

 Further assume that $-\frac{1}{4} \le a_i$.  
If $K$ satisfies the Strong Slope Conjecture  with $SS(i)$,  
then its Mazur double also satisfies the Strong Slope Conjecture with $SS(i)$. 
\end{theorem}

\begin{proof}
Write $\delta_{M(K)}(n)=a_M(n)n^2 + b_M(n)n + c_M(n)$.   
As in Proposition~\ref{maxdeg_Mazur}, defining $i \in \{0,1\}$ so that $i \equiv n \pmod 2$,
we may write 
$a_{M,i} = a_M(i), b_{M,i} = b_M(i),  c_{M,i} = c_M(i)$.
Write $a_i= r_i/s_i$ where $r_i$ and $s_i$ are coprime integers and $r_i,\ s_i>0$.   
Then, as a ratio of coprime integers, the denominator of $4a_i$ is $s_i/\gcd(4,s_i)$.
\medskip

Assume that $-\frac{1}{4} \le a_i$.  
Since $K$ satisfies the Strong Slope Conjecture with $SS(i)$, 
there is a properly embedded essential surface $S_{K,i}$ 
in the exterior of $K$ whose boundary slope is $4a_i$ and 
\[\frac{\chi(S_{K,i})}{|\bdry S_{K,i}| \cdot \frac{s_i}{\gcd(4,s_i)}} = 2 b_i. \]

We show that an essential surface $\tilde{S}_i$ in $E(M(K))$ given in the proof of Theorem~\ref{thm:SC_and_SSC_M}(1)  
satisfies the condition of the Strong Slope Conjecture: 
\[\tilde{S}_i\ \textrm{has boundary slope}\ p/q = 4a_{M,i} \quad 
\mbox{and} \quad \frac{\chi(\tilde{S}_i)}{|\partial \tilde{S}_i| q} = 2b_{M,i}.\]

\smallskip

We separate the argument into two cases: 
$a_i > 0$ and $-\frac{1}{4} \le a_i \le 0$. 

\medskip

\noindent
\textbf{Case 1.\ $a_i > 0$.}\quad
When addressing the Slope Conjecture for $M(K)$ in this case, 
we constructed a properly embedded essential surface $\tilde{S}_i = m S_{K,i} \cup n S_i$ 
in the exterior of $M(K)$ 
by joining $m$ copies of $S_{K,i}$ in $E(K)$ to $n$ copies of the surface $S_i$  in 
$f(V) - N(M(K))$.  
This requires that 
\[ m |\bdry S_{K,i}| = n |\bdry S_i \cap T_K|. \]

As in the proof of Theorem~\ref{thm:SC_and_SSC_M}(1) (Case 1), 
we consider two subcases depending upon $a_i \ge \frac{1}{12}$ or $0 < a_i \le \frac{1}{12}$.

\medskip

\noindent
\textbf{(i)} $a_i \ge \frac{1}{12}$.\quad  
In this case we regard $S^3 - \mathrm{int}N(k_1 \cup k_2)$ as $V - \mathrm{int}N(k_2)$. 
Recall that the surface $S_i$ is identified with a surface of type $F_{\gamma_1}$ in the exterior of the $[2, 1, 4]$ two-bridge link by the embedding $f \colon V \to S^3$. 
$F_{\gamma_1} \cap \partial V$ has slope $\frac{\alpha}{3\beta}$ with respect to the preferred meridian-longitude $(\mu_V, \lambda_V)$. 
Choose $\alpha, \beta$ ($\alpha \ge \beta$) as 
$\frac{\alpha}{\beta} = 12a_i = \frac{12r_i}{s_i} \ge 1$ 
so that $S_i$ has boundary slope $4a_i$ on $\partial V = T_K$.  
Note that $S_i$ has boundary slope $3\frac{\alpha}{\beta} = 36a_i$ on $\partial N(M(K)) = T_M$. 
For convenience, we choose $\beta = 2s_i$, $\alpha = 24r_i$ 
so that $F_{\gamma_1} = F_{\gamma_1, \alpha, \beta}$ is orientable; see Subsection~\ref{algorithm}. 
 
Then, using Table~\ref{table:5/14paths},
we calculate the following:
\begin{itemize}
\item $\chi(S_i) 
= -2\alpha -2\beta = -4(12r_i+s_i)$,

\item  slope of $\bdry S_i$  on $T_M$ is 
$3\frac{\alpha}{\beta} 
= 3 \frac {12r_i}{s_i} 
= \frac {36r_i}{s_i}$,  

\item $|\bdry S_i \cap T_K| 
= \gcd(3\beta, \alpha) 
= \gcd(6s_i,24 r_i) 
= 6 \gcd(4,s_i)$, and

\item $|\bdry S_i \cap T_M| 
= \gcd(3\alpha, \beta) 
= \gcd(72 r_i, 2s_i) 
= 2\gcd(36,s_i)$.
\end{itemize}

The boundary of $\tilde{S}_i$ consists of $n$ copies of the boundary of $S_i$  on $T_M$, 
so we have

\begin{itemize}
\item 
$|\bdry \tilde{S}_i| = n | \partial S_i \cap T_M | = 2n \gcd(36,s_i)$. 
\end{itemize}

Moreover, the boundary slope of $\tilde{S}_i$ is the  slope of $\bdry S_i$  on $T_M$, 
and so this has denominator $ \frac{s_i}{\gcd(36,s_i)}$.  
We may now calculate
\begin{align*}
  \frac{\chi(\tilde{S}_i)}{|\bdry \tilde{S}_i| \cdot \frac{s_i}{\gcd(36,s)}} 
   &= \frac{m \chi(S_{K,i}) + n \chi(S_i)}{ 2n \gcd(36,s_i) \cdot \frac{s_i}{\gcd(36,s_i)}} \\
  &= \frac{ 2 b_i m |\bdry S_{K,i}| \cdot \frac{s_i}{\gcd(4,s_i)} -4n(12 r_i+s_i)}{2n s_i} \\
  &= \frac{ 12 b_i n \gcd(4,s_i) \cdot \frac{s_i}{\gcd(4,s_i)} -4n(12 r_i+s_i)}{2n s_i} \\
  &= 6 b_i -24r_i/s_i-2 = 2(-12a_i + 3b_i-1) = 2b_{M,i} 
\end{align*}
    as desired.

If the glued surface $\tilde{S}_i = m S_{K,i} \cup n S_i$ is non-orientable, 
then as in the proof of Theorem~\ref{thm:SC_and_SSC_M}(1), 
we replace $\tilde{S}_i$ by the frontier $\tilde{S}_i'$ of the tubular neighborhood of $\tilde{S}_i$, 
but \cite[Lemma 5.1]{BMT_tgW} shows that $\tilde{S}_i'$ and $\tilde{S}_i$ has the same boundary slope and 
$\displaystyle \frac{\chi(\tilde{S}_i')}{|\bdry\tilde{S}_i'| \cdot \frac{s_i}{\gcd(36,s_i)}} 
= \frac{\chi(\tilde{S}_i)}{|\bdry \tilde{S}_i| \cdot \frac{s}{\gcd(36,s_i)}}$. 
Thus the essential surface $\tilde{S}_i$ or $\tilde{S}_i'$ (when $\tilde{S}_i$ is non-orientable) is the desired 
essential surface.

\medskip

\noindent
\textbf{(ii)} $0 < a_i \le \frac{1}{12}$.\quad  
In this case we regard $S^3 - \mathrm{int}N(k_1 \cup k_2)$ as $V - \mathrm{int}N(k_1)$. 
$F_{\gamma_1} \cap \partial V$ has slope $\frac{\beta}{3\alpha}$ with respect to 
the preferred meridian-longitude $(\mu_V, \lambda_V)$. 
Choose $\alpha, \beta$ ($\alpha \ge \beta$) as 
$\frac{\beta}{\alpha} = 12a_i = \frac{12r_i}{s_i} \le 1$ 
so that $S_i$ has boundary slope $4a_i$ on $\bdry V$.  
Note that $S_i$ has boundary slope $3\frac{\beta}{\alpha} = 36a_i$ on $\partial N(M(K))$. 
For convenience, we choose $\alpha = 2s_i$, $\beta = 24r_i$ 
so that $F_{\gamma_1} = F_{\gamma_1, \alpha, \beta}$ is orientable; see Subsection~\ref{algorithm}. 
 
Then, using Table~\ref{table:5/14paths},
we calculate the following:
\begin{itemize}
\item $\chi(S_i) 
= -2\alpha -2\beta = -4(12r_i+s_i)$,

\item  slope of $\bdry S_i$  on $T_M$ is 
$3\frac{\beta}{\alpha} 
= 3 \frac {12r_i}{s_i} 
= \frac {36r_i}{s_i}$,  

\item $|\bdry S_i \cap T_K| 
= \gcd(3\alpha, \beta) 
= \gcd(6s_i,24 r_i) 
= 6 \gcd(4,s_i)$, and

\item $|\bdry S_i \cap T_M| 
= \gcd(3\beta, \alpha) 
= \gcd(72 r_i, 2s_i) 
= 2\gcd(36,s_i)$.
\end{itemize}

The boundary of $\tilde{S}_i$ consists of $n$ copies of the boundary of $S_i$  on $T_M$, 
so we have

\begin{itemize}
\item 
$|\bdry \tilde{S}_i| = n | \partial S_i \cap T_M | = 2n \gcd(36,s_i)$. 
\end{itemize}

Then the identical argument as (i) shows that the essential surface $\tilde{S}_i$ or $\tilde{S}_i'$ (when $\tilde{S}_i$ is non-orientable) is the desired 
essential surface.

\medskip

\noindent
\textbf{Case 2.\ $-\frac{1}{4} \le a_i \le 0$.}\quad 
Write $a_i= r_i/s_i$ where $r_i$ and $s_i$ are coprime integers and $s_i>0$.   
Then, as a ratio of coprime integers, the denominator of $4a_i$ is $s_i/\gcd(4,s_i)$. 
We take the surface $S_{K,i}$ in the proof of Case 1 
whose boundary slope is $4a_i$ and 
\[\frac{\chi(S_{K,i})}{|\bdry S_{K,i}| \cdot \frac{s_i}{\gcd(4,s_i)}} = 2 b_i. \]
When addressing the Slope Conjecture for $M(K)$ in this case, 
we constructed a properly embedded essential surface $\tilde{S}_i = m S_{K,i} \cup n S_i$ 
in the exterior of $M(K)$ 
by joining $m$ copies of $S_{K,i}$ in $E(K)$ to $n$ copies of the surface $S_i$  in 
$f(V) - N(M(K))$.  
This requires that 
\[ m |\bdry S_{K,i}| = n |\bdry S_i \cap T_K|. \]

Let us regard $S^3 - \mathrm{int}N(k_1 \cup k_2)$ as $V - \mathrm{int}N(k_1)$.

Recall that the surface $S_i$ is identified with a surface of type $F_{\gamma_2}$ in the exterior of the $[2, 1, 4]$ two-bridge link by the embedding $f \colon V \to S^3$.

$F_{\gamma_2} \cap \partial V$ has slope $-\frac{\beta}{\alpha}$ with respect to the preferred meridian-longitude $(\mu_V, \lambda_V)$. 
Choose $\alpha, \beta$ ($\alpha \ge \beta$) as 
$\frac{\beta}{\alpha} = -4a_i = -\frac{4r_i}{s_i} \le 1$ 
so that $S_i$ has boundary slope $4a_i$ on $\partial V = T_K$.  
Note that $S_i$ has boundary slope $-\frac{\beta}{\alpha} = 4a_i$ on $\partial N(M(K)) = \partial T_M$. 
For convenience, we choose $\alpha = 2s_i$, $\beta = -8r_i$ 
so that $F_{\gamma_2} = F_{\gamma_2, \alpha, \beta}$ is orientable; see Subsection~\ref{algorithm}.

Then, 
Table~\ref{table:5/14paths} gives the following:
\begin{itemize}
\item $\chi(S_i) 
=  -2\alpha = -4s_i$,

\item  slope of $\partial S_i$  on $T_M$ is 
$-\frac{\beta}{\alpha}  
= \frac {4r_i}{s_i}$,  

\item $|\partial S_i \cap T_K| 
= \gcd(\alpha, \beta) 
= \gcd(2s_i,-8 r_i) 
= 2 \gcd(s_i, 4)$, and

\item $|\partial S_i \cap T_M| 
= \gcd(\beta, \alpha) 
= 2\gcd(4,s_i)$.
\end{itemize}

The boundary of $\tilde{S}_i$ consists of $n$ copies of the boundary of $S_i$  on $T_M$, 
so we have

\begin{itemize}
\item 
$|\bdry \tilde{S}_i| = n | \partial S_i \cap T_M | = 2n \gcd(4,s_i)$. 
\end{itemize}

Moreover, the boundary slope of $\tilde{S}_i$ is the  slope of $\bdry S_i$  on $T_M$, 
and so this has denominator $\frac{s_i}{\gcd(4r_i, s_i)} = \frac{s_i}{\gcd(4, s_i)}$.  
We may now calculate
\begin{align*}
  \frac{\chi(\tilde{S}_i)}{|\bdry \tilde{S}_i| \cdot \frac{s_i}{\gcd(4,s)}} 
   &= \frac{m \chi(S_{K,i}) + n \chi(S_i)}{ 2n \gcd(4,s_i) \cdot \frac{s_i}{\gcd(4,s_i)}} \\
  &= \frac{ 2 b_i m |\bdry S_{K,i}| \cdot \frac{s_i}{\gcd(4,s_i)} -4n s_i}{2ns_i} \\
  &= \frac{ 4 b_i n \gcd(4,s_i) \cdot \frac{s_i}{\gcd(4,s_i)} }{2n s_i} -2\\
  &= 2 b_i -2 = 2(b_i-1) = 2b_{M,i} 
\end{align*}
    as desired.

If the glued surface $\tilde{S}_i = m S_{K,i} \cup n S_i$ is non-orientable, 
then as in the proof of Theorem~\ref{thm:SC_and_SSC_M}(1), 
we replace $\tilde{S}_i$ by the frontier $\tilde{S}_i'$ 
of the tubular neighborhood of $\tilde{S}_i$, as in the proof of Case 1, 
 the essential surface $\tilde{S}_i$ or $\tilde{S}_i'$ (when $\tilde{S}_i$ is non-orientable) is the desired 
essential surface.
\end{proof}

As shown in Theorem~\ref{thm:SC_and_SSC_M}(1),  
even when $a_i < -\frac{1}{4}$, 
$M(K)$ satisfies the Slope Conjecture if $K$ does.
However, concerning the Strong Slope Conjecture, 
we have an unexpected situation in such a case.

\begin{addendum}
\label{add:mazurSSCfailure}
Let $K$ be a knot. 
We assume that the period of $\delta_K(n)$ is less than or equal to $2$.  
Assume that $b_i \le 0$, and if $b_i = 0$, then $a_i \ne 0$.  Further assume that $a_i < -\frac{1}{4}$.

If $K$ satisfies the Strong Slope Conjecture with $SS(i)$,  
then either its Mazur double fails to satisfy the Strong Slope Conjecture with $SS(i)$
or any Jones surface for $M(K)$ fails to restrict to the exterior of $K$ as a Jones surface for $K$.
\end{addendum}

\begin{proof}

Assume that $a_i < - \frac{1}{4}$. 
Write $a_i = \frac{r_i}{s_i}$ as before. 
By Proposition~\ref{maxdeg_Mazur}, for each $i\in \{0,1\}$, $4a_i$ is the Jones slope of both $K$ and $M(K)$.  (More specifically, the set of Jones slopes for both $K$ and $M(K)$ is $\{4a_0, 4a_1\}$.)

If $\tilde{S}_i$ is a properly embedded essential surface in the exterior of $M(K)$, it may be isotoped to intersect the satellite torus $\bdry V$ minimally.  This torus $\bdry V$ divides the exterior fo $M(K)$ into the exterior of $K$ and the exterior of the Mazur pattern in the solid torus $V$.  Hence it divides the surface $\tilde{S}_i$ into essential surfaces $\tilde{S}_{K,i}$ and $\tilde{S}_{V,i}$ in the exterior of $K$ and the exterior of the Mazur pattern, respectively.

So assume that $\tilde{S}_i$ is a Jones surface for $M(K)$ with boundary slope $4a_i$.
We will proceed as in the proof of Theorem~\ref{strong_slope_conjecture_SS(i)} to show that if $\tilde{S}_{K,i}$, the restriction of $\tilde{S}_i$ to the exterior of $K$, is a union of Jones surfaces for $K$, then we must have $a_i=-\frac14$, contrary to assumption.

\smallskip
Suppose that $\tilde{S}_{K,i}$ is the union of Jones surfaces for $K$.  Then a single component $S_{K,i}$ has boundary slope $4a_i = \frac{4r_i}{s_i}$ on $\bdry V$ and 
\[\frac{\chi(S_{K,i})}{|\bdry S_{K,i}| \cdot \frac{s_i}{\gcd(4,s_i)}} = 2 b_i. \]
Since this equality holds for all components of $\tilde{S}_{K,i}$ and any number of parallel copies of these components, we may assume that $\tilde{S}_{K,i} = m S_{K,i}$, the union of some number $m$ parallel copies of $S_{K,i}$ for some number $m$.  (If needed, we may first replace $\tilde{S}_i$, and hence $\tilde{S}_{K_i}$, with a collection of parallel copies.)

Recall from the construction in Case 2 of the proof of Theorem~\ref{strong_slope_conjecture_SS(i)} that 
we need to regard $S^3 - \mathrm{int}N(k_1 \cup k_2)$ as $V - \mathrm{int}N(k_2)$ where $k_1 \cup k_2$ is the two bridge link $[2,1,4]$.  Then $\tilde{S}_{V,i} \cap \partial V$ has slope $-\frac{\alpha}{\beta}$ with respect to the preferred meridian-longitude $(\mu_V, \lambda_V)$ so that 
$\frac{\alpha}{\beta} = -4a_i = -\frac{4r_i}{s_i} > 1$. 
Therefore $\tilde{S}_{V,i}$ has boundary slope $4a_i$ on $\partial V = T_K$ and  
boundary slope $-\frac{\alpha}{\beta} = 4a_i$ on $\partial N(M(K)) = \partial T_M$. Converting back to the slopes for $k_1 \cup k_2$, we have the pair of slopes $\{4a_i, \frac1{4a_i}\}$.
Examining Table~\ref{table:5/14paths} for the cases where the pair of boundary slopes are reciprocals for $\alpha,\beta>0$ such that one of them is a real number less than $-1$ (in order to be the Jones slope $4a_i$), 
one sees that $F_{\gamma_2}$ is the only candidate surface for a component (or two) $S_{V,i}$ of $\tilde{S}_{V,i}$.  
For convenience, we set $\beta = 2s_i$ and $\alpha = -8r_i$ 
so that $F_{\gamma_2} = F_{\gamma_2, \alpha, \beta} = S_{V,i}$ is orientable; see Subsection~\ref{algorithm}. Thus $\tilde{S}_{V,i}$ is a union of $n$ parallel copies of $S_{V_i} = F_{\gamma_2, \alpha, \beta}$ for some number $n$.

Table~\ref{table:5/14paths} gives the following:
\begin{itemize}
\item $\chi(S_{V,i}) 
=  -2\alpha = 16r_i$,

\item  slope of $\partial S_{V,i}$  on $T_M$ is 
$-\frac{\alpha}{\beta}  
= \frac {4r_i}{s_i}$,  

\item $|\partial S_{V,i} \cap T_K| 
= \gcd(\beta, \alpha) 
= \gcd(2s_i, -8 r_i) 
= 2 \gcd(s_i, 4)$, and

\item $|\partial S_{V,i} \cap T_M| 
= \gcd(\alpha, \beta) 
= 2\gcd(4,s_i)$.
\end{itemize}

The boundary of $\tilde{S}_i$ consists of $n$ copies of the boundary of $S_{V,i}$  on $T_M$, 
so we have

\begin{itemize}
\item 
$|\bdry \tilde{S}_i| = n | \partial S_{V,i} \cap T_M | = 2n \gcd(4, s_i)$. 
\end{itemize}

Moreover, the boundary slope of $\tilde{S}_i$ is the slope $4a_i = \frac{4r_i}{s_i}$ of $\bdry S_{V,i}$ on $T_M$, 
and so this has denominator $\frac{s_i}{\gcd(4r_i, s_i)} = \frac{s_i}{\gcd(4, s_i)}$.  
Since $\tilde{S}_i = m S_{K,i} \cup n S_{V,i}$, we may now calculate
\begin{align*}
  \frac{\chi(\tilde{S}_i)}{|\bdry \tilde{S}_i| \cdot \frac{s_i}{\gcd(4,s)}} 
  &= \frac{m \chi(S_{K,i}) + n \chi(S_{V,i})}{ 2n \gcd(4,s_i) \cdot \frac{s_i}{\gcd(4,s_i)}} \\
  &= \frac{ 2 b_i m |\bdry S_{K,i}| \cdot \frac{s_i}{\gcd(4,s_i)} +16n r_i}{2ns_i} \\
  &= \frac{ 4 b_i n \gcd(4,s_i) \cdot \frac{s_i}{\gcd(4,s_i)} }{2n s_i} +8\frac{r_i}{s_i}\\
  &= 2 b_i + 8a_i
\end{align*}

Given that $\tilde{S}_i$ is a Jones surface, we have that
\[
\frac{\chi(\tilde{S}_i)}{|\bdry \tilde{S}_i| \cdot \frac{s_i}{\gcd(4,s)}} 
= 2b_{M,i}.
\]
But then this implies $2 b_i + 8a_i = 2b_{M,i} = 2b_i - 2$, 
and hence $a_i = - \frac{1}{4}$.  
This contradicts the assumption. 
\end{proof}

\section{The minimum degree of a Mazur double.}

Using the work established in Section~\ref{Jones_generalized_M}, 
we determine the minimum degree of a Mazur double of a knot that satisfies a few conditions.

\begin{proposition}[\textbf{min-degree of $J_{M(K), n}(q)$ with $d_-[J_{K, n}(q)] = \delta^*(n)$ for all $n \ge 1$}]
\label{mindeg_Mazur}
Let $K$ be a knot in $S^3$ and 
$d_-[J_{K,n}(q)]$ is a quadratic quasi-polynomial 
$\delta^*_K(n)=a^*(n) n^2 +b^*(n) n +c^*(n)$ 
for all $n \ge 2$. 
Define $i \in \{ 0, 1 \}$ by $i \equiv n\ (\mathrm{mod}\ 2)$ and put $a^*_i=a^*(i), 
b^*_i = b^*(i)$, and $c^*_i = c^*(i)$. 
We assume that the period of $\delta^*(n)$ is less than or equal to $2$ and 
that $b^*_i \ge 0$, $a^*_1+b^*_1 + c^*_1 \ge 0$. 
 Assume further that if $a^*_i = -\frac{1}{12}$, 
then $b^*_i \ne 0$. 
Then, for suitably large $n$, the minimum degree of the colored Jones polynomial 
of its Mazur  double is given by 
\begin{equation}
  \delta^*_{M(K)}(n)=
   \left\{ \begin{array}{ll} 
    -\frac{5}{4}n^2+\frac{1}{2}n -\frac{1}{2} + C_{i+1} &  (a^*_i > -\frac{1}{12}) \\
    -\frac{5}{4}n^2+\frac{1}{2}n -\frac{1}{2} + C_{i+1} &  (a^*_i = -\frac{1}{12})\\
    (9a^*_i -\frac{1}{2})n^2 +(- 12a^*_i+3b^*_i -\frac{1}{2}) n+4a^*_i-2b^*_i+c^*_i +\frac{1}{2} & (a^*_i< -\frac{1}{12}), 
\end{array} \right. 
\end{equation}
where $C_0 = \frac{3}{4}$ and $C_1 = 4a^*_0 + 2b^*_0 + c^*_0 + \frac{1}{2}$.

Furthermore, suppose we only know that $d_-[J_{K,n}(q)]$ is a quadratic quasi-polynomial 
$\delta^*_K(n)=a^*_i n^2 +b^*_i n +c^*_i$ of period at most $2$ for sufficiently large $n$.   Assume $a_i^* \leq -\frac{1}{12}$.  Then, for suitably large $n$, the minimum degree of the colored Jones polynomial of the colored Jones polynomial of the Mazur double of $K$ is given by 
\[  \delta^*_{M(K)}(n)=    (9a^*_i -\frac{1}{2})n^2 +(- 12a^*_i+3b^*_i -\frac{1}{2}) n+4a^*_i-2b^*_i+c^*_i +\frac{1}{2}.\]
\end{proposition}

\begin{remark}
The assumption that the quadratic quasi-polynomial $\delta_K^*$ has period at most $2$ is mostly for convenience as it simplifies the structure of the argument.  Similar results should hold when the period is larger.

However, the assumption that the minimum degree $d_-[J_{K,n}(q)]$ is the quadratic quasi-polynomial $\delta_K^*$ for {\em all} $n \geq 2$ rather than just for suitably large $n$ is needed to say something definitive when $a_i^* \geq -\frac{1}{12}$.
\end{remark}

\begin{proof}
\bigskip
This
 is obtained from Proposition~\ref{mindeg_Mazur_normalized} by using the formula
\[
\delta^*_{M(K)}(n) = \delta'^*_{M(K)}(n-1) - \frac{1}{2}n + \frac{1}{2}
\]
given in Subsection~\ref{transition}.
\end{proof}

\subsection{Computations of the minimum degree}
\label{computation_mindegrees}

\begin{proposition}[\textbf{min-degree of $J'_{M(K), n}(q)$}]
\label{mindeg_Mazur_normalized}
Let $K$ be a knot in $S^3$ and 
$d_-[J'_{K,n}(q)]$ is a quadratic quasi-polynomial 
$\delta'^*_K(n)=\alpha^*(n) n^2 +\beta^*(n) n+\gamma^*(n)$ for all $n \ge 1$.  
We define $i \in \{0,1 \}$ by $i \equiv n \pmod 2$ and 
 put $\alpha^*_i:=\alpha^* (i)$, $\beta^*_i:=\beta^* (i)$, and $\gamma^*_i:=\gamma^* (i)$. 
We assume that the period of $\delta'^*_K(n)$ is less than or equal to $2$ and that 
$-2\alpha^*_i+\beta^*_i+\frac{1}{2} \ge 0$ and $\gamma^*_0 \ge 0$. 
Assume further that if $\alpha^*_i = -\frac{1}{12}$, 
then $\beta^*_i \ne -\frac{1}{3}$. 
Then, for suitably large $n$,  the minimum degree of the normalized colored Jones polynomial 
of the Mazur double of $K$
is given by 
\begin{equation*}
  \delta'^*_{M(K)}(n)=
   \left\{ \begin{array}{ll} 
    -\frac{5}{4} n^2 - \frac{3}{2} n + C'_i &  (\alpha^*_i > -\frac{1}{12})\\
    -\frac{5}{4} n^2 - \frac{3}{2} n + C'_i &  (\alpha^*_i = -\frac{1}{12}) \\
     (9\alpha^*_i - \frac{1}{2})n^2 +(3\beta^*_i-\frac{5}{2}) n+ \gamma^*_i & (\alpha_i < -\frac{1}{12}), \\
      \end{array} \right.
\end{equation*}
where $C'_0 = 0$ and $C'_1 = \alpha^*_1 + \beta^*_1 + \gamma^*_1 - \frac{1}{4}$.

Furthermore, suppose we only know that $d_-[J'_{K,n}(q)]$ is a quadratic quasi-polynomial $\delta'^*_K(n)= \alpha^*_i n^2 + \beta^*_i n + \gamma^*_i$ of period at most $2$ for sufficiently large $n$. Assume $\alpha^*_i < -\frac{1}{12}$. Then, for suitably large $n$, the minimum degree of the normalized colored Jones polynomial of the Mazur double of $K$ is given by
\[\delta'^*_{M(K)}(n)=(9\alpha^*_i - \frac{1}{2})n^2 +(3\beta^*_i-\frac{5}{2}) n+ \gamma^*_i. \]

\end{proposition}

\begin{proof}
From Proposition~\ref{CJP_M(K)} we have that 
\begin{equation*}
\langle n \rangle J'_{M(K), n}(q) =\sum_{\substack{j,k=0 \\ j\le 2k}}^n \left( \sum_{\substack{|2k-n| \le l \le 2k+n \\ n+l:even}}
g(j,k,l;q) \right).
\end{equation*}
Thus our goal now is to prove that
in each case of $\alpha^*_i>0$ and $\alpha^*_i \le 0$,
for each fixed integer $n \geq 1$,
 there exists a unique triple  $(j_0,k_0,l_0)$ such that 
\[
  \min_{\substack{0\le j,k \le n \\ j \le 2k\\ |2k-n|\le l \le 2k+n \\ n+l:even}} d_-[g(j,k,l;q)]
    =d_-[g(j_0,k_0,l_0;q)].
\]
Then, due to the uniqueness, there can be no cancellation among the lowest degree terms of the double sum Equation~(\ref{doublesum}) shown above so that $\delta'^*_{M(K)}(n)$ may be identified.

\medskip
From Lemma~\ref{lem:partialdegreecomputation}, 
we have that 
\begin{multline*}
d_-[g(j,k,l;q)] = \frac{1}{2}\left(-2j - 2j^2 + 3k + k^2 - \frac{l}{2} + \frac{n}{2} - n^2\right) + d_-[J'_{K,l}(q)] \\
+  \left\{\begin{array}{ll}
           - \frac 1 2 (j+k+\frac l 2+\frac n2) & (j\le \frac {n-l} 2+k), \\
           - \frac 1 2 (-j^2+2jk-k^2+2k-jl+kl+jn-kn-\frac {l^2} 4+ \frac n 2 l-\frac{n^2} 4+n) 
              & (j\ge \frac {n-l} 2+k),
            \end{array}
      \right.  \\
+ \left\{\begin{array}{ll}
            -\frac 1 2 (j+k+n) & (j\le k), \\
            -\frac 1 2 (-j^2+2jk-k^2+2k+n) 
              & (j\ge k).
            \end{array}
      \right. 
\end{multline*}

To handle these last two summands, we will accordingly partition the $(j,k,l)$ domain and consider the following four cases:

\medskip

\noindent
\textbf{Case (1).}\quad 
$j\le k$ and  $j\le \frac {n-l} 2+k$.

\medskip

\noindent
\textbf{Case (2).}\quad  
$j\le k$ and $j\ge \frac {n-l} 2+k$. 

\medskip

\noindent
\textbf{Case (3).}\quad 
$j\ge k$ and  $j\le \frac {n-l} 2+k$. 

\medskip

\noindent
\textbf{Case (4).}\quad  
$j\ge k$ and $j\ge \frac {n-l} 2+k$. 

\medskip

\begin{claim}\label{claim:incase4}
Any $(j,k,l)$ realizing the minimum of $d_-[g(j,k,l;q)]$ satisfies $j\ge k$ and $j\ge \frac {n-l} 2+k$.  That is, it occurs in the domain of Case (4).
\end{claim}
\begin{proof}

We will show the minima of $d_-[g(j,k,l;q)]$ on the domain of Case (1) occurs at its boundary with the domains of Cases (2) and (3), and the minima on their domains occur at their boundaries with Case (4).

\medskip

\noindent
\textbf{Case (1).}\quad 
$j\le k$ and  $j\le \frac {n-l} 2+k$.

From the equalities (\ref{pdeg3j}), (\ref{pdegdelta}), (\ref{pdeg6ja}),  and (\ref{pdeg6jb}), 
one can  see that 

\begin{eqnarray*}
d_-[g(j,k,l;q)]&=&\frac 1 2(-4j-2j^2+k+k^2-l-n^2-n)+d_-[J'_{K, l}(q)]\\
    &=&-(j+1)^2+\frac 1 2(2+k+k^2-l-n^2-n)+d_-[J'_{K, l}(q)]. 
\end{eqnarray*}

If $n\ge l$, then $0 \le k \le \frac{n-l}{2}+k$ and $d_-[g(j,k,l;q)]$ is minimized at $j=k$, 
and so it is contained in the case (3). 
If $n\le l$, then $\frac{n-l}{2}+k \le k$ and $d_-[g(j,k,l;q)]$ is minimized at $j= \frac {n-l} 2+k$, 
and so it is contained in the case (2).

\medskip

\noindent
\textbf{Case (2).}\quad  
$j \le k$ and $j \ge \frac{n-l}{2}+k$. 

In this case we have 
\begin{multline*}
d_-[g(j,k,l; q)]\\
\phantom{xxxxx} = \frac 1 2(-3j-j^2 -2jk +2k^2+jl-kl-jn+kn+\frac{l^2}{4}-\frac{n+1}{2} l
      -\frac{3}{4}n^2-\frac{3}{2}n)+d_-[J'_{K, l}(q)],\\
 \phantom{xxxxx}  =-\frac{1}{2}(j-\frac{1}{2}(-3-2k+l-n))^2+g_4(k,l,n),
\end{multline*} 
where $g_4(k,l,n)$ is a function of $k$, $l$, and $n$. 

Since $l \le 2k + n$ (Proposition~\ref{CJP_M(K)}), 
$-2k + l - n \le 0$, and thus   
$\frac{1}{2}(-3-2k+l-n) < 0$.  
Also $l \le 2k + n$, together with the assumption of Case (2), 
implies $0 \le \frac{n-l}{2}+k \le j \le k$. 
Hence $d_-[g(j,k,l; q)]$ is minimized at $j= k$. 
This shows that Case (2) is contained in Case (4). 

\medskip

\noindent
\textbf{Case (3).}\quad 
$j\ge k$ and  $j\le \frac {n-l} 2+k$. 

In this case we have 
\begin{eqnarray*}
d_-[g(j,k,l;q)]&=&\frac{1}{2}(-3j -j^2 -2jk+2k^2-l-n^2-n)+d_-[J'_{K, l}(q)],\\
	&=& -\frac{1}{2}(j-\frac{1}{2}(-3-2k))^2 +g_5(k,l,n),
\end{eqnarray*} 
where $g_5(k,l,n)$ is a function of $k$, $l$, and $n$. 

Since $\frac{1}{2}(-3-2k)  < 0$ and $0 \le k \le j \le \frac{n-l}{2}+k$, 
$d_-[g(j,k,l;q)]$ is minimized at  $j= \frac {n-l}{2}+k\le 2k$.
(Because $|2k-n| \le l$, we have $\frac{n-l}{2} \le k$.) 
Thus, this case is contained in the Case (4).
\end{proof}

Due to Claim~\ref{claim:incase4}, any triple $(j,k,l)$ that minimizes $d_-[g(j,k,l;q)]$ satisfies:

\noindent
{\bf Case (4)}.  $j\ge k$ and $j\ge \frac {n-l} 2+k$. 

We henceforth assume $(j,k,l)$ satisfies these inequalities.
Then we have
\begin{multline*}
d_-[g(j,k,l;q)]\\
\phantom{xxxxx}=\frac{1}{2}(-2j-k-4jk+3k^2+jl-kl-jn+kn+\frac {l^2} 4-\frac {n+1} 2 l
      -\frac{3}{4}n^2-\frac 3 2n)+d_-[J'_{K, l}(q)]\\
\phantom{xxxxx}=\frac{1}{2}(-2-4k+l-n)j+ g_6(k,l,n),
\end{multline*} 
where $g_6(k,l,n)$ is a function of $k$, $l$, and $n$. 

Since $l \le 2k + n$ (Proposition~\ref{CJP_M(K)}), 
$-2k + l -n < 0$, 
and hence $-2-4k + l -n = (-2-2k) + (-2k +l -n) < 0$. 
Following Proposition~\ref{CJP_M(K)}, $j \le \mathrm{min}\{ n, 2k \}$.
Hence $d_-[g(j,k,l;q)]$ is minimized at either $j=n$ or $j=2k$, depending on which is smaller.
So we now separate the argument into those two cases, either (A) $n  \le 2k$ or (B) $n \ge 2k$. 

\medskip

\bigskip

\noindent
\textbf{Case (4A)}\   Assume $n \le 2k$.

The function $d_-[g(j,k,l;q)]$ of $j$ is minimized at $j=n$. One can see that 
\begin{align*}
\min_{\max \{k,\frac {n-l} 2+k \}\le j \le n}d_-[g(j,k,l;q)]
 &=  d_-[g(n,k,l;q)]\\
&=\frac{1}{2}(-k+3k^2-kl-3kn+\frac {l^2}{4}+\frac{n-1}{2}l
      -\frac{7}{4}n^2-\frac{7}{2}n)\\ 
      & \phantom{justanotherbighugegrandioussteptotheright}+d_-[J'_{K, l}(q)]\\
      &=\frac{3}{2}(k-\frac{1}{6}(1+l+3n))^2+ \frac{1}{12}l^2 - \frac{1}{3}l -\frac{5}{4}n^2 - 2n - \frac{1}{24}\\
      & \phantom{justanotherbighugegrandioussteptotheright}+d_-[J'_{K, l}(q)].
\end{align*}

\medskip
Therefore $d_-[g(n,k,l;q)]$ is minimized when $k$ is an integer closest to $\frac{1}{6}(1+l+3n)$.
Since $n+l$ is even, 
the numerator is odd.  
So $1+l+3n$ equals either $6m + 1$, $6m +3$, or $6m + 5$ for some integer $m$.  
Hence, 
$l+1$ equals either $1, 0$ or $-1$ mod $3$ respectively. 
Hence, depending on the value of $l+1$ mod $3$, 
we have that $d_-[g(n,k,l;q)]$ is minimized exactly at
$k =\frac{1}{6}(1+l+3n)-\frac{x}{6}$ for $x \in \{-3,-1,1,3\}$.  
More specifically, the minimum occurs at
 the following values of $k$:

\begin{center}
\begin{tabular}{rl}
 ($l+1 \equiv_3 1$): & $k=\frac{1}{6}(1+l+3n)-\frac{1}{6}$ \\
 ($l+1 \equiv_3 -1$):  &$k=\frac{1}{6}(1+l+3n)+\frac{1}{6}$\\
 ($l+1 \equiv_3 0$): & $k=\frac{1}{6}(1+l+3n)-\frac{3}{6}$ and  $k=\frac{1}{6}(1+l+3n)+\frac{3}{6}$.
\end{tabular}
\end{center}
See Figure~\ref{min_deg_case4}.

\begin{figure}[!ht]
\includegraphics[width=1.0\linewidth]{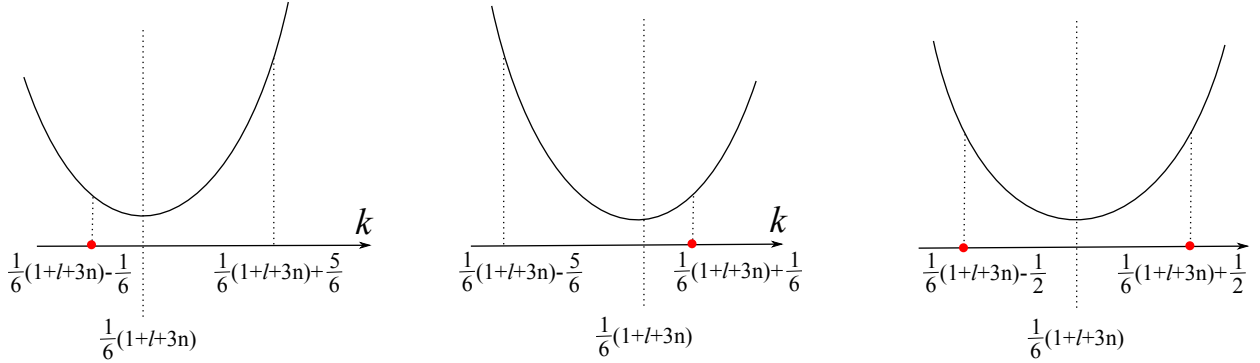}
\caption{The integer closest to the axis;
$l+1  \equiv_3 1 $ (Left), 
$l+1  \equiv_3 -1$ (Middle) and 
$l+1 \equiv_3 -0$ (Right).
}
\label{min_deg_case4}
\end{figure}

Now recall that we aim to determine the minimum degree of the sum 
\[ \sum_{\substack{0\le j,k \le n \\ |2k-n|\le l \le 2k+n \\ n+l:even}} g(j,k,l;q) \]
which we have determined occurs only when $j\geq k$ and $j \geq \frac{n-l}{2} + k$.  
Furthermore we are presently examining the case that $2k \ge n$ in which any minimum occurs at $j=n$ and where we have special values for $k$ depending on $l+1$ mod $3$.
Hence, let us investigate the minimum degree of the sum
\begin{multline*}
\sum_{\substack{0\le k \le n  \le 2k \\ 2k-n\le l \le 2k+n \\ n+l:even}} g(n,k,l;q) =\\
=\sum_{\substack{0 \le l \le 3n\\ l+1 \equiv 1\ \mathrm{mod}\ 3\\ n + l \colon \mathrm{even}}}
g(n, \frac{1}{6}(1+l+3n)- \frac{1}{6}, l; q)
+ 
\sum_{\substack{0 \le l \le 3n\\ l+1 \equiv -1\ \mathrm{mod}\ 3\\ n + l \colon \mathrm{even}}}
g(n, \frac{1}{6}(1+l+3n)+ \frac{1}{6}, l; q)
\\ 
\phantom{a bit of space}
+
\sum_{\substack{0 \le l \le 3n\\ l+1 \equiv 0\ \mathrm{mod}\ 3\\ n + l \colon \mathrm{even}}}
\left(g\bigl(n, \frac{1}{6}(1+l+3n)- \frac{1}{2}, l; q \bigr)+ g\bigl(n, \frac{1}{6}(1+l+3n)+ \frac{1}{2}, l; q\bigr)\right), 
\tag{$\dagger$}
\end{multline*}

\medskip

where the inequalities $k \le n \le 2k$ and $2k-n \le l \le 2k+n$ implies $0 \le l \le 3n$.

To compare the functions $g(n, \frac{1}{6}(1+l+3n) + \frac{x}{6}, l; q)$ for $x \in \{-3,-1,1,3\}$, it will be convenient to consider the function
\begin{equation}\label{hAl}
h_A(l) = \begin{cases} (\alpha^*_i + \frac{1}{12})l^2 + (\beta^*_i - \frac{1}{3})l +\gamma^*_i & l \geq 1 \\
0 & l=0 \end{cases}.
\end{equation}
Recall that  $\delta'^*_K(n) = \alpha^*_i n^2 + \beta^*_i n + \gamma^*_i$ and we have defined $i \in \{0,1\}$ as the value of $n \mod 2$.  Furthermore, since $n + l$ is even, $n$ and $l$ have the same parity. 
So for a fixed $n$, $l$ varies over only even integers if $n$ is an even integer, and 
$l$ varies over only odd integers if $n$ is an odd integer. 
In particular $i \in \{0,1\}$ is also the value of $l \mod 2$.

Furthermore, note that we have only assumed that 
$d_-[J'_{K,n}(q)]=\delta'^*_K(n)= \alpha^*_i n^2 + \beta^*_i n + \gamma^*_i$ for $n \geq 1$. 
So $d_-[J'_{K,0}(q)]$ need not equal $\gamma^*_0$.  
Regardless, since $J'_{K, 0}(q) = 1$, 
we have $d_-[J'_{K, 0}(q)] = 0$.  
For this reason, we define $h_A(0) = 0$.

{\bf Subcase ($l +1\equiv_3 1$)}:
The function $d_-[g(n,k,l;q)]$ is minimized at $k= \frac{1}{6}(1+l+3n)- \frac{1}{6}$.
The situation that $l=0$ needs to be considered separately.  

For $l \ne 0$, we have: 
\begin{eqnarray*}
d_-[g(n, \frac{1}{6}(1+l+3n)- \frac{1}{6}, l; q)]
&=&\frac{1}{12}l^2 - \frac{1}{3}l - \frac{5}{4}n^2 - 2n + d_-[J'_{K, l}(q)]\\ 
&=&\frac{1}{12}l^2 - \frac{1}{3}l - \frac{5}{4}n^2 - 2n +\alpha^*_i l^2 + \beta^*_i l + \gamma^*_i \\ 
&=&(\alpha^*_i + \frac{1}{12})l^2 + (\beta^*_i - \frac{1}{3})l +\gamma^*_i - \frac{5}{4}n^2 - 2n\\
&=& h_A(l) - \frac{5}{4}n^2 - 2n, 
\end{eqnarray*}

If $l = 0$, then we have
\begin{eqnarray*}
d_-[g(n, \frac{1}{6}(1+3n)- \frac{1}{6}, 0; q)]
&=& - \frac{5}{4}n^2 - 2n + d_-[J'_{K, 0}(q)]\\ 
&=& -\frac{5}{4}n^2 - 2n
\end{eqnarray*}

\medskip

{\bf Subcase  ($l+1 \equiv_3 -1$)}:
The function $d_-[g(n,k,l;q)]$ is minimized at $k= \frac{1}{6}(1+l+3n)+ \frac{1}{6}$ and we may calculate:

\begin{eqnarray*}
d_-[g(n, \frac{1}{6}(1+l+3n)+\frac{1}{6}, l; q)]
&=&\frac{1}{12}l^2 - \frac{1}{3}l - \frac{5}{4}n^2 - 2n + d_-[J'_{K, l}(q)]\\ 
&=&\frac{1}{12}l^2 - \frac{1}{3}l - \frac{5}{4}n^2 - 2n +\alpha^*_i l^2 + \beta^*_i l + \gamma^*_i \\ 
&=&(\alpha^*_i + \frac{1}{12})l^2 + (\beta^*_i - \frac{1}{3})l +\gamma^*_i - \frac{5}{4}n^2 - 2n\\
&=& h_A(l) - \frac{5}{4}n^2 - 2n. 
\end{eqnarray*} 

\medskip

{\bf Subcase ($l+1 \equiv_3 0$)}:
The function $d_-[g(n,k,l;q)]$ is minimized at each at $k= \frac{1}{6}(1+l+3n)- \frac{1}{2}$ and  $k= \frac{1}{6}(1+l+3n)+ \frac{1}{2}$.   Note that potentially the terms of degrees $d_-[g(n, \frac{1}{6}(1+l+3n)- \frac{1}{2}, l; q]$ and $d_-[g(n, \frac{1}{6}(1+l+3n) + \frac{1}{2}, l; q]$ might cancel.  
We may calculate their degrees:
\begin{eqnarray*}
& &d_-[g(n, \frac{1}{6}(1+l+3n) - \frac{1}{2}, l; q)] = d_-[g(n, \frac{1}{6}(1+l+3n)+ \frac{1}{2}, l; q)]\\
&=& \frac{1}{12}l^2 - \frac{1}{3}l - \frac{5}{4}n^2 - 2n + \frac{1}{3} + d_-[J'_{K, l}(q)]\\
&=& (\alpha^*_i + \frac{1}{12})l^2 + (\beta^*_i - \frac{1}{3})l +\gamma^*_i - \frac{5}{4}n^2 - 2n + \frac{1}{3}\\
&=& h_A(l) - \frac{5}{4}n^2 - 2n + \frac{1}{3}.
\end{eqnarray*}

\medskip

Summarizing, we have found that
\[
  \left\{
    \begin{array}{ll}
      d_-[g(n, \frac{1}{6}(1+l+3n)- \frac{1}{6}, l; q)] = h_A(l) - \frac{5}{4}n^2 - 2n & (l+1 \equiv 1\ \mathrm{mod}\ 3) \\
      d_-[g(n, \frac{1}{6}(1+l+3n) + \frac{1}{6}, l; q)] = h_A(l) - \frac{5}{4}n^2 - 2n  & (l+1 \equiv -1\ \mathrm{mod}\ 3) \\
      d_-[g(n, \frac{1}{6}(1+l+3n) - \frac{1}{2}, l; q)] =h_A(l) - \frac{5}{4}n^2 - 2n + \frac{1}{3}  & (l+1 \equiv 0\  \mathrm{mod}\ 3)\\
       d_-[g(n, \frac{1}{6}(1+l+3n)+ \frac{1}{2}, l; q)]
= h_A(l) - \frac{5}{4}n^2 - 2n + \frac{1}{3}  & (l+1 \equiv 0\  \mathrm{mod}\ 3)
    \end{array}
  \right.
\]
where from Equation~\ref{hAl} we have
\[h_A(l) = \begin{cases} (\alpha^*_i + \frac{1}{12})l^2 + (\beta^*_i - \frac{1}{3})l +\gamma^*_i & l \geq 1 \\ 0 & l=0 \end{cases}.\]
Due to the leading term of $h_A(l)$, there are three cases to consider depending upon whether (i) $12\alpha^*_i + 1 > 0$,\ 
(ii) $12\alpha^*_i + 1 = 0$, or (iii) $12\alpha^*_i + 1 < 0$. 
Note that if $12\alpha^*_i + 1 \ne 0$ and $l\neq 0$, 
we may write
\begin{equation*}
h_A(l) = (\alpha^*_i + \frac{1}{12})\left(l + \frac{2(3\beta^*_i - 1)}{12\alpha^*_i + 1}\right)^2 
- \frac{(3\beta^*_i - 1)^2}{3(12\alpha^*_i + 1)} + \gamma^*_i.
\end{equation*}

\smallskip
As discussed previously, since $n+l$ is even, the parity of $l$ agrees with that of $n$.
For each choice of $n$ even or $n$ odd, we need to see among the sum $(\dagger)$ which value of $l$ gives the minimum degree where $l$ varies over  even integers or over odd integers, respectively.
Hence we now examine the two cases \textbf{Case (4--A--even)} where $n$ is even and \textbf{Case (4--A--odd)} where $n$ is odd.  As $i$ equals the reduction of $n$ mod $2$, these two cases each split into three according to whether $12\alpha^*_i + 1$ is positive, zero, or negative
\medskip

\noindent
\textbf{Case (4--A--even)}
Assume that $n$ is even. 
Then as mentioned above, $l$ varies over only even integers and 
$(\alpha^*_i, \beta^*_i, \gamma^*_i) = (\alpha^*_0, \beta^*_0, \gamma^*_0)$.

\medskip

\noindent
\textbf{Case (4--A--even--(i))}
Assume $12\alpha^*_0 + 1 > 0$. 
Then using the initial assumption $-2\alpha^*_0 + \beta^*_0 - \frac{1}{2} \ge 0$, 
we have
$\beta^*_0 - \frac{1}{2} \ge 2\alpha^*_0 > -\frac{1}{6}$, 
and thus $3\beta^*_0 - 1 > 0$. 
This implies $-\frac{2(3\beta^*_i - 1)}{12\alpha^*_i + 1} < 0$. 
Hence $h_A(l)$ is a strictly increasing function on $[0, 3n]$. (Recall that $\gamma^*_0 \ge 0$ by the assumption of Proposition~\ref{mindeg_Mazur_normalized} and $h_A(0) = 0$ by the definition of $h_A(l)$.)

Since $h_A(l)$ is strictly increasing on $[0, 3n]$ and 
$h_A(0) = 0$,
\[
- \frac{5}{4}n^2 - 2n = h_A(0)- \frac{5}{4}n^2 - 2n 
< h_A(4) - \frac{5}{4}n^2 - 2n
\]
 and 
\[
- \frac{5}{4}n^2 - 2n = h_A(0) - \frac{5}{4}n^2 - 2n 
< h_A(2)  - \frac{5}{4}n^2 - 2n + \frac{1}{3}. 
\]
It follows that $(*)$ is uniquely minimized at $l = 0$ and 
\[
d_-[g(n, \frac{1}{6}(1+3n)- \frac{1}{6}, 0; q)] = - \frac{5}{4}n^2 - 2n.
\] 

\medskip

\noindent
\textbf{Case (4--A--even--(ii))}
Assume $12\alpha^*_0 + 1 = 0$. 
The assumption $-2\alpha^*_0 + \beta^*_0 - \frac{1}{2} \ge 0$ implies 
$\beta^*_0 - \frac{1}{2} \ge 2\alpha^*_0 = -\frac{1}{6}$, 
i.e. $\beta^*_0 \ge \frac{1}{3}$. 
Recall that since we assumed $(\alpha^*_0, \beta^*_0) \ne (-\frac{1}{12}, \frac{1}{3})$, 
we have $\beta^*_0 > \frac{1}{3}$. 
Then because $12\alpha^*_0 + 1 = 0$, we have 
$h_A(l) = (\beta^*_0 - \frac{1}{3})l + \gamma^*_0$ for $l \geq 1$ and $h_A(0) = 0$, 
which is a strictly increasing function on $[0, 3n]$.

Since $h_A(l)$ is strictly increasing on $[0, 3n]$ and 
$\gamma^*_0 \ge 0$ (by the assumption of Proposition~\ref{mindeg_Mazur_normalized}),
$(\dagger)$ is uniquely minimized at $l = 0$ and 
\[
d_-[g(n, \frac{1}{6}(1+3n)- \frac{1}{6}, 0; q)] = - \frac{5}{4}n^2 - 2n.
\] 

\medskip

\noindent
\textbf{Case (4--A--even--(iii))}
Assume $12\alpha^*_0 + 1 < 0$. 
Then $-\frac{2(3\beta^*_i - 1)}{12\alpha^*_i + 1} < n$ for sufficiently large $n$,
and 
$h_A(l)$ is strictly decreasing on $[n, 3n]$.

Since $h_A(l)$ is strictly decreasing on $[n, 3n]$, 
\[
h_A(3n)-\frac{5}{4}n^2 - 2n < h_A(3n-2) -\frac{5}{4}n^2 - 2n < h_A(3n-4) -\frac{5}{4}n^2 - 2n + \frac{1}{3}, 
\]
and hence $(\dagger)$  is uniquely minimized at $l = 3n$ so that 
\[
d_-[g(n, \frac{1}{6}(1+6n)- \frac{1}{6}, 3n; q)] = h_A(3n) - \frac{5}{4}n^2 - 2n 
= (9\alpha^*_0 - \frac{1}{2})n^2 + (3\beta^*_0 -3) n + \gamma^*_0
\] 

\medskip

It follows from Cases (4--A--even--(i), (ii), (iii)), 
we obtain the following. 
For suitably large $n$,
\[
d_-\left[
\sum_{\substack{k \ge \frac{n}{2}\\ |2k-n| \le l \le 2k+n\\ n,\ l \colon even}} g(n, k, l; q)\right]
= 
\begin{cases}
-\frac{5}{4}n^2 - 2n & \mathrm{if}\ 12\alpha^*_0 + 1 > 0, \\
-\frac{5}{4}n^2 - 2n & \mathrm{if}\ 12\alpha^*_0 + 1 = 0, \\
(9\alpha^*_0 - \frac{1}{2})n^2 + (3\beta^*_0 -3) n + \gamma^*_0 & \mathrm{if}\ 12\alpha^*_0 + 1 < 0
\end{cases}
\]

\bigskip

\noindent
\textbf{Case (4--A--odd)}
Assume that $n$ is odd. 
Then as mentioned above,
$l$ varies over only odd integers and 
$(\alpha^*_i,\ \beta^*_i, \gamma^*_i) = (\alpha^*_1,\ \beta^*_1, \gamma^*_1)$.   Notably, we need not be bothered with the special case of $h_A(l)$ for $l=0$.

\smallskip

\noindent
\textbf{Case (4--A--odd--(i))}
Assume
$12\alpha^*_1 + 1 > 0$. 
The initial assumption 
$-2\alpha^*_1 + \beta^*_1 - \frac{1}{2} \ge 0$, 
implies that $3\beta^*_1 - 1 > 0$.
Hence $-\frac{2(3\beta^*_1 - 1)}{12\alpha^*_1 + 1} < 0$. 
Then 
$h_A(l) = (\alpha^*_i + \frac{1}{12})l^2 + (\beta^*_i - \frac{1}{3})l +\gamma^*_i$ is a strictly increasing function on $[0, 3n]$.

Since $h_A(1) < h_A(3) < h_A(5)$, 
$(\dagger)$ is uniquely minimized at $l = 1$ and 
\[
d_-[g(n, \frac{1}{6}(2+3n)+ \frac{1}{6}, 1; q)] = h_A(1) - \frac{5}{4}n^2 - 2n 
= - \frac{5}{4}n^2 - 2n +\alpha^*_1 +\beta^*_1 + \gamma^*_1 - \frac{1}{4}. 
\]

\medskip

\noindent
\textbf{Case (4--A--odd--(ii))}
Assume $12\alpha^*_1 + 1 = 0$. 
As in Case (4--A--even--(ii)), 
$\beta^*_1 > \frac{1}{3}$ and 
$h_A(l) = (\beta^*_1 - \frac{1}{3})l$ is a strictly increasing function on $[0, 3n]$. 

Since 
\[
h_A(1) -\frac{5}{4}n^2 - 2n < h_A(5) - \frac{5}{4}n^2 - 2n + \frac{1}{3}
\]
and 
\[
h_A(1) -\frac{5}{4}n^2 - 2n < h_A(3) - \frac{5}{4}n^2 - 2n,
\]
$(\dagger)$ is uniquely minimized at $l = 1$ and 
\[
d_-[g(n, \frac{1}{6}(2+3n)- \frac{1}{6}, 1; q)] = h_A(1) - \frac{5}{4}n^2 - 2n
= - \frac{5}{4}n^2 - 2n +\alpha^*_1 +\beta^*_1 + \gamma^*_1 - \frac{1}{4}. 
\] 

\medskip

\noindent
\textbf{Case (4--A--odd--(iii))}
Assume $12\alpha^*_1 + 1 < 0$. 
Then as in Case (4--A--even--(iii)), 
For sufficiently large $n$, 
$h_A(l)$ is strictly decreasing on $[n, 3n]$ and 
\[
h_A(3n)-\frac{5}{4}n^2 - 2n < h_A(3n-2) -\frac{5}{4}n^2 - 2n < h_A(3n-4) -\frac{5}{4}n^2 - 2n + \frac{1}{3}, 
\]
and hence $(\dagger)$ is uniquely minimized at $l = 3n$ so that  
\[
d_-[g(n, \frac{1}{6}(1+6n)- \frac{1}{6}, 3n; q)] = h_A(3n) - \frac{5}{4}n^2 - 2n 
= (9\alpha^*_1 - \frac{1}{2})n^2 + (3\beta^*_1 -3) n + \gamma^*_1
\]

\medskip

It follows from Cases (4--A--odd--(i), (ii), (iii)), 
we obtain the following. 
For suitably large $n$, 
\[
d_-\left[
\sum_{\substack{k \ge \frac{n}{2}\\ |2k-n| \le l \le 2k+n\\ n,\ l \colon odd}} g(n, k, l; q)\right]
= 
\begin{cases}
- \frac{5}{4}n^2 - 2n +\alpha^*_1 +\beta^*_1 + \gamma^*_1 - \frac{1}{4} & \mathrm{if}\ 12\alpha^*_1 + 1 > 0, \\
- \frac{5}{4}n^2 - 2n +\alpha^*_1 +\beta^*_1 + \gamma^*_1 - \frac{1}{4} & \mathrm{if}\ 12\alpha^*_1 + 1 = 0, \\
(9\alpha^*_1 - \frac{1}{2})n^2 + (3\beta^*_1 -3) n + \gamma^*_1 & \mathrm{if}\ 12\alpha^*_1 + 1 < 0
\end{cases}
\]

\bigskip

For convenience, 
we summarize Case (4--A) where $k \ge \frac{n}{2}$ as follows. 

\begin{claim}[$k \ge \frac{n}{2}$]
\label{4A_claim}
For suitably large $n$, 
\[
d_-\left[
\sum_{\substack{k \ge \frac{n}{2}\\ |2k-n| \le l \le 2k+n\\ n + l \colon even}} g(n, k, l; q)\right]
= 
\begin{cases}
-\frac{5}{4}n^2 - 2n + h_A(i) & \mathrm{if}\ 12\alpha^*_i + 1 > 0, \\
-\frac{5}{4}n^2 - 2n + h_A(i) & \mathrm{if}\ 12\alpha^*_i + 1 = 0, \\
(9\alpha^*_i - \frac{1}{2})n^2 + (3\beta^*_i -3) n + \gamma^*_i & \mathrm{if}\ 12\alpha^*_i + 1 < 0,
\end{cases}
\]
\end{claim}
where $h_A(0) = 0$ and $h_A(1) = \alpha^*_1 +\beta^*_1 + \gamma^*_1 - \frac{1}{4}$.

\begin{remark}\label{rem:4Aquadquasi}
In Cases (4--A--even(i), (ii)) and Cases (4--A--odd(i), (ii)) we specifically use the assumption that $d_-[J'_{K,l}(q)]=\delta'^*_K(l)$ for all $l \geq 1$ to conclude that $d_-[g(n,k,l;q)]$ cases occurs at $l=0$ or $l=1$ in these cases.  If $d_-[J'_{K,l}(q)]$ were only a quadratic quasi-polynomial for suitably large $l$, then $d_-[g(n,k,l;q)]$  might be much smaller and occur at another value of $l$ in these four cases.  Consequently, without the assumption, we could only obtain the equation of Claim~\ref{4A_claim} for $12\alpha^*_0 + 1 < 0$, Cases (4--A--even(iii)) and (4--A--odd(iii)).
\end{remark}

\vskip 1.0cm

\noindent
\textbf{Case (4--B)}\   Assume $2k \le n$.

The function 
$d_-[g(j, k, l; q)]$ of $j$ is minimized at $j = 2k$.
One can see that
\begin{align*}
  \min_{\max \{k,\frac {n-l} 2+k \}\le j \le 2k}d_-[g(j,k,l;q)]
&=  d_-[g(2k,k,l;q)]\\
&=\frac{1}{2}(-5k -5k^2 + kl - kn+\frac {l^2}{4}-\frac{n+1}{2}l
      -\frac{3}{4}n^2-\frac{3}{2}n)\\ 
      & \phantom{justanotherbighugegrandioussteptotheright}+d_-[J'_{K, l}(q)]\\
      &=-\frac{5}{2}(k - \frac{1}{10}(-5 + l -n))^2+ g_7(l, n),
\end{align*} 
where $g_7(l, n)$ is a function of $l$ and $n$. 

Since $l \le 2k + n$ (Proposition~\ref{CJP_M(K)}), 
the assumption $2k\le n$ implies that 
$l \le 2k + n \le 2n$. 
Thus
\[
\frac{1}{10}(-5 + l - n) \le \frac{1}{10}(-5 + 2n - n) = \frac{n-5}{10} < \frac{n-1}{4}.
\]
Therefore, since $0\leq k \leq \frac{n}{2}$, this shows that $d_-[g(2k,k,l;q)]$ is minimized at $k=\frac{n}{2}$ if $n$ is even and at $k=\frac{n-1}{2}$ if $n$ is odd. 
See Figure~\ref{4B}. 

\medskip

\noindent
\textbf{Case (4--B--even)}
If $n$ is even, then $\frac{n}{2}$ is an integer and $d_-[g(2k, k, l; q)]$ is minimized at $k = \frac{n}{2}$ (Figure~\ref{4B}(Left)).
Hence this case is contained in the Case (A). 

\medskip

\noindent
\textbf{Case (4--B--odd)}
If $n$ is odd, then $\frac{n-1}{2}$ is an integer and $d_-[g(2k, k, l; q)]$ is minimized at $k =\frac{n-1}{2}$ (Figure~\ref{4B}(Right)). 

\begin{figure}[!ht]
\includegraphics[width=0.65\linewidth]{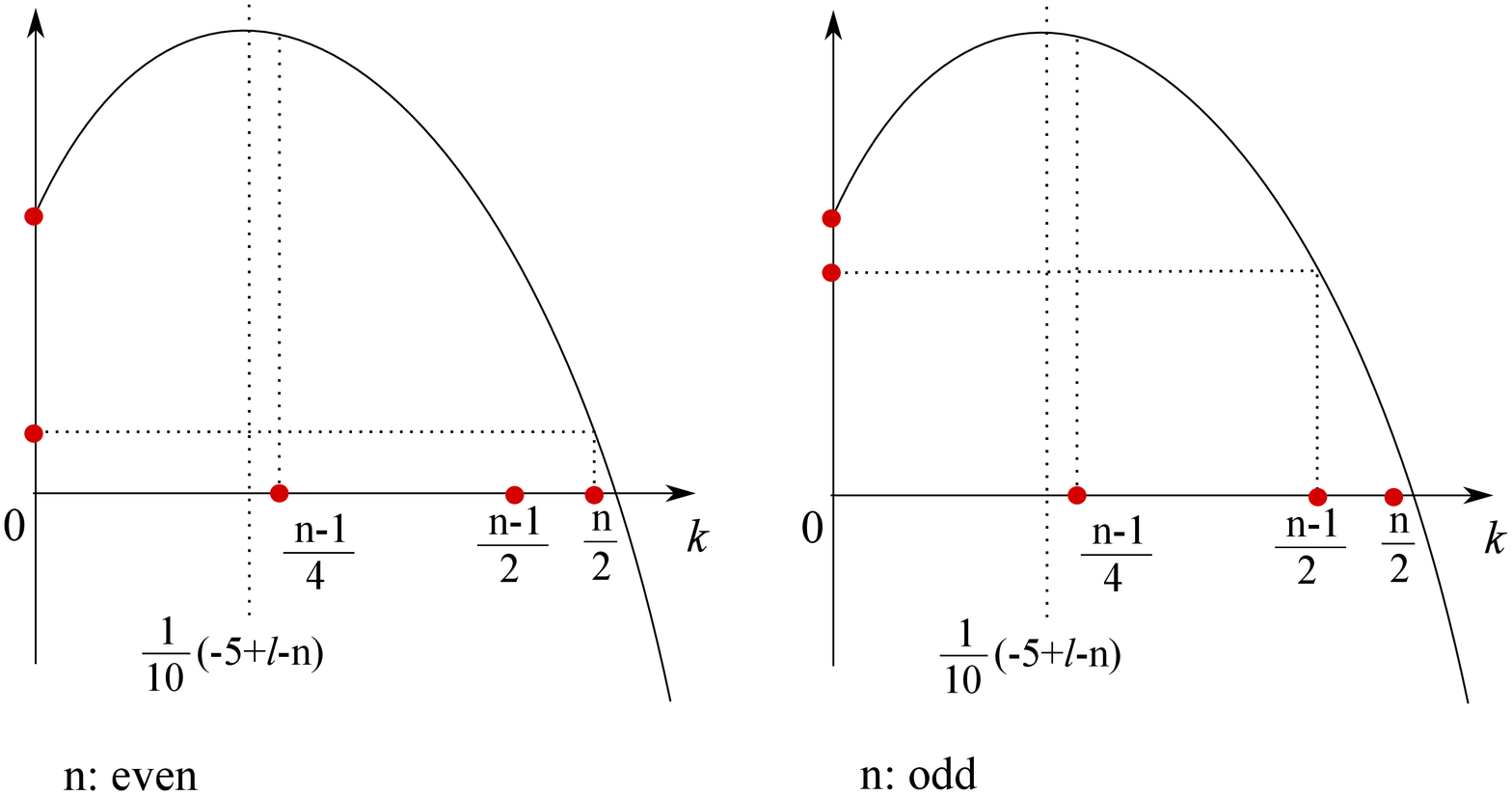}
\caption{When $n$ is odd, $d_-[g(2k, k, l; q)]$ is minimized at the integer $k =\frac{n-1}{2}$. }
\label{4B}
\end{figure}

\medskip

Hence we see that 
\begin{eqnarray*}
 \min_{k < \frac{n}{2}}d_-[g(2k, k, l; q)]
&=&  d_-[g(2(\frac{n-1}{2}), \frac{n-1}{2}, l; q)]\\
&=& \frac {1}{8}l^2- \frac{1}{2}l -\frac{5}{4}n^2- \frac{1}{2}n + \frac{5}{8}+d_-[J'_{K, l}(q)]\\
&=& h_B(l) - \frac{5}{4}n^2 - \frac{1}{2}n + \frac{5}{8}, 
\end{eqnarray*} 
where $h_B(l) = \frac {1}{8}l^2- \frac{1}{2}l + d_-[J'_{K, l}(q)]$. 
Since $l \le 2k + n$ (Proposition~\ref{CJP_M(K)}), 
we have $l \le (n-1) + n = 2n -1$. Similarly, $|2k-n|\le l$, so $1 \le l$ since $n$ is odd.  Moreover, $l$ is also odd because $n+l$ is even.  
Thus we may write $h_B(l)$ as 
\[
h_B(l) = (\alpha^*_1 + \frac{1}{8})l^2 + (\beta^*_1-\frac{1}{2})l + \gamma^*_1
\]
where $l$ ranges over the odd integers in the interval $[1,2n-1]$.

If $\alpha^*_1 + \frac{1}{8} \ne 0$, 
then 

\begin{align*}
h_B(l) &= (\alpha^*_1 + \frac{1}{8})\left(l + \frac{4(\beta^*_1 -\frac12)}{8\alpha^*_1+1}\right)^2 
-(\alpha^*_1 + \frac{1}{8})\left(\frac{4(\beta^*_1 -\frac12)}{8\alpha^*_1+1}\right)^2 + \gamma^*_1 \\
&= (\alpha^*_1 + \frac{1}{8})\left(l + \frac{4\beta^*_1 -2}{8\alpha^*_1+1}\right)^2 
-(\alpha^*_1 + \frac{1}{8})\left(\frac{4\beta^*_1 -2}{8\alpha^*_1+1}\right)^2 + \gamma^*_1
\end{align*}

Then there are three cases to consider 
(i) $8\alpha^*_1 + 1>0$,\ (ii) $8\alpha^*_1 + 1 = 0$, or (iii) $8\alpha^*_1 + 1 < 0$. 

\medskip

\noindent
\textbf{Case (4--B--odd--(i))}
Assume that $8\alpha^*_1 + 1 > 0$. 
For sufficiently large $n$, 
we may assume $-\frac{4\beta^*_1 -2}{8\alpha^*_1+1} < 2n-1$. 
Since $8\alpha^*_1 + 1>0$, 
$h_B(l)$ is a strictly increasing function on the interval
$[\max \{1, -\frac{4\beta^*_1 -2}{8\alpha^*_1+1}\}, 2n-1]$
(which is contained in $[1, 2n-1]$). 
Then there exists an odd integer $l_0$ such that 
\[
d_-\left[
\sum_{\substack{1 \le l \le 2n-1\\ l \colon odd}} g\left(2(\frac{n-1}{2}), \frac{n-1}{2}, l; q\right)\right]
= - \frac{5}{4}n^2 - \frac{1}{2}n + \frac{5}{8} + h_B(l_0)
= - \frac{5}{4}n^2 - \frac{1}{2}n + c_+, 
\]
where $c_+ =  \frac{5}{8} + h_B(l_0)$. 

(Note: $l_0$ is unique unless $-\frac{4(\beta^*_1 -\frac12)}{8\alpha^*_1+1}$ is an even integer in the interval $[2,2n-2]$.)
\medskip

\noindent
\textbf{Case (4--B--odd--(ii))}
Assume that $8\alpha^*_1 + 1 = 0$.   We then consider the sign of $\beta^*_1 -\frac{1}{2}$.

\medskip

If $\beta^*_1 -\frac{1}{2}= 0$, 
then $h_B(l)$ is the constant $\gamma^*_1$. 
Thus the minimum degree of $g(2(\frac{n-1}{2}), \frac{n-1}{2}, l; q)$ is the constant 
$-\frac{5}{4}n^2 - \frac{1}{2}n + \frac{5}{8} + \gamma^*_1$ independent of $l$.
So among the sum 
$\displaystyle \sum_{\substack{1 \le l \le 2n-1\\ l \colon odd}} g\left(2(\frac{n-1}{2}), \frac{n-1}{2}, l; q\right)$, 
these minimum degree terms may cancel out one another. 
Hence we only have the following bound from the below: 
\begin{align*}
d_-\left[
\sum_{\substack{1 \le l \le 2n-1\\ l \colon odd}} g\left(2(\frac{n-1}{2}), \frac{n-1}{2}, l; q\right)\right]
&\ge -\frac{5}{4}n^2 - \frac{1}{2}n + \frac{5}{8} + \gamma^*_1\\
&= -\frac{5}{4}n^2 - \frac{1}{2}n + d, 
\end{align*}
where $d = \frac{5}{8} + \gamma^*_1$. 

\medskip

If $\beta^*_1 - \frac{1}{2}> 0$, 
then $h_B(l) = (\beta^*_1 - \frac{1}{2})l + \gamma^*_1$ is strictly increasing on $[1, 2n-1]$. 
Then $h_B(l)$ is uniquely minimized at $l = 1$, 
and we have 
\begin{align*}
d_-\left[
\sum_{\substack{1 \le l \le 2n-1\\ l \colon odd}} g\left(2(\frac{n-1}{2}), \frac{n-1}{2}, l; q\right)\right]
&= h_B(1) -\frac{5}{4}n^2 - \frac{1}{2}n + \frac{5}{8}\\
&= -\frac{5}{4}n^2 - \frac{1}{2}n + \beta^*_1 + \gamma^*_1 +\frac{1}{8}. 
\end{align*}

\medskip

If $\beta^*_1 - \frac{1}{2}< 0$, 
then $h_B(l) = (\beta^*_1 - \frac{1}{2})l + \gamma^*_1$ is strictly decreasing on $[1, 2n-1]$. 
Therefore, 
$h_B(l)$ is uniquely minimized at $l = 2n-1$, 
and we have 
\begin{align*}
d_-\left[
\sum_{\substack{1 \le l \le 2n-1\\ l \colon odd}} g\left(2(\frac{n-1}{2}), \frac{n-1}{2}, l; q\right)\right]
&= h_B(2n-1) -\frac{5}{4}n^2 - \frac{1}{2}n + \frac{5}{8}\\
&= -\frac{5}{4}n^2 -\frac{1}{2}n + (2\beta^*_1 - 1)n - \beta^*_1 + \gamma^*_1 + \frac{9}{8}.
\end{align*}

\medskip

\noindent
\textbf{Case (4--B--odd--(iii))}
Assume that $8\alpha^*_1 + 1 < 0$. 
For sufficiently large $n$ 
we may assume $-\frac{2(2\beta^*_1 -1)}{8\alpha^*_1+1} < \frac{2n-1}{2}$.
Since $h_B(l)$ is strictly decreasing on $[-\frac{2(2\beta^*_1 -1)}{8\alpha^*_1+1},\ 2n-1]$,\  
$h_B(l)$ is uniquely minimized at $l = 2n-1$, 
and we have 
\begin{align*}
d_-\left[
\sum_{\substack{1 \le l \le 2n-1\\ l \colon odd}} g\left(2(\frac{n-1}{2}), \frac{n-1}{2}, l; q\right)\right]
&= h_B(2n-1) -\frac{5}{4}n^2 - \frac{1}{2}n + \frac{5}{8}\\
&= (4\alpha^*_1 - \frac{3}{4})n^2 + (-4\alpha^*_1 + 2\beta^*_1 -2)n + \alpha^*_1 - \beta^*_1 + \gamma^*_1 + \frac{5}{4}.  
\end{align*} 

Summarizing, since the assumption $k \leq \frac{n}{2}$ of Case (4--B) reduces to $k=\frac{n}{2}$ when $n$ is even and is thus covered by Case (4--A),   Case (4--B--odd) where $n$ is odd covers the minimization of $d_-[g(j,k,l;q)]$ when $k \leq \frac{n}{2}$.  This is summarized below.  Compare with Claim~\ref{4A_claim} for the case that $k \geq \frac{n}{2}$.

\begin{claim}[$k < \frac{n}{2}$, $n$ odd]
\label{4B_claim}
For sufficiently large $n$,
\begin{multline*}
  d_-\left[
\sum_{\substack{k < \frac{n}{2}\\ |2k-n| \le l \le 2k+n\\ n + l \colon even}} g(n, k, l; q)\right] \\
=
\begin{cases}
-\frac{5}{4}n^2 - \frac{1}{2}n + c_+ & \mathrm{if}\ 8\alpha^*_1 + 1 > 0, \\
-\frac{5}{4}n^2 - \frac{1}{2}n + \beta^*_1 + \gamma^*_1 + \frac{1}{8} & \mathrm{if}\ 8\alpha^*_1 + 1 = 0,\ 2\beta^*_1 -1> 0  \\
-\frac{5}{4}n^2 - \frac{1}{2}n + (2\beta^*_1 - 1)n - \beta^*_1 + \gamma^*_1 + \frac{9}{8} & \mathrm{if}\ 8\alpha^*_1+ 1 = 0,\ 2\beta^*_1 -1< 0  \\
(4\alpha^*_1 - \frac{3}{4})n^2 + (-4\alpha^*_1 + 2\beta^*_1 -2)n + \alpha^*_1 - \beta^*_1 + \gamma^*_1 + \frac{5}{4} & \mathrm{if}\ 8\alpha^*_1 + 1 < 0, 
\end{cases}
\end{multline*} 
and if $8\alpha^*_i + 1 = 0,\ 2\beta^*_1 -1 = 0$, 
then we have 
\[
d_-\left[
\sum_{\substack{k < \frac{n}{2}\\ |2k-n| \le l \le 2k+n\\ n + l \colon even}} g(n, k, l; q)\right] 
\ge -\frac{5}{4}n^2 - \frac{1}{2}n + d, 
\]
where $c_+$ and $d$ are constants depending only on $K$. 
\end{claim}

\begin{remark}\label{rem:4Bquadquasi}
Similar to Remark~\ref{rem:4Aquadquasi}, parts of Cases (4--B--odd--(i),(ii)) rely upon the assumption that $d_-[J'_{K,l}(q)]=\delta'^*_K(l)$ for all $l \geq 1$.  Without the assumption, we could only obtain the equation of Claim~\ref{4B_claim} under the assumptions that $8\alpha^*_1 + 1 = 0, 2 \beta^*_1 -1<0$ for part of Cases (4--B--odd(ii)) or that $8\alpha^*_1 + 1 < 0$ for Case (4--B--odd(iii)).
\end{remark}

\begin{figure}[!ht]
\includegraphics[width=0.9\linewidth]{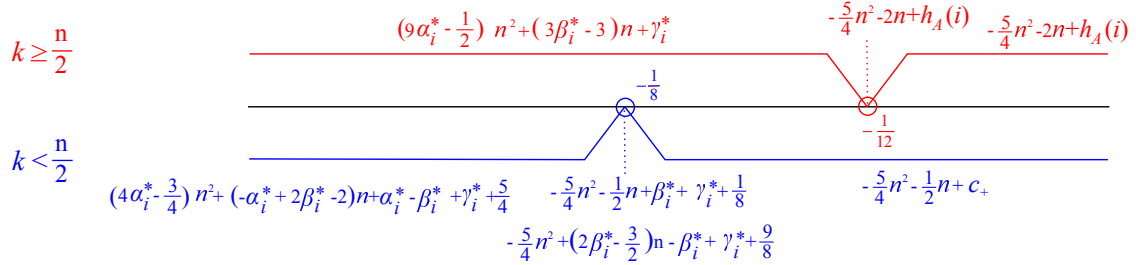}
\caption{$d_-[g(2k, k, l; q)]$ is minimized at the integer $k =\frac{n-1}{2}$. }
\label{AB}
\end{figure}

For sufficiently large $n$, 
we have that
$-\frac{5}{4}n^2 - 2n < -\frac{5}{4}n^2 - \frac{1}{2}n + c_+$. 
Furthermore, 
$12\alpha^*_i + 1 <0$ implies that both $9\alpha^*_i - \frac{1}{2} < -\frac{5}{4}$ and 
$9\alpha^*_i - \frac{1}{2} < 4\alpha^*_i - \frac{3}{4}$. 
Thus, since $8\alpha^*_1+1  \leq 0$ implies $12\alpha^*_1+1  \leq  -\frac{1}{2} < 0$, Claims~\ref{4A_claim} and \ref{4B_claim} give:

\[
d_{-}\left[\sum_{\substack{j,k=0 \\ j\le 2k}}^n \sum_{\substack{|2k-n| \le l \le 2k+n \\ n+l:even}}g(j,k,l;q)\right]
= 
\begin{cases}
-\frac{5}{4}n^2 - 2n + h_A(i) & \mathrm{if}\ 12\alpha^*_i + 1 > 0, \\
-\frac{5}{4}n^2 - 2n + h_A(i) & \mathrm{if}\ 12\alpha^*_i + 1 = 0, \\
(9\alpha^*_i - \frac{1}{2})n^2 + (3\beta^*_i -3) n + \gamma^*_i & \mathrm{if}\ 12\alpha^*_i + 1 < 0.
\end{cases}
\]

Since
\[
<n>J'_{M(K), n}(q) =\sum_{\substack{j,k=0 \\ j\le 2k}}^n \sum_{\substack{|2k-n| \le l \le 2k+n \\ n+l:even}}
g(j,k,l;q), 
\]
and $d_-[\langle n \rangle] = -\frac{1}{2}n$, 
we have 
\[
-\frac{1}{2}n + d_{-}[J'_{K, n}(q)] 
= 
d_{-}\left[\sum_{\substack{j,k=0 \\ j\le 2k}}^n \sum_{\substack{|2k-n| \le l \le 2k+n \\ n+l:even}}g(j,k,l;q)\right]. 
\]
Finally putting $C'_i = h_A(i)$ we obtain 
\begin{equation*}
  \delta'^*_{M(K)}(n)=
   \left\{ \begin{array}{ll} 
    -\frac{5}{4} n^2 - \frac{3}{2} n + C'_i&  (\alpha^*_i > -\frac{1}{12})\\
    -\frac{5}{4} n^2 - \frac{3}{2} n +C'_i &  (\alpha^*_i = -\frac{1}{12}) \\
     (9\alpha^*_i - \frac{1}{2})n^2 +(3\beta^*_i-\frac{5}{2}) n+ \gamma^*_i & (\alpha^*_i < -\frac{1}{12}), 
      \end{array} \right. 
\end{equation*}
where $C'_0 = 0$ and $C'_1 = \alpha^*_1 + \beta^*_1 + \gamma^*_1 - \frac{1}{4}$, as desired.

Moreover, due to Remarks~\ref{rem:4Aquadquasi} and \ref{rem:4Bquadquasi}, 
if we only have that $d_-[J'_{K,n}(q)]=\delta'^*_K(n)= \alpha^*_i n^2 + \beta^*_i n + \gamma^*_i$ for sufficiently large $n$ with 
$\alpha^*_i < -\frac{1}{12}$, then 
we may conclude that 
\[\delta'^*_{M(K)}(n)=(9\alpha^*_i - \frac{1}{2})n^2 +(3\beta^*_i-\frac{5}{2}) n+ \gamma^*_i. \]

\end{proof}

\bigskip

\subsection{The Slope Conjecture at the minimum degree of a Mazur double}

Theorem~\ref{strong_slope_conjecture_min_deg}(1) follows from 
Theorem~\ref{slope_conjecture_min_deg} below. 

\begin{theorem}
\label{slope_conjecture_min_deg}
Let $K$ be a knot. 
We assume that the period of $\delta^*_K(n)$ is less than or equal to  $2$ and that $b^*_i \ge 0$, 
if $a^*_i = 0$, then $b^*_i \ne 0$ and $a^*_1 + b^*_1 + c^*_1 \ge 0$. 
Furthermore, $d_{-}[J_{K, n}(q)] = \delta^*_K(n)$ for all $n \ge 2$. 

If $K$ satisfies the Slope Conjecture  for the minimum degree, 
then its Mazur double also satisfies the Slope Conjecture with the minimum degree. 
\end{theorem}

\begin{proof}
Let $K$ be a knot such that $\delta^*_K(n)=a^*(n) n^2 +b^*(n) n+c^*(n)$ is quadratic quasi-polynomial of period $\le 2$ with 
$b^*(n) \ge 0$.  
Then Proposition~\ref{mindeg_Mazur} shows that the Jones slopes of $M(K)$ satisfy 
\[js_{M(K)} \subset \{ -5,\ 36 a^*(n) -2\}.\]

Let us find essential surfaces in $E(M(K))$ whose boundary slopes are these Jones slopes.

According to Proposition~\ref{mindeg_Mazur} 
we divide into two cases depending upon $a^*_i \ge - \frac{1}{12}$ or $a^*_i < - \frac{1}{12}$. 

\medskip

\noindent
\textbf{Case 1.\ $a^*_i \ge - \frac{1}{12}$.}\quad 
In this case Proposition~\ref{mindeg_Mazur} shows that the Jones slope is $-5$. 
Referring Table~\ref{table:5/14paths}, 
we take an (orientable) essential surface $F_{\gamma_7}$, with $\alpha = 2$ and $\beta = 0$,
in $S^3 - \mathrm{int}N(k_1 \cup k_2)$ described in Section~\ref{essential_surfaces}.   Let us denote this surface as $F_0$.
Then $F_0$ has a pair of boundary slopes 
$(-5, \emptyset)$ on $\partial N(k_1),\  \partial N(k_2)$, 
i.e.\ it has boundary slope $-5$ on $\partial N(k_1)$ and 
does not intersect $\partial N(k_2)$.  

Regard $S^3 - \mathrm{int}N(k_1 \cup k_2)$ as $V - \mathrm{int}N(k_1)$.  
Then $F_{0} \cap \partial V = \emptyset$ while $F_0$ has the boundary slope $-5$ on 
$\partial N(k_1)$. 

Next, let $S_0$ be its image $f(F_0)$ in $X = f(V - \mathrm{int}N(k_1)) = E(M(K))$, the exterior of the Mazur double of $K$. 
Then $S_0$ is an orientable essential surface in $X=E(M(K))$ which has the boundary slope $-5$ on $\partial N(M(K))$, and therefore it is the desired essential surface. 

\medskip

\noindent
\textbf{Case 2.\ $a^*_i < - \frac{1}{12}$.}\quad 
By the assumption $K$ satisfies the Slope Conjecture, 
hence the Jones slope $4a^*_i$ is realized by the boundary slope of an essential surface 
$S^*_{K, i} \subset E(K)$. 

\begin{claim}
\label{F^*_i_case1}
There exists an essential surface $F^*_i$ in $V - \mathrm{int}N(k_2)$ 
such that each component of $F^*_i \cap \partial V$ has slope $4a^*_i$ and 
each component of $F_i \cap \partial N(k_2)$ has $36a^*_i -2$.  
\end{claim}

\begin{proof}
Let us take an essential surface $F_{\gamma_5}$ in 
$S^3 - \mathrm{int}N(k_1 \cup k_2)$ 
described in Section~\ref{essential_surfaces}. 
Then it has a pair of boundary slopes 
$(-3\frac{\beta}{\alpha},\  -3\frac{\alpha}{\beta}-2)$ on $\partial N(k_1),\  \partial N(k_2)$, 
i.e.\ $F_{\gamma_5}$ has boundary slopes $-3\frac{\beta}{\alpha}$ on $\partial N(k_1)$ and 
$-3\frac{\alpha}{\beta}-2$ on $\partial N(k_2)$.  
 
Regard $S^3 - \mathrm{int}N(k_1 \cup k_2)$ as $V - \mathrm{int}N(k_2)$. 
Using the preferred meridian-longitude $(\mu_V, \lambda_V)$ instead of $(\mu_1, \lambda_1)$, 
$F_{\gamma_5} \cap \partial V$ has slope $-\frac{\alpha}{3\beta}$. 
Choose $\alpha, \beta$ ($\alpha \ge \beta$) so that 
$-\frac{\alpha}{3\beta} = 4a^*_i$, 
i.e. $\frac{\alpha}{\beta} = -12a^*_i > 1$. 
Let us denote $F_{\gamma_5}$ associated with $a^*_i$ by $F^*_i$.  
Then each component of $F^*_i \cap \partial V$ has slope $4a^*_i$,  
and each component of $F^*_i \cap \partial N(k_2)$ has slope $36a^*_i-2$ as desired. 
\end{proof}

\medskip

Next, let $S^*_i$ be the image $f(F^*_i)$ in $X = f(V - \mathrm{int}N(k_2))$.
As in the proof of Theorem~\ref{thm:SC_and_SSC_M}(1), 
we obtain the desired essential surface 
$\tilde{S}^*_i = m S^*_i \cup n S^*_{K,i}$ or the $\partial I$--subbundle of 
a twisted $I$--bundle of $\tilde{S}_i$ (when $\tilde{S}_i $ is non-orientable). 
\end{proof}

\subsection{The Strong Slope Conjecture at the minimum degree of a Mazur double}

Let us now turn to the Strong Slope Conjecture with the minimum degree. 
Recall that for a quadratic quasi-polynomial $\delta^*_K(n) = a^*(n)n^2 + b^*(n) n + c^*(n)$ with period $\le 2$, 
we put $a^*_i = a^*(i),\ b^*_i = b^*(i)$ and $c^*_i= c^*(i)$ for $i = 1, 2$.

\begin{theorem}
\label{strong_slope_conjecture_min_deg_SS*}
Let $K$ be a knot. 
We assume that the period of $\delta^*_K(n)$ is less than or equal to  $2$ and that $b^*_i \ge 0$, 
if $a^*_i = 0$, then $b^*_i \ne 0$ and $a^*_1 + b^*_1 + c^*_1 \ge 0$. 
Furthermore, $d_{-}[J_{K, n}(q)] = \delta^*_K(n)$ for all $n \ge 2$. 

If $K$ satisfies the Strong Slope Conjecture with $SS^*(i)$, 
then its Mazur double also satisfies the Strong Slope Conjecture with $SS^*(i)$. \end{theorem}

\begin{proof}
Write $\delta^*_{M(K)}(n)=a^*_M(n)n^2 + b^*_M(n)n + c^*_M(n)$.   
From Proposition~\ref{mindeg_Mazur}, 
we may also write 
$a^*_{M,i} = a^*_M(i), b^*_{M,i} = b^*_M(i),  c^*_{M,i} = c^*_M(i)$ for $i = 1, 2$. 
\medskip

\noindent
\textbf{Case 1.\ $a^*_i \ge -\frac{1}{12}$.}\quad 
Following the proof of Theorem~\ref{slope_conjecture_min_deg}, 
we take $S^*_i$ as the image of $F_0$ which is $F_{\gamma_7}$ (with $\alpha = 2$ and $\beta = 0$) 
in $S^3 - \mathrm{int}N(k_1 \cup k_2)$ described in Section~\ref{essential_surfaces}. 
Recall that $S^*_i$ does not intersect $f(\partial V) = T_K$. 
Precisely, $S^*_i$ is a twice punctured torus whose boundary slope on $\partial E(M(K))$ is $-5$. 
So $\chi(S_i) = -2$ and thus we have
\[
\frac{\chi(S^*_i)}{|\partial S^*_i|\cdot 1} = \frac{-2}{2} = -2 \frac{1}{2} = -2b^*_{M, i}.
\]
Hence $S^*_i$ is the desired Jones surface. 

\medskip

\noindent
\textbf{Case 2.\ $a^*_i < -\frac{1}{12}$.}\quad 
Write $a^*_i = r_i / s_i$ for some coprime integers $r_i, s_i$, where $s_i > 0$. 
Then, as a ratio of coprime integers, 
the denominator of $4a^*_i$ is $s_i/\mathrm{gcd}(4, s_i)$. 
Since $K$ satisfies the Strong Slope Conjecture with $SS^*(i)$, 
there is a properly embedded essential surface $S^*_{K, i}$ 
in the exterior of $K$ whose boundary slope is $4a^*_i$ and 
\[\frac{\chi(S^*_{K,i})}{|\partial S^*_{K,i}| \cdot \frac{s_i}{\gcd(4,s_i)}} = -2 b^*_i. \]

We show that an essential surface $\tilde{S^*}_i$ in $E(M(K))$ given in the proof of Theorem~\ref{thm:SC_and_SSC_M}(1) 
satisfies the condition of the Strong Slope Conjecture for minimum degree: 
\[\widetilde{S^*}_i\ \textrm{has boundary slope}\ p/q = 4a^*_{M,i} \quad 
\mbox{and} \quad \frac{\chi(\widetilde{S^*}_i)}{|\partial \widetilde{S^*}_i| q} = -2b^*_{M,i}.\]

When addressing the Slope Conjecture with the minimum degree for $M(K)$ in this case, 
we constructed a properly embedded essential surface 
$\widetilde{S^*}_i = m S^*_{K,i} \cup n S^*_i$ 
in the exterior of $M(K)$ 
by joining $m$ copies of $S^*_{K,i}$ in $E(K)$ to $n$ copies of the surface $S^*_i$  in 
$f(V) - N(M(K))$.  
This requires that 
\[ m |\bdry S^*_{K,i}| = n |\bdry S^*_i \cap T_K|. \]

Regard $S^3 - \mathrm{int}N(k_1 \cup k_2)$ as $V - \mathrm{int}N(k_2)$. 
Recall that the surface $S^*_i$ is identified with a surface of type $F_{\gamma_5}$ in the exterior of the $[2, 1, 4]$ two-bridge link by the embedding $f \colon V \to S^3$. 
$F_{\gamma_5} \cap \partial V$ has slope $-\frac{\alpha}{3\beta}$ with respect to the preferred meridian-longitude $(\mu_V, \lambda_V)$. 
Choose $\alpha, \beta$ ($\alpha \ge \beta$) as 
$\frac{\alpha}{\beta} = -12a^*_i = -\frac{12r_i}{s_i} > 1$ 
so that $S^*_i$ has boundary slope $4a^*_i$ on $\partial V = T_K$.  
Note that $S^*_i$ has boundary slope 
$-3\frac{\alpha}{\beta} -2 = 36a^*_i-2$ on $\partial N(M(K)) = T_M$. 
For convenience, we choose $\beta = 2s_i$, $\alpha = -24r_i$ 
so that $F_{\gamma_5} = F_{\gamma_5, \alpha, \beta}$ is orientable; see Subsection~\ref{algorithm}. 
 
Then, using Table~\ref{table:5/14paths},
we calculate the following:
\begin{itemize}
\item 
$\chi(S^*_i) 
= -2\alpha + \beta = 48r_i + 2s_i$,

\item  
slope of $\partial S^*_i$  on $T_M$ is 
$-3\frac{\alpha}{\beta} -2
= -3(-\frac{12r_i}{s_i}) -2
= 36 \frac {r_i}{s_i} -2
= \frac{36r_i - 2s_i}{s_i}$,

\item $|\partial S^*_i \cap T_K| 
= \gcd(3\beta, \alpha) 
= \gcd(6s_i,\ -24 r_i) 
= 6 \gcd(4,\ s_i)$, and

\item $|\partial S^*_i \cap T_M| 
= \gcd(3\alpha, \beta) 
= \gcd(-72 r_i,\ 2s_i) 
= 2\gcd(36,\ s_i)$.
\end{itemize}

The boundary of $\tilde{S^*}_i$ consists of $n$ copies of the boundary of $S^*_i$  on $T_M$, 
so we have

\begin{itemize}
\item 
$|\partial \tilde{S^*}_i| = n | \partial S^*_i \cap T_M | = 2n \gcd(36,\ s_i)$. 
\end{itemize}

Moreover, the boundary slope of $\tilde{S^*}_i$ is the slope of $\partial S^*_i$ on $T_M$, 
and so this has denominator $\frac{s_i}{\gcd(36r_i - 2s_i, s_i)} = \frac{s_i}{\gcd(36,s_i)}$.  
We may now calculate
\begin{align*}
  \frac{\chi(\widetilde{S^*}_i)}{|\partial \widetilde{S^*}_i| \cdot \frac{s_i}{\gcd(36, s_i)}} 
  &= \frac{m \chi(S^*_{K, i}) + n \chi(S^*_i)}{ 2n \gcd(36,s_i) \cdot \frac{s_i}{\gcd(36,s_i)}} \\
  &= \frac{-2 b^*_i m |\partial S^*_{K, i}| \cdot \frac{s_i}{\gcd(4,s_i)} + n(48r_i + 2s_i)}{2n s_i} \\
  &= \frac{-12 b^*_i n \gcd(4, s_i) \cdot \frac{s_i}{\gcd(4, s_i)} + n(48r_i + 2s_i)}{2n s_i} \\
  &= \frac{-12 b^*_i n s_i + n(48r_i + 2s_i)}{2n s_i}\\
  &= -6 b^*_i + 24r_i/s_i +1\\
  &= -2(-12a^*_i + 3b^*_i -\frac{1}{2}) = -2b^*_{M,i} 
\end{align*}
    as desired.

If the glued surface $\widetilde{S^*}_i = m S^*_{K,i} \cup n S^*_i$ is non-orientable, 
then as in the proof of Theorem~\ref{thm:SC_and_SSC_M}(1), 
we replace $\widetilde{S^*}_i$ by the frontier $\widetilde{S^*}_i'$ of the tubular neighborhood of $\widetilde{S^*}_i$, 
but \cite[Lemma 5.1]{BMT_tgW} shows that $\widetilde{S^*}_i'$ and $\widetilde{S^*}_i$ has the same boundary slope and 
$\displaystyle \frac{\chi(\widetilde{S^*}_i')}{|\partial\widetilde{S^*}_i'| \cdot \frac{s_i}{\gcd(36,s_i)}} 
= \frac{\chi(\widetilde{S^*}_i)}{|\partial \widetilde{S^*}_i| \cdot \frac{s}{\gcd(36,s_i)}}$. 
Thus the essential surface $\widetilde{S^*}_i$ or $\widetilde{S^*}_i'$ (when $\widetilde{S^*}_i$ is non-orientable) is the desired 
essential surface.
\end{proof}

\bigskip

\section{Crossing Number for Mazur doubles of knots}
\label{sec:crossing}

The \textit{crossing number} of $K$ is defined to be the minimal number of crossings in any of its diagrams, 
and denoted by $c(K)$. 
The crossing number is one of the most basic invariants of a knot. 
However, crossing number is notoriously intractable. 
For instance, for a satellite knot $K$ with a companion knot $k$, 
although it is conjectured that $c(K) \ge c(k)$ \cite{Kir}, it remains widely open. 
Important progress was made by Lackenby \cite{Lac_sate}, 
who proves that $\displaystyle c(K) \ge \frac{1}{10^{13}}c(k)$. 

Very recently Kalfagianni and Ruey Shan Lee \cite{K-RSL}  explicitly determine 
the crossing number of untwisted Whitehead doubles of a nontrivial adequate knots with trivial writhe. 
Precisely, they show that $c(W_{\pm 1}^0(K)) = 4c(K) + 1$ for any nontrivial adequate knot with $wr(K) = 0$.  
This is the first instance of results that determine the crossing numbers for broad families of prime satellite knots.

First we review notions of adequacy for knots.
Let $D$ be a diagram of a knot in $S^3$. 
A \textit{state} for $D$ is a choice of replacing every crossing of $D$ by \textit{$A$--resolution} or 
\textit{$B$--resolution} as in Figure~\ref{fig:resolution} with the dotted segment recording the location of 
the crossing before the replacement. 
The state $\sigma_{+}$ ( resp. $\sigma_{-}$ ) denotes the choice of $A$ (resp. $B$) 
-resolution at each crossing of $D$. 
Applying a state to $D$, we obtain a set of disjoint circles called \textit{state circles}. 
We form \textit{$\sigma$-state graph} $G_{\sigma}(D)$ for a state $\sigma$ 
by letting the resulting circles be vertices and the dotted segments be edges. 

\quad

\begin{figure}[H]
\centering 
\includegraphics[width=0.5\linewidth]{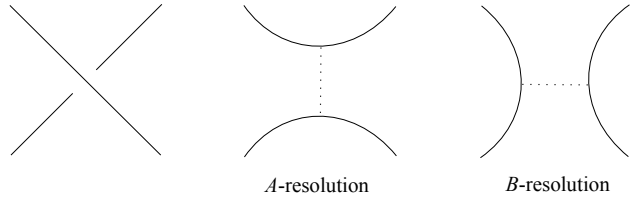}
\caption{$A$--resolution and $B$--resolution}
\label{fig:resolution}
\end{figure}

\quad

\begin{definition} 
A diagram $D$ is \textit{$A$--adequate} (resp. \textit{$B$--adequate}) 
if the graph $G_{\sigma_+}(D)$ ( resp. $G_{\sigma_-}(D)$ ) has no one-edged loop. 
If $D$ is both $A$--adequate and  $B$--adequate,  then $D$ is called \textit{adequate}.  
A knot is \textit{$A$--adequate} (resp. \textit{$B$--adequate}) 
if it has a $A$--adequate diagram (resp. $B$--adequate diagram).
A knot  is \textit{adequate} if it has an adequate diagram. 
\end{definition}

In this section we apply the argument in the proof of \cite{K-RSL} together with 
Propositions~\ref{maxdeg_Mazur} and \ref{mindeg_Mazur} to prove the following. 

\begin{thm_mazur_crossing}
Let $K$ be an adequate knot with crossing number $c(K)$ and the writhe $wr(K)$. 
Then we have the following. 
\begin{enumerate}
\item 
The Mazur double of $K$ is non-adequate. 
\item 
Assume that $c_+(K),\ c_-(K) \ne 0$. 
Then the crossing number of Mazur double of $K$ satisfies:
\[
9c(K) + 2 \le  c(M(K)) \le 9c(K)+3 + 6|wr(K)|.
\]
\item
If $wr(K) = 0$, 
then $c(M(K))$ is either $9c(K) + 2$ or $9c(K) + 3$. 
\end{enumerate}
\end{thm_mazur_crossing}

\subsection{Jones diameter of Mazur doubles}

Define $d[J_{K,n}(q)] = 4d_+[J_{K,n}(q)] - 4d_-[J_{K,n}(q)]$.
Then for suitably large $n$, $d[J_{K,n}(q)]  =4\delta_K(n) - 4\delta^*_K(n)$
Setting $\delta_K(n) = a(n) n^2 +b(n) n +c(n)$ and $\delta^*_K(n) = a^*(n) n^2 +b^*(n) n +c^*(n)$,
the {\em Jones diameter} of $K$ is defined to be $dj_K = \max \{ 4|a(n)-a^*(n)| \}$.

By the assumption of Theorem~\ref{crossing_Mazur} 
we may assume that 
\begin{itemize}
\item $a_i>0$, $a^*_i<-\frac{1}{12}$,
\item  $b_i \le 0$, $b^*_i \ge 0$,  and
\item $a^*_1+b^*_1 + c^*_1 \ge 0$
\end{itemize}

Then following Propositions~\ref{maxdeg_Mazur} and \ref{mindeg_Mazur}, 
\begin{align*}
  \delta_{M(K)}(n)&=            9a_i n^2 +(- 12a_i+3b_i -1) n+4a_i-2b_i+c_i+1 \\
 \delta^*_{M(K)}(n)&=    (9a^*_i -\frac{1}{2})n^2 +(- 12a^*_i+3b^*_i -\frac{1}{2}) n+4a^*_i-2b^*_i+c^*_i +\frac{1}{2}.
 \end{align*}
Our assumptions have that $a_i > 0 > a^*_i$, so $a_i-a^*_i >0$. 
Hence the sets of Jones slopes for maximum degree and minimum degree are given by 
\[
js_{M(K)} = \{ 36a_i \}\quad \mathrm{and}\quad js^*_{M(K)} = \{ 36a_i^*-2\}.
\]
Then we have the Jones diameter of $M(K)$ as follows. 
\[dj_{M(K)} = \max\{ | 36a_i - 36a^*_i + 2| \} = 9 \max\{ 4a_i - 4a^*_i \} + 2 = 9 dj_K + 2. \]

\subsection{Proof of Theorem~\ref{crossing_Mazur}}
It follows from \cite[Theorem 1.1]{K-RSL}  that 
\[
9 dj_K + 2 = dj_{M(K)} \le 2 c(M(K)).
\]

If $K$ is adequate, then $dj_K = 2c(K)$ and we obtain that 
\[
c(M(K)) \geq 9c(K) + 1.
\]

On the other hand, 
consider a reduced adequate diagram $D$ of $K$ which has crossing number 
$c(D)=c(K)$ with writhe $wr(D)$. 
Note that $wr(D)$ is minimal over all diagrams representing $K$. 
Thus it is invariant of $K$, and we may denote it by $wr(K)$. 

Adjust $D$ by $|wr(D)|$ Reidemeister I moves to obtain a diagram $D_0$ with trivial writhe, i.e.\ $wr(D_0)=0$.
Then the induced diagram $M(D_0)$ of $M(K)$ with blackboard framing has crossing number 
\[c(M(D_0)) = 9c(D_0)+3 = 9c(D)+3 + 9|wr(K)| = 9c(K)+3 + 9|wr(K)|.\]  
Each Reidemiester I curl of $D_0$ gives $9$ crossings of $M(D_0)$, 
but for each curl we may apply ``belt trick'' to reduce the $9$ crossings to $6$; see Figure~\ref{belt_trick}. 
\begin{figure}[!ht]
\includegraphics[width=0.35\linewidth]{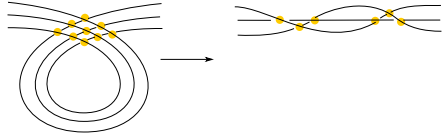}
\caption{Belt trick reduces the crossing number $9$ to $6$ for each curl of $D$. }
\label{belt_trick}
\end{figure}

Let us call the resulting diagram $M(D)$ so that 
\[c(M(D)) = 9c(K)+3 + 6|wr(D)|.\]
Hence, for an adequate knot $K$, 
we obtain the bound from the below and above
\[ 9c(K) + 1 \le  c(M(K)) \le 9c(K)+3 + 6|wr(K)|.\]

Let us exclude the possibility of 
$c(M(K)) = 9c(K) + 1$. 
We calculated that $dj_{M(K)} = 9 dj_K + 2$ so 
\[ dj_{M(K)} = 9 \cdot 2 c(K) + 2 = 2(9c(K)+1). \]
Since Lemma~\ref{MazurDoubleNotAadequate} below shows that $M(K)$ is not adequate, 
\cite[Theorem 1.1]{K-RSL} gives the strict inequality 
\[ c(M(K)) > \frac{dj_{M(K)}}{2}  = 9c(K)+1. \]
Hence we have
\[ 9c(K) + 2 \le  c(M(K)) \le 9c(K)+3 + 6|wr(K)|.\]
This establishes $(2)$ of Theorem~\ref{crossing_Mazur}. 
 
In particular if $wr(K) = 0$, 
then we have 
\[
c(M(K)) = 9c(K) + 2,\ \mathrm{or}\quad 9c(K)+3, 
\]
proving $(3)$ of Theorem~\ref{crossing_Mazur}.

So it remains to show that if $K$ is adequate, 
then its Mazur double is non-adequate ($(1)$ of Theorem~\ref{crossing_Mazur}).  
Actually, more strongly we prove:

\begin{lemma}
\label{MazurDoubleNotAadequate}
Let $K$ be a non-trivial A-adequate knot with crossing number $c(K)$ and writhe $wr(K)$. 
Then $M(K)$ is not A-adequate. 
In particular, 
if $K$ is adequate, then $M(K)$ is non-adequate. 
\end{lemma}

\begin{proof}
Let $D$ be a reduced A-adequate diagram of an A-adequate knot $K$.
Then \cite[Lemma~3.6]{KT} (see also \cite[Theorem~3.2]{Ka}) shows that 
\begin{align*}
d_-[J_K(n)] &= a^*n^2 + b^*n + c^* \\
  &= -\frac{c_-(D)}{2}n^2 + \left(\frac{c(D)-v_A(D)}{2}\right)n + \frac{v_A(D)-c_+(D)}{2}.
\end{align*}
It follows that $a^* + b^* + c^* = 0$, $2a^*$ is an integer, and also that $b^* \ge 0$.  
This last inequlaity follows from $-2b^* = v_A(D)-c(D)$ being the Euler characteristic of the associated A-state surface. 
Because $K$ is a non-trivial knot, this A-state surface cannot be a disk.
Therefore we may apply Theorem~\ref{mindeg_Mazur} to obtain $\delta^*_{M(K)}(n)$.

Now suppose $M(K)$ is also A-adequate.  
Then we would have $d_-[J_{M(K)}(n)] = \delta^*_{M(K)}(n)$ for all $n \geq 1$ (by \cite[Lemma~3.6]{KT}) and the sum of its coefficents would equal $0$.
When $c_-(D)>0$ so that $a^* \leq -\frac12$, 
then Theorem~\ref{mindeg_Mazur} gives
\[ d_-[J_{M(K)}(n)] = \left(9a^* -\frac{1}{2}\right)n^2 + \left(-12a^* + 3b^* - \frac{1}{2}\right)n + \left(4a^* -2b^* +c^*+\frac{1}{2} \right)\]
from which 
 one readily observes that the sum of the coefficients is
\[
\left(9a^* -\frac{1}{2}\right)+ \left(-12a^* + 3b^* - \frac{1}{2}\right) + \left(4a^* -2b^* +c^*+\frac{1}{2} \right) = a^* + b^* + c^* + \frac12 = \frac12.
\]
Of course this is not $0$, contradicting the assumption that $M(K)$ was A-adequate.  

If $c_-(D)=0$ so that $a^* = 0$, then Theorem~\ref{mindeg_Mazur} gives
\[ d_-[J_{M(K)}(n)] = -\frac54 n^2 + \frac12 n - \frac12 + C\]
where $C$ is either $\frac34$ or $4a^* + 2b^* + c^* + \frac12 = b^* + \frac12$.
In each case the sum of the coefficients is either $-\frac12$ or $-\frac34 + b^* = -\frac34 + \frac{c(D)-v_A}{2}$, respectively.  
Neither is an integer, so this sum cannot be $0$.  
Again, this contradicts the assumption that $M(K)$ was A-adequate.
\end{proof}

This completes a proof of Theorem~\ref{crossing_Mazur}. 
\hspace*{\fill} $\square$

\medskip

Lemma~\ref{MazurDoubleNotAadequate} shows that 
if $K$ is a non-trivial A-adequate knot, then its Mazur double is not A-adequate. 
However, for B-adequacy we observe the following. 

\begin{addendum}
\label{B_adequate_Mazur_B_adequate}
Assume that $K$ is B-adequate with B-adequate diagram $D$. 
If $wr(D) = 0$, 
then 
the diagram of $M(K)$ obtained from $D$ depicted in Figure~\ref{B_adequate_Mazur_double} is also $B$-adequate. 
\end{addendum}

\begin{proof}
Denote by $D_-(K)$ a diagram obtained from $D(K)$ by applying $B$--resolution at each crossing.
Then we obtain diagram Figure~\ref{B_adequate_Mazur_double} (Right), 
where $\mathcal{D}^3(K)$ denotes a diagram obtained from $\mathcal{D}(K)$ by replacing each of its component by $3$ parallels and $\mathcal{D}^3_-(K)$ denotes a diagram obtained from $\mathcal{D}^3(K)$ by applying $B$--resolution at each crossing. 

Since $\mathcal{D}(K)$ is $B$--adequate, 
\cite[Lemma 2.17]{Le-lec} implies that $\mathcal{D}^3(K)$ is also $B$--adequate, 
and hence the $\sigma_{-}$--state graph associated to $\mathcal{D}_{-}^3(K)$ has no loop edge. 
(In \cite{Le-lec} Le uses the terminology \textit{plus}-adequate (resp.\ \textit{minus}-adequate) to mean $A$--adequate 
(resp. $B$--adequate).)
Then it is easy to see that the graph associated to $\mathcal{D}^3_-(K) \cup C$ has no loop edge, 
and hence $\mathcal{D}(M(K))$ is $B$--adequate as well. 
Thus $M(K)$ is $B$--adequate. 
\end{proof}
\begin{figure}[!ht]
\includegraphics[width=0.7\linewidth]{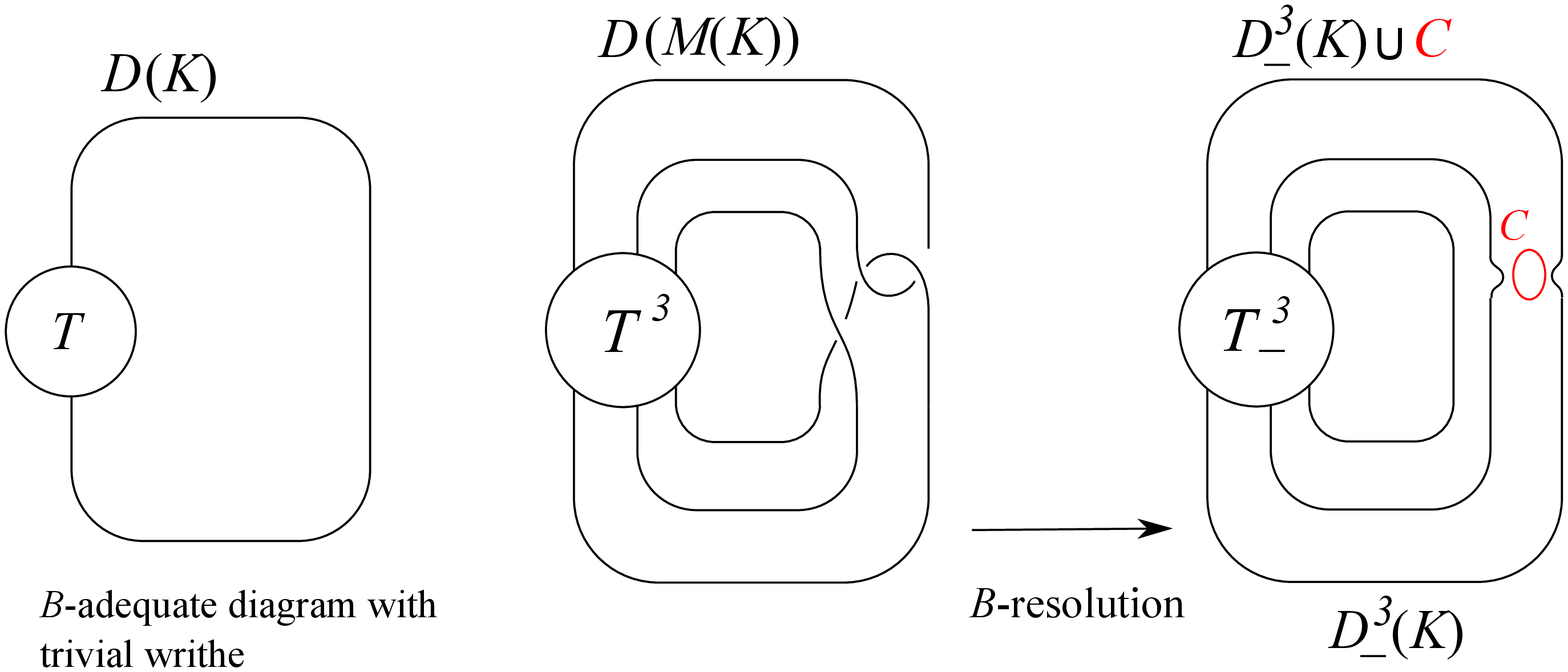}
\caption{A Mazur double of $B$-adequate knot $K$ is $B$-adequate.}
\label{B_adequate_Mazur_double}
\end{figure}

\end{document}